\documentclass[a4paper,12pt]{article}

\usepackage[utf8]{inputenc}
\usepackage[T1]{fontenc}
\usepackage[english]{babel}

\usepackage{mathrsfs}

\usepackage{amsmath,amsfonts,amssymb,amsthm}
\usepackage[all]{xy}
\usepackage{graphicx}

\usepackage{thmbox}

\newtheorem{defn}{Definition}
\newtheorem{prop}{Proposition}
\newtheorem{lem}{Lemma}
\newtheorem{thm}{Theorem}

\newcommand{\Z}{\mathbb{Z}}
\renewcommand{\S}{\mathbb{S}}
\newcommand{\R}{\mathbb{R}}
\newcommand{\C}{\mathbb{C}}
\newcommand{\CC}{\widehat{\mathbb{C}}}

\DeclareMathOperator{\Sym}{\mathrm{Sym}}
\DeclareMathOperator{\Aut}{\mathrm{Aut}}
\DeclareMathOperator{\IMG}{\mathrm{IMG}}
\DeclareMathOperator{\Ker}{\mathrm{Ker}}
\DeclareMathOperator{\Id}{\mathrm{Id}}
\def\recursion[#1]{{\langle\!\langle}#1{\rangle\!\rangle}}
\def\bigrecursion[#1]{{\Big\langle\!\Big\langle}#1{\Big\rangle\!\Big\rangle}}
\DeclareMathOperator{\Mating}{\text{\raisebox{-2pt}{\rotatebox{90}{$\models$}}}}

\title{Introduction to Iterated Monodromy Groups}
\author{Sébastien Godillon}
\date{March 24, 2012}

\begin{document}


\maketitle

\begin{abstract}
The theory of iterated monodromy groups was developed by Nekrashevych \cite{SelfSimilarGroups}. It is a wonderful example of application of group theory in dynamical systems and, in particular, in holomorphic dynamics. Iterated monodromy groups encode in a computationally efficient way combinatorial information about any dynamical system induced by a post-critically finite branched covering. Their power was illustrated by a solution of the Hubbard Twisted Rabbit Problem given by Bartholdi and Nekrashevych \cite{BartholdiNekrashevych}.

These notes attempt to introduce this theory for those who are familiar with holomorphic dynamics but not with group theory. The aims are to give all explanations needed to understand the main definition (Definition \ref{DefIteratedMonodromyGroup}) and to provide skills in computing any iterated monodromy group efficiently (see examples in Section \ref{SecExamples}). Moreover some explicit links between iterated monodromy groups and holomorphic dynamics are detailed. In particular, Section \ref{SecCombinatorialInvariance} provides some facts about combinatorial equivalence classes, and Section \ref{SecMatings} deals with matings of polynomials.
\end{abstract}

\tableofcontents

\newpage

\begin{quote}
Representations of groups, automata, bimodules and virtual endomorphisms are intentionally omitted in order to make this introduction more elementary. All the proofs are mainly based on the path and homotopy lifting properties (Proposition \ref{PropLift}) from algebraic topology. For further reading see \cite{PilgrimBook}, \cite{SelfSimilarGroups} and \cite{SelfSimilarity}.

These notes come from lectures given in the Chinese Academy of Sciences in Beijing in September 2011. The way to introduce iterated monodromy groups by using two trees (one left to label the vertices and one right for the edges, see Section \ref{SecTreePreimages}) was explained to the author by Tan Lei although it implicitly appears in \cite{SelfSimilarGroups} and some others works. The author would like very much to thank Laurent Bartholdi for his fruitful discussion and patience in explaining his works, the referee for his helpful comments and relevant suggestions and Tan Lei for her support and encouragement.
\end{quote}

\newpage

\section{Preliminaries}\label{SecPreliminaries}

\subsection{Tree automorphism}\label{SecTreeAutomorphism}

A tree $T$ is a (simple undirected) graph which is connected and has no cycles. More precisely, a tree $T=(V,E)$ is the data of a set of vertices $V$ and a set of edges $E$ which are pairs of two distinct vertices, such that for any two distinct vertices $v,v'$ there is a unique path of edges from $v$ to $v'$. For every edge $\{v,v'\}$, the vertices $v$ and $v'$ are said to be adjacent which is denoted $v\stackrel{T}{\sim}v'$ (being adjacent is a symmetric binary relation).

A tree $T$ is said rooted if one vertex $t\in V$ has been designated the root. In this case, one can write the set of all vertices as the following partition $V=\bigsqcup_{n\geqslant 0} V^{n}$ where $V^{n}$ is the set of all the vertices linked to the root $t$ by a path of exactly $n$ edges (and $V^{0}=\{t\}$). Each $V^{n}$ is called the set of all vertices of level $n$.

\begin{defn}
Two rooted trees $T=(V,E)$ and $\widehat{T}=(\widehat{V},\widehat{E})$ are said to be isomorphic if there is a bijection $\varphi$ from $V=\bigsqcup_{n\geqslant 0} V^{n}$ onto $\widehat{V}=\bigsqcup_{n\geqslant 0} \widehat{V}^{n}$ satisfying the following two axioms
\begin{description}
\item[Level preserving:] $\forall n\geqslant 0,\ \varphi(V^{n})=\widehat{V}^{n}$
\item[Edge preserving:] $\forall v,v'\in V,\ v\stackrel{T}{\sim}v'\Rightarrow\varphi(v)\stackrel{\widehat{T}}{\sim}\varphi(v')$
\end{description}
Such a bijection $\varphi$ is called a tree isomorphism. A tree automorphism is a tree isomorphism from a rooted tree $T$ onto itself.

The set of all tree automorphisms of $T$ is denoted by $\Aut(T)$ and it is equipped with the group structure coming from composition of maps. For every pair of tree automorphisms $g,h$ in $\Aut(T)$, their composition is denoted by $g.h$ where the map $g$ is performed first (this notation is more convenient for computations in $\Aut(T)$ than $h\circ g$).
\end{defn}

Given the alphabet $\mathcal{E}=\{0,1,\dots,d-1\}$ of $d\geqslant2$ letters, consider the following sets of words
\begin{itemize}
\item $\mathcal{E}^{0}=\{\emptyset\}$
\item $\forall n\geqslant 1,\ \mathcal{E}^{n}=\{\text{words of length }n\text{ with letters in }\mathcal{E}\}=\{\varepsilon_{1}\varepsilon_{2}\dots\varepsilon_{n}\ /\ \forall k,\ \varepsilon_{k}\in \mathcal{E}\}$
\item $\mathcal{E}^{\star}=\bigsqcup_{n\geqslant 0} \mathcal{E}^{n}$
\end{itemize}

\begin{defn}
The regular rooted tree $T_{d}$ is defined as follows
\begin{description}
\vspace{-5pt}
\item[Root:] the empty word $\emptyset$
\vspace{-5pt}
\item[Vertices:] the set of words $\mathcal{E}^{\star}=\bigsqcup_{n\geqslant 0} \mathcal{E}^{n}$
\vspace{-5pt}
\item[Edges:] all the pairs $\{w,w\varepsilon\}$ where $w$ is a word in $\mathcal{E}^{\star}$ and $\varepsilon$ is a letter in $\mathcal{E}$
\vspace{-5pt}
\end{description}
\end{defn}
The graph below shows the first three levels of the regular rooted tree $T_{2}$.
$$\xymatrix @!0 @R=0.2pc @C=6pc {
&&&& \\
&&& 000\ar@{-}[ru]\ar@{-}[rd] & \\
&&&& \\
&& 00\ar@{-}[ruu]\ar@{-}[rdd] && \\
&&&& \\
&&& 001\ar@{-}[ru]\ar@{-}[rd] & \\
&&&& \\
& 0 \ar@{-}[ruuuu]\ar@{-}[rdddd] &&& \\
&&&& \\
&&& 010\ar@{-}[ru]\ar@{-}[rd] & \\
&&&& \\
&& 01\ar@{-}[ruu]\ar@{-}[rdd] && \\
&&&& \\
&&& 011\ar@{-}[ru]\ar@{-}[rd] & \\
&&&& \\
\emptyset \ar@{-}[ruuuuuuuu]\ar@{-}[rdddddddd] &&&& \hspace{30pt}\dots \\
&&&& \\
&&& 100\ar@{-}[ru]\ar@{-}[rd] & \\
&&&& \\
&& 10\ar@{-}[ruu]\ar@{-}[rdd] && \\
&&&& \\
&&& 101\ar@{-}[ru]\ar@{-}[rd] & \\
&&&& \\
& 1 \ar@{-}[ruuuu]\ar@{-}[rdddd] &&& \\
&&&& \\
&&& 110\ar@{-}[ru]\ar@{-}[rd] & \\
&&&& \\
&& 11\ar@{-}[ruu]\ar@{-}[rdd] && \\
&&&& \\
&&& 111\ar@{-}[ru]\ar@{-}[rd] & \\
&&&&
}$$

The regular rooted tree $T_{d}$ is an example of self-similar object. Namely for every word $v\in \mathcal{E}^{\star}$, the map $\varphi_{v}:\mathcal{E}^{\star}\rightarrow v\mathcal{E}^{\star},\ w\mapsto vw$ is a tree isomorphism from $T_{d}$ onto the regular subtree $T_{d}|_{v}$ rooted at $v$.

For any tree automorphism $g\in\Aut(T_{d})$ and any word $v\in\mathcal{E}^{\star}$, notice that the following map
$$\mathcal{R}_{v}(g)=(\varphi_{g(v)})^{-1}\circ g|_{v\mathcal{E}^{\star}}\circ\varphi_{v}$$
is well defined from $\mathcal{E}^{\star}$ onto itself since the restriction $g|_{v\mathcal{E}^{\star}}:v\mathcal{E}^{\star}\rightarrow g(v)\mathcal{E}^{\star}$ is a tree isomorphism from the regular subtree $T_{d}|_{v}$ rooted at $v$ onto the regular subtree $T_{d}|_{g(v)}$ rooted at $g(v)$. Actually $\mathcal{R}_{v}(g)$ defines a tree automorphism of $T_{d}$ as it is shown in the commutative diagram below.
$$\xymatrix{ T_{d} \ar[rr]^{\textstyle \mathcal{R}_{v}(g)} \ar[d]_{\textstyle \varphi_{v}} && T_{d} \ar[d]^{\textstyle \varphi_{g(v)}} \\ T_{d}|_{v} \ar[rr]_{\textstyle g|_{v\mathcal{E}^{\star}}} && T_{d}|_{g(v)} }$$

\begin{defn}\label{DefSelfSimilar}
For any tree automorphism $g\in\Aut(T_{d})$ and any word $v\in \mathcal{E}^{\star}$, the following tree automorphism
$$\mathcal{R}_{v}(g)=(\varphi_{g(v)})^{-1}\circ g|_{v\mathcal{E}^{\star}}\circ\varphi_{v}$$
is called the renormalization of $g$ at $v$.

A subgroup $G$ of $\Aut(T_{d})$ is said to be self-similar if the following condition holds
$$\forall g\in G,\ \forall v\in \mathcal{E}^{\star},\ \mathcal{R}_{v}(g)\in G$$
\end{defn}
Since $\varphi_{v}\circ\varphi_{v'}=\varphi_{vv'}$ for every pair of words $v,v'$ in $\mathcal{E}^{\star}$, it follows that $\mathcal{R}_{v'}\left(\mathcal{R}_{v}(g)\right)=\mathcal{R}_{vv'}(g)$ for every tree automorphism $g\in\Aut(T_{d})$. Therefore a quick induction shows that one only needs to check the condition above for words $v\in\mathcal{E}^{\star}$ of length $1$. That is
$$G\text{ is self-similar}\Longleftrightarrow\forall g\in G,\ \forall \varepsilon\in \mathcal{E},\ \mathcal{R}_{\varepsilon}(g)\in G$$

\begin{example}\quad\\
One can remark that any tree automorphism $g\in\Aut(T_{d})$ induces a permutation $g|_{\mathcal{E}^{1}}$ of $\mathcal{E}^{1}=\mathcal{E}$. Although it is not a one-to-one correspondence, one can conversely define a tree automorphism $g_{\sigma}\in\Aut(T_{d})$ from a given permutation $\sigma\in\Sym(\mathcal{E})$  by $g_{\sigma}:\mathcal{E}^{\star}\rightarrow \mathcal{E}^{\star},\ \varepsilon_{1}\varepsilon_{2}\dots\varepsilon_{n}\mapsto\sigma(\varepsilon_{1})\sigma(\varepsilon_{2})\dots\sigma(\varepsilon_{n})$. Such a tree automorphism satisfies $g_{\sigma}(ww')=g_{\sigma}(w)g_{\sigma}(w')$ for every pair of words $w,w'$ in $\mathcal{E}^{\star}$. Therefore every renormalization of $g_{\sigma}$ is equal to $g_{\sigma}$ and any subgroup of $\Aut(T_{d})$ generated by such tree automorphisms induced by some permutations of $\mathcal{E}$ is self-similar.
\end{example}

\noindent Remark that every tree automorphism $g\in\Aut(T_{d})$ satisfies
$$\forall\varepsilon\in\mathcal{E},\forall w\in\mathcal{E}^{\star},\ g(\varepsilon w)=\Big(g|_{\mathcal{E}^{1}}(\varepsilon)\Big)\Big(\mathcal{R}_{\varepsilon}(g)(w)\Big)$$
Consequently any tree automorphism is entirely described by its renormalizations at every vertex in the first level together with its restriction on the first level which describes how the regular subtrees rooted at every vertex in the first level are interchanged. That provides a convenient way to encode tree automorphisms in order to make computations in $\Aut(T_{d})$.

\newpage

\begin{defn}\label{DefWreathRecursion}
Every tree automorphism $g\in\Aut(T_{d})$ may be uniquely written as follows
$$g=\sigma_{g}\recursion[g_{0},g_{1},\dots,g_{d-1}]$$
where
\begin{itemize}
\item $\sigma_{g}=g|_{\mathcal{E}^{1}}\in\Sym(\mathcal{E})$ is called the root permutation of $g$
\item and for every letter $\varepsilon\in\mathcal{E}$, $g_{\varepsilon}=\mathcal{R}_{\varepsilon}(g)\in\Aut(T_{d})$ is the renormalization of $g$ at $\varepsilon$
\end{itemize}
This decomposition is called the wreath recursion of $g$.
\end{defn}
More precisely the map $g\mapsto\left(g|_{\mathcal{E}^{1}},\left(\mathcal{R}_{0}(g),\mathcal{R}_{1}(g),\dots,\mathcal{R}_{d-1}(g)\right)\right)$ is a group isomorphism from $\Aut(T_{d})$ onto the semi-direct product $\Sym(\mathcal{E})\ltimes(\Aut(T_{d}))^{d}$ called the permutational wreath product (its binary operation is described below). Remark that a subgroup $G$ of $\Aut(T_{d})$ is said self-similar if and only if its image under this group isomorphism is a subgroup of $\Sym(\mathcal{E})\ltimes G^{d}$.

It is more convenient to think wreath recursion as in the graph below.
$$\xymatrix @!0 @R=2pc @C=10pc { 0 \ar[r]^{\textstyle g_{0}} & \sigma_{g}(0) \\ 1 \ar[r]^{\textstyle g_{1}} & \sigma_{g}(1) \\ \dots & \dots \\ d-1 \ar[r]^{\textstyle g_{d-1}} & \sigma_{g}(d-1) }$$
Be aware that each arrow does not depict the map on its label. In fact, all the arrows describe the root permutation $\sigma_{g}$ whereas the labels correspond to the renormalizations of $g$. In practice the arrows are often ``tied'' to sort out the image on the right-hand side in the same order as on the left-hand side. The root permutation $\sigma_{g}$ is then described by intertwined arrows. Furthermore a label is often forgotten if the corresponding renormalization is the identity tree automorphism $\Id\in\Aut(T_{d})$.

This kind of graph provides an easy way to compute with wreath recursions. Namely for every pair of tree automorphisms $g,h$ in $\Aut(T_{d})$, one get
$$\xymatrix @!0 @R=2pc @C=10pc { 0 \ar[r]^{\textstyle g_{0}} & \sigma_{g}(0) \ar[r]^-{\textstyle h_{\sigma_{g}(0)}} & \sigma_{h}(\sigma_{g}(0)) \\ 1 \ar[r]^{\textstyle g_{1}} & \sigma_{g}(1) \ar[r]^-{\textstyle h_{\sigma_{g}(1)}} & \sigma_{h}(\sigma_{g}(1)) \\ \dots & \dots & \dots \\ d-1 \ar[r]^{\textstyle g_{d-1}} & \sigma_{g}(d-1) \ar[r]^-{\textstyle h_{\sigma_{g}(d-1)}} & \sigma_{h}(\sigma_{g}(d-1)) }$$

\begin{lem}\label{LemComputationsWreathRecursion}
The wreath recursion of a composition of two tree automorphism $g,h$ in $\Aut(T_{d})$ is given by
$$g.h=(\sigma_{h}\circ\sigma_{g})\bigrecursion[g_{0}.h_{\sigma_{g}(0)},g_{1}.h_{\sigma_{g}(1)},\dots,g_{d-1}.h_{\sigma_{g}(d-1)}]$$
In particular, the inverse wreath recursion of a tree automorphism $g$ in $\Aut(T_{d})$ is given by
$$g^{-1}=\sigma_{g}^{-1}\bigrecursion[g_{\sigma_{g}^{-1}(0)}^{-1},g_{\sigma_{g}^{-1}(1)}^{-1},\dots,g_{\sigma_{g}^{-1}(d-1)}^{-1}]$$
\end{lem}

\newpage

\begin{example}[Example - the adding machine]\quad\\
Every word in $\mathcal{E}^{\star}$ may be thought as a $d$-ary integer whose digits are written from left to right. Let $g\in\Aut(T_{d})$ be the adding machine on $\mathcal{E}^{\star}$, namely the process of adding one to the left most digit of every $d$-ary integer (with the convention that adding one to the word $11\dots 1\in\mathcal{E}^{n}$ gives the word $00\dots 0\in\mathcal{E}^{n}$ of same length). More precisely, the adding machine $g$ is recursively defined by
$$\forall\varepsilon\in\mathcal{E},\forall w\in\mathcal{E}^{\star},\ g(\varepsilon w)=\left\{\begin{array}{ll} (\varepsilon+1)w & \text{if }\varepsilon\in\{0,1,\dots,d-2\} \\ 0g(w) & \text{if }\varepsilon=d-1 \end{array}\right.$$
Then $g$ may be seen as the wreath recursion $g=\sigma\recursion[\Id,\Id,\dots,\Id,g]$ where $\sigma$ is the cyclic permutation such that $\sigma(\varepsilon)=\varepsilon+1$ if $\varepsilon\in\{0,1,\dots,d-2\}$ and $\sigma(d-1)=0$, namely $\sigma=(0,1,\dots,d-1)$ (using circular notation).
$$\xymatrix @!0 @R=1.7pc @C=6pc { 0 \ar[rd] & 0 \\ 1 \ar[rd] & 1 \\ \dots & 2 \\ d-2 \ar[rd] & \dots \\ d-1 \ar@(r,l)[ruuuu]^{\textstyle g} & d-1 }$$
Lemma \ref{LemComputationsWreathRecursion} allows to compute easily the inverse wreath recursion $g^{-1}=\sigma^{-1}\recursion[g^{-1},\Id,\dots,\Id]$ since
$$\xymatrix @!0 @R=1.7pc @C=6pc { 0 \ar[rd] & 0 \ar@(r,l)[rdddd]^{\textstyle g^{-1}} & 0 \\ 1 \ar[rd] & 1 \ar[ru] & 1 \\ \dots & 2 \ar[ru] & \dots \\ d-2 \ar[rd] & \dots & d-2 \\ d-1 \ar@(r,l)[ruuuu]^{\textstyle g} & d-1 \ar[ru] & d-1}$$
gives after ``untying''
$$\xymatrix @!0 @R=1.7pc @C=12pc { 0 \ar[r] & 0 \\ 1 \ar[r] & 1 \\ \dots & \dots \\ d-2 \ar[r] & d-2 \\ d-1 \ar[r]^{\textstyle g.g^{-1}=\Id} & d-1 }$$

A similar computation gives $g^{d}=\recursion[g,g,\dots,g]$, and thus a quick induction shows that the adding machine acts as a cyclic permutation of order $d^{n}$ on the $n$-th level of $T_{d}$ for every $n\geqslant 1$.
\end{example}

\begin{example}[Example - Hanoi Towers group (due to Grigorchuk and \u{S}uni\'{c} \cite{SelfSimilarity}))]\quad\\
The popular Towers of Hanoi Problem deals with three rods and a given number $n$ of disks of different sizes which can slide onto every rod. The problem starts with all the disks in ascending order of size on one rod making a conical shape (see Figure \ref{FigStartHanoi}), and consists to move the entire stack to another rod with respect to the following rules
\begin{enumerate}
\item Only one disk may be moved at a time.
\item Each move consists of taking the upper disk from one of the rods and sliding it onto another rod, on top of the other disks that may already be present on that rod.
\item No disk may be placed on top of a smaller disk.
\end{enumerate}

\begin{figure}[!h]
\begin{center}
\includegraphics[width=10cm]{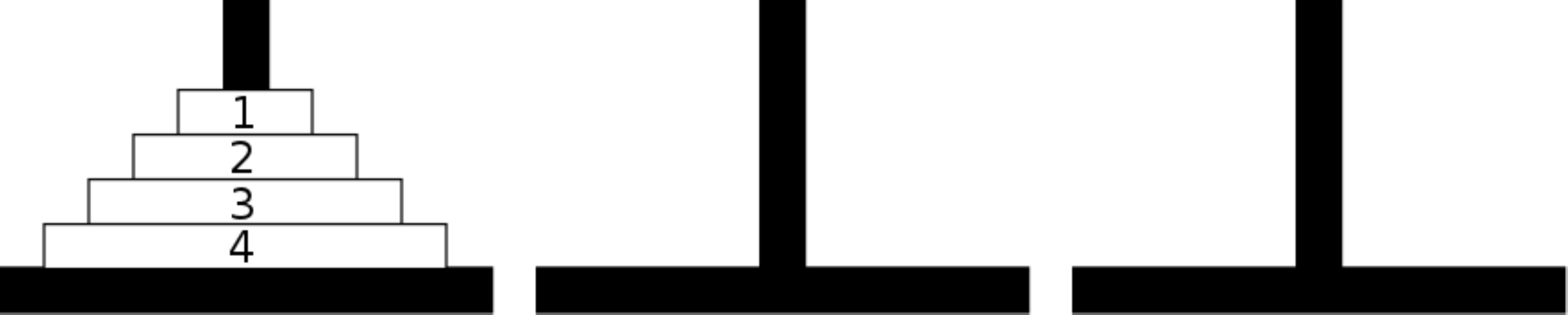}\\
\caption{The starting configuration 0000 for the Hanoi Towers with four disks}\label{FigStartHanoi}
\end{center}
\end{figure}

If the $n$ disks $1,2,\dots,n$ are labeled with their size ($1$ being the smallest and $n$ the largest) and each rod is labeled with one letter from the alphabet $\mathcal{E}=\{0,1,2\}$, then every word $w=\varepsilon_{1}\varepsilon_{2}\dots\varepsilon_{n}\in\mathcal{E}^{n}$ of length $n$ encodes a unique configuration of the problem in which the $k$-th disk is placed on the rod $\varepsilon_{k}$ (and then the order of disks on any rod is determined by their size). Figure \ref{FigStartHanoi} and Figure \ref{FigExHanoi} depict the starting configuration $0000\in\mathcal{E}^{4}$ and the configuration $1210\in\mathcal{E}^{4}$ for $n=4$ disks.

\begin{figure}[!h]
\begin{center}
\includegraphics[width=10cm]{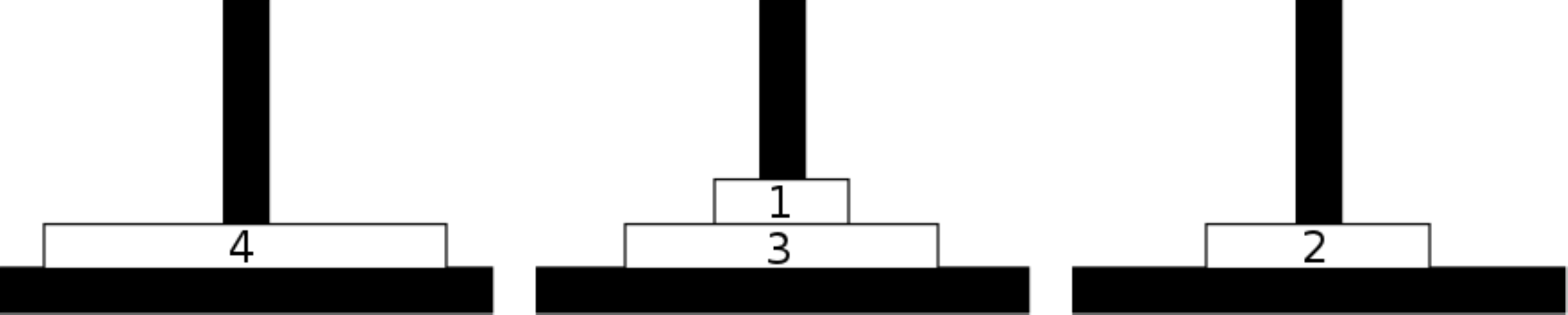}\\
\caption{The configuration 1210 for the Hanoi Towers with four disks}\label{FigExHanoi}
\end{center}
\end{figure}

It turns out that each move between two rods is represented by either of the following wreath recursions

\begin{center}
\begin{tabular}{|c|c|c|}
\hline 
$a=(1,2)\recursion[a,\Id,\Id]$ & $b=(0,1)\recursion[\Id,\Id,b]$ & $c=(0,2)\recursion[\Id,c,\Id]$ \\ 
for a move & for a move & for a move \\
between rods $1$ and $2$ & between rods $0$ and $1$ & between rods $0$ and $2$ \\
\hline
$\xymatrix @!0 @R=2pc @C=5pc { 0 \ar[r]^{\textstyle a} & 0 \\ 1 \ar[rd] & 1 \\ 2 \ar[ru] & 2}$ & $\xymatrix @!0 @R=2pc @C=5pc { 0 \ar[rd] & 0 \\ 1 \ar[ru] & 1 \\ 2 \ar[r]^{\textstyle b} & 2 }$ & $\xymatrix @!0 @R=2pc @C=5pc { 0 \ar@(r,l)[rdd] & 0 \\ 1 \ar[r]^(0.25){\textstyle c } & 1 \\ 2 \ar@(r,l)[ruu] & 2 }$ \\ 
\hline
\end{tabular}
\end{center}

For instance, one can go from the starting position 0000 in Figure \ref{FigStartHanoi} to the position 1210 in Figure \ref{FigExHanoi} by the following sequence of basic moves between two rods
$$b\rightarrow c\rightarrow a\rightarrow b\rightarrow a$$
In terms of wreath recursions, that gives
$$\xymatrix @!0 @R=2pc @C=5pc { 0 \ar[rd] & 0 \ar@(r,l)[rdd] & 0 \ar[r]^{\textstyle a} & 0 \ar[rd] & 0 \ar[r]^{\textstyle a} & 0 \\ 1 \ar[ru] & 1 \ar[r]^(0.25){\textstyle c } & 1 \ar[rd] & 1 \ar[ru] & 1 \ar[rd] & 1 \\ 2 \ar[r]^{\textstyle b} & 2 \ar@(r,l)[ruu] & 2 \ar[ru] & 2 \ar[r]^{\textstyle b} & 2 \ar[ru] & 2 }$$
or shortly after ``untying'': $b.c.a.b.a=(0,1)\recursion[c.b,a,b.a]$
$$\xymatrix @!0 @R=2pc @C=25pc { 0 \ar[rd]^(0.2){\textstyle c.b} & 0 \\ 1 \ar[ru]_(0.2){\textstyle a} & 1 \\ 2 \ar[r]^{\textstyle b.a} & 2 }$$

The Hanoi Towers group $\mathcal{H}$ is defined to be the subgroup of $\Aut(T_{3})$ generated by the wreath recursions $a,b,c$. It follows that the Towers of Hanoi Problem is equivalent to find an element $g$ in the Hanoi Tower group $\mathcal{H}=\langle a,b,c\rangle$ (that is a sequence of basic moves between two rods) such that the image of the starting configuration (the word $00\dots0\in\mathcal{E}^{n}$) is a goal configuration: $g(00\dots0)=11\dots1\text{ or }22\dots2$.

Furthermore remark that the Hanoi Towers group $\mathcal{H}$ is self-similar since every renormalization of the wreath recursions $a,b,c$ is either $a,b,c$ or $\Id$.
\end{example}

The following lemma is often used to prove that a given tree automorphism is actually equal to the identity tree automorphism. Indeed it may happen although its wreath recursion is not trivial. For instance, it turns out that the wreath recursion $g=\recursion[g,g,\dots,g]$ is actually the identity tree automorphism $\Id\in\Aut(T_{d})$ (for instance by applying lemma below).

\begin{lem}\label{LemIdentityTreeAutomorphism}
Let $g_{1},g_{2},\dots,g_{m}$ be $m\geqslant 1$ tree automorphisms in $\Aut(T_{d})$ such that
\begin{itemize}
\item every root permutation $\sigma_{g_{k}}=g_{k}|_{\mathcal{E}^{1}}$ is the identity permutation on the the alphabet $\mathcal{E}$
\item and every renormalization $g_{k,\varepsilon}=\mathcal{R}_{\varepsilon}(g_{k})$ belongs to the subgroup of $\Aut(T_{d})$ generated by $g_{1},g_{2},\dots,g_{m}$
\end{itemize}
In terms of wreath recursions, $g_{1},g_{2},\dots,g_{m}$ are assumed to be written as follows
$$\left\{\begin{array}{rcl} g_{1} & = & \recursion[g_{1,0},g_{1,1},\dots,g_{1,d-1}] \\ g_{2} & = & \recursion[g_{2,0},g_{2,1},\dots,g_{2,d-1}] \\ & \dots & \\ g_{m} & = & \recursion[g_{m,0},g_{m,1},\dots,g_{m,d-1}] \end{array}\right.\quad\text{where}\quad\forall k,\forall\varepsilon\in\mathcal{E},\ g_{k,\varepsilon}\in\Big\langle g_{1},g_{2},\dots,g_{m}\Big\rangle$$
Then $g_{1},g_{2},\dots,g_{m}$ are all equal to the identity tree automorphism $\Id\in\Aut(T_{d})$.
\end{lem}
\begin{proof}
Let $n\geqslant 2$ be an integer and assume by induction that every $g_{k}$ acts as the identity on the set of all words of lenght $n-1$. Let $w=\varepsilon_{1}\varepsilon_{2}\dots\varepsilon_{n}\in\mathcal{E}^{\star}$ be a word of length $n$. The image of $w$ under any tree automorphism $g_{k}$ may be written as follows
$$g_{k}(w)=g_{k}(\varepsilon_{1}\varepsilon_{2}\dots\varepsilon_{n})=\Big(g_{k}|_{\mathcal{E}^{1}}(\varepsilon_{1})\Big)\Big(\mathcal{R}_{\varepsilon_{1}}(g_{k})(\varepsilon_{2}\dots\varepsilon_{n})\Big)=\Big(\sigma_{g_{k}}(\varepsilon_{1})\Big)\Big(g_{k,\varepsilon_{1}}(\varepsilon_{2}\dots\varepsilon_{n})\Big)$$
The first assumption gives $\sigma_{g_{k}}(\varepsilon_{1})=\varepsilon_{1}$. Furthermore it follows from the second assumption and from the inductive hypothesis that $g_{k,\varepsilon_{1}}(\varepsilon_{2}\dots\varepsilon_{n})=\varepsilon_{2}\dots\varepsilon_{n}$ since $\varepsilon_{2}\dots\varepsilon_{n}$ is a word of length $n-1$. Finally $g_{k}(w)=(\varepsilon_{1})(\varepsilon_{2}\dots\varepsilon_{n})=w$ and the result follows by induction (the inductive start is given by the first assumption).
\end{proof}

In practice, to show that a given tree automorphism $g\in\Aut(T_{d})$ is actually the identity tree automorphism, the aim is to find together with $g=g_{1}$ some tree automorphisms $g_{2},\dots,g_{m}$ which satisfy assumptions from Lemma \ref{LemIdentityTreeAutomorphism}.

\newpage

\subsection{Partial self-covering}

\begin{defn}
Let $\mathcal{M}$ be a path connected and locally path connected topological space. A partial self-covering of $\mathcal{M}$ is a degree $d\geqslant 2$ covering $p:\mathcal{M}_{1}\rightarrow\mathcal{M}$ where $\mathcal{M}_{1}\subseteq\mathcal{M}$. 

A partial self-covering can be iterated and the iterates, denoted by $p^{n}:\mathcal{M}_{n}\rightarrow\mathcal{M}$ where $\mathcal{M}_{n}=p^{-n}(\mathcal{M})\subseteq\mathcal{M}$, are also partial self-coverings.
$$\xymatrix{\dots \ar[r] & \mathcal{M}_{3} \ar[r]_{\textstyle p} & \mathcal{M}_{2} \ar[r]_{\textstyle p} & \mathcal{M}_{1} \ar[r]_{\textstyle p} & \mathcal{M} }$$
\end{defn}

\begin{example}[Examples]\quad\\
Let $f$ be a post-critically finite branched covering map on the topological sphere $\S^{2}$ and denote by $P_{f}$ its post-critical set. Since $P_{f}\subset f^{-1}(P_{f})$, $f$ induces the following partial self-covering.
$$f:\mathcal{M}_{1}=\S^{2}\backslash f^{-1}(P_{f})\longrightarrow\mathcal{M}=\S^{2}\backslash P_{f}$$
The same holds for post-critically finite rational maps on the Riemann sphere $\CC$ or for post-critically finite polynomials map on the complex plane $\C$.
\end{example}

Recall that a partial self-covering satisfies the following path and homotopy lifting properties.

\begin{prop}\label{PropLift}
\begin{description}
\item[(1)] For every path $\ell$ in $\mathcal{M}$ with base point $\ell(0)=t\in\mathcal{M}$ and any preimage $x\in p^{-1}(t)$, there exists a unique path $\mathscr{L}_{x}$ in $\mathcal{M}_{1}$ with base point $\mathscr{L}_{x}(0)=x$ such that $p\circ\mathscr{L}_{x}=\ell$ (see the commutative diagram below). $\mathscr{L}_{x}$ is called the $p$-lift of $\ell$ from $x$.
$$\xymatrix{ && [0,1] \ar[lld]_{\textstyle \mathscr{L}_{x}}\ar[d]^{\textstyle \ell} \\ \mathcal{M}_{1} \ar[rr]_{\textstyle p} && \mathcal{M} }$$
\item[(2)] Furthermore if $l:[0,1]\times[0,1]\rightarrow\mathcal{M}$ is a homotopy of paths with $l(0,.)=\ell$ then there exists a unique homotopy of paths $L_{x}:[0,1]\times[0,1]\rightarrow\mathcal{M}_{1}$ such that $L_{x}(s,.)$ is the $p$-lift of $l(s,.)$ from $l(s,0)$ for every $s\in[0,1]$ (in particular $L_{x}(0,.)=\mathscr{L}_{x}$ for $s=0$).

\item[(3)] Therefore for every loop $\gamma$ in $\mathcal{M}$ with base point $\ell(0)=t\in\mathcal{M}$ and any preimage $x\in p^{-1}(t)$, the terminal point $y=\Gamma_{x}(1)$ of the $p$-lift $\Gamma_{x}$ of $\gamma$ from $x$ depends on $\gamma$ only through its homotopy class $[\gamma]\in\pi_{1}(\mathcal{M},t)$. Since $y$ is also a preimage of $t$ under $p$, it turns out that the fundamental group $\pi_{1}(\mathcal{M},t)$ acts on $p^{-1}(t)$ by $[\gamma]x=\Gamma_{x}(1)$.
$$\xymatrix @!0 @R=2pc @C=3pc { x \ar@{~>}[dd]_{\textstyle \Gamma_{x}}\ar[rrd]^{\textstyle p} && \\ && t \ar@{~>}@(ur,dr)[]^{\textstyle [\gamma]} \\ y\ar[rru]_{\textstyle p} && }$$
\end{description}
\end{prop}
The same holds for $p^{n}:\mathcal{M}_{n}\rightarrow\mathcal{M}$ as well, namely the fundamental group $\pi_{1}(\mathcal{M},t)$ acts on the set of preimages $p^{-n}(t)$ by $[\gamma]x=\Gamma_{x}(1)$ where $\Gamma_{x}$ is the $p^{n}$-lift of $\gamma$ from $x$.

\section{Tree of preimages}\label{SecTreePreimages}

Let $p:\mathcal{M}_{1}\rightarrow\mathcal{M}$ be a partial self-covering of degree $d\geqslant 2$ and $t$ be a point in $\mathcal{M}$.

\subsection{A right-hand tree}

\begin{defn}\label{DefRightTree}
 The (right-hand) tree of preimages $T(p,t)$ is the rooted tree defined as follows
\begin{description}
\item[Root:] the point $t$
\item[Vertices:] the abstract set of preimages $\bigsqcup_{n\geqslant 0} p^{-n}(t)$
\item[Edges:] all the pairs $\{p(x),x\}$ where $p(x)\in p^{-n}(t)$ and $x\in p^{-(n+1)}(t)$ for some $n\geqslant 0$
\end{description}
By ``abstract set'', one distinguishes a same point that belongs to two distinct levels. More precisely, some preimages corresponding to distinct levels may coincide in $\mathcal{M}_{1}$ but are distinguished in $T(p,t)$ (in particular every edge is well defined).
\end{defn}

The graph below shows the first levels of a tree of preimages of a degree $d=2$ partial self-covering.
$$\xymatrix @!0 @R=0.8pc @C=6pc {
&&& \\
&& \ar@{<-}[ru]^{p}\ar@{<-}[rd]_{p} & \\
&&& \\
& \ar@{<-}[ruu]^{p}\ar@{<-}[rdd]_{p} && \\
&&& \\
&& \ar@{<-}[ru]^{p}\ar@{<-}[rd]_{p} & \\
&&& \\
t \ar@{<-}[ruuuu]^{p}\ar@{<-}[rdddd]_{p} &&& \hspace{30pt}\dots \\
&&& \\
&& \ar@{<-}[ru]^{p}\ar@{<-}[rd]_{p} & \\
&&& \\
& \ar@{<-}[ruu]^{p}\ar@{<-}[rdd]_{p} && \\
&&& \\
&& \ar@{<-}[ru]^{p}\ar@{<-}[rd]_{p} & \\
&&&
}$$

\begin{example}\quad\\
Consider the degree $d=2$ partial self-covering $Q_{0}:\C\backslash\{0\}\rightarrow\C\backslash\{0\},\ z\mapsto z^{2}$ and let $t=1$ be the root. The first two levels of the tree of preimages $T(Q_{0},1)$ are then
$$\xymatrix @!0 @R=0.8pc @C=6pc {
&&& \\
&& 1 \ar@{<-}[ru]^{Q_{0}}\ar@{<-}[rd]_{Q_{0}} & \\
&&& \\
& 1 \ar@{<-}[ruu]^{Q_{0}}\ar@{<-}[rdd]_{Q_{0}} && \\
&&& \\
&& -1 \ar@{<-}[ru]^{Q_{0}}\ar@{<-}[rd]_{Q_{0}} & \\
&&& \\
t=1 \ar@{<-}[ruuuu]^{Q_{0}}\ar@{<-}[rdddd]_{Q_{0}} &&& \hspace{30pt}\dots \\
&&& \\
&& i \ar@{<-}[ru]^{Q_{0}}\ar@{<-}[rd]_{Q_{0}} & \\
&&& \\
& -1\ar@{<-}[ruu]^{Q_{0}}\ar@{<-}[rdd]_{Q_{0}} && \\
&&& \\
&& -i \ar@{<-}[ru]^{Q_{0}}\ar@{<-}[rd]_{Q_{0}} & \\
&&&
}$$
\end{example}

It turns out that the tree of preimages $T(p,t)$ of a degree $d$ partial self-covering is isomorphic to the regular rooted tree $T_{d}$. However there is no canonical choice for a tree isomorphism between them. Actually Definition \ref{DefRightTree} does not provide a canonical labeling of all vertices of $T(p,t)$.

\subsection{A left-hand tree}

\begin{defn}
A labeling choice $(L)$ for the partial covering $p:\mathcal{M}_{1}\rightarrow\mathcal{M}$ and the base point $t\in\mathcal{M}$ is the data of
\begin{itemize}
\item a numbering of the set $p^{-1}(t)=\{x_{0},x_{1},\dots,x_{d-1}\}$
\item and for every letter $\varepsilon\in \mathcal{E}=\{0,1,\dots,d-1\}$, a path $\ell_{\varepsilon}$ in $\mathcal{M}$ from $t$ to $x_{\varepsilon}$ 
\end{itemize}
\end{defn}

\begin{defn}\label{DefLeftTree}
Let $(L)$ be a labeling choice. Applying Proposition \ref{PropLift} for every path $\ell_{\varepsilon}$, one can consider the $p$-lifts of $\ell_{\varepsilon}$ from $x_{w}$ for every $w\in \mathcal{E}^{1}=\mathcal{E}$. Terminal points of those lifts are denoted by $x_{\varepsilon w}$ and are preimages of $t$ under $p^{2}$. One can iterate this process by induction. More precisely from every preimage $x_{w}\in p^{-n}(t)$ labelled with a word $w\in\mathcal{E}^{\star}$ of length $n$, there is a unique $p^{n}$-lift of $\ell_{\varepsilon}$ whose terminal point is a preimage of $t$ under $p^{n+1}$ denoted by $x_{\varepsilon w}$.

Then the left-hand tree of preimages $T^{(L)}(p,t)$ is the rooted tree defined as follows
\begin{description}
\item[Root:] the point $t=x_{\emptyset}$
\item[Vertices:] the abstract set of labeled preimages $\bigsqcup_{n\geqslant 0} p^{-n}(t) =\bigsqcup_{n\geqslant 0}\{x_{w}\ /\ w\in \mathcal{E}^{n}\}$
\item[Edges:] all the pairs $\{x_{\varepsilon w},x_{w}\}$ where $w$ is a word in $\mathcal{E}^{\star}$ and $\varepsilon$ is a letter in $\mathcal{E}$
\end{description}
\end{defn}

The graph below shows the first levels of a left-hand tree of preimages of a degree $d=2$ partial self-covering (for convenience, the lifts of $\ell_{\varepsilon}$ are still denoted by $\ell_{\varepsilon}$ for every $\varepsilon\in \mathcal{E}$).
$$\xymatrix @!0 @R=0.7pc @C=6pc {
&&&& \\
& x_{000} \ar@{~>}[lu]_{\ell_{0}}\ar@{~>}[ld]^{\ell_{1}} &&& \\
&&&& \\
&& x_{00} \ar@{~>}[luu]_{\ell_{0}}\ar@{~>}[ldd]^{\ell_{1}} && \\
&&&& \\
& x_{100} \ar@{~>}[lu]_{\ell_{0}}\ar@{~>}[ld]^{\ell_{1}} &&& \\
&&&& \\
&&& x_{0} \ar@{~>}[luuuu]_{\ell_{0}}\ar@{~>}[ldddd]^{\ell_{1}} & \\
&&&& \\
& x_{010} \ar@{~>}[lu]_{\ell_{0}}\ar@{~>}[ld]^{\ell_{1}} &&& \\
&&&& \\
&& x_{10} \ar@{~>}[luu]_{\ell_{0}}\ar@{~>}[ldd]^{\ell_{1}} && \\
&&&& \\
& x_{110} \ar@{~>}[lu]_{\ell_{0}}\ar@{~>}[ld]^{\ell_{1}} &&& \\
&&&& \\
\dots\hspace{30pt} &&&& t \ar@{~>}[luuuuuuuu]_{\ell_{0}}\ar@{~>}[ldddddddd]^{\ell_{1}} \\
&&&& \\
& x_{001} \ar@{~>}[lu]_{\ell_{0}}\ar@{~>}[ld]^{\ell_{1}} &&& \\
&&&& \\
&& x_{01} \ar@{~>}[luu]_{\ell_{0}}\ar@{~>}[ldd]^{\ell_{1}} && \\
&&&& \\
& x_{101} \ar@{~>}[lu]_{\ell_{0}}\ar@{~>}[ld]^{\ell_{1}} &&& \\
&&&& \\
&&& x_{1} \ar@{~>}[luuuu]_{\ell_{0}}\ar@{~>}[ldddd]^{\ell_{1}} & \\
&&&& \\
& x_{011} \ar@{~>}[lu]_{\ell_{0}}\ar@{~>}[ld]^{\ell_{1}} &&& \\
&&&& \\
&& x_{11} \ar@{~>}[luu]_{\ell_{0}}\ar@{~>}[ldd]^{\ell_{1}} && \\
&&&& \\
& x_{111} \ar@{~>}[lu]_{\ell_{0}}\ar@{~>}[ld]^{\ell_{1}} &&& \\
&&&
}$$

It turns out that the set of all vertices of a left-hand tree of preimages $T^{(L)}(p,t)$ is the same as that one of the (right-hand) tree of preimages $T(p,t)$. However $T^{(L)}(p,t)$ provides a labeling of all vertices of $T(p,t)$ (according to Definition \ref{DefLeftTree}).

\begin{example}\quad\\
For $Q_{0}:\C\backslash\{0\}\rightarrow\C\backslash\{0\},\ z\mapsto z^{2}$, the preimages of $t=1$ are $x_{0}=1$ and $x_{1}=-1$. Choose the paths $\ell_{0},\ell_{1}$ as follows
$$\xymatrix{ x_{1} & 0 & x_{0} \ar@{~>}@/_{2pc}/[ll]_{\textstyle \ell_{1}}\ar@{~>}@(dr,ur)[]_{\textstyle \ell_{0}} }$$
Lifting these paths gives (for convenience, the lifts of $\ell_{0},\ell_{1}$ are still denoted by $\ell_{0},\ell_{1}$ respectively)
$$\xymatrix{ & x_{10} & \\ x_{01} \ar@{~>}@/_{1pc}/[rd]_{\textstyle \ell_{1}}\ar@{~>}@(ul,dl)[]_{\textstyle \ell_{0}} & 0 & x_{00} \ar@{~>}@/_{1pc}/[lu]_{\textstyle \ell_{1}}\ar@{~>}@(dr,ur)[]_{\textstyle \ell_{0}} \\ & x_{11} & }$$
One can deduce the first two levels of the left-hand tree of preimages $T^{(L)}(Q_{0},1)$
$$\xymatrix @!0 @R=0.6pc @C=6pc {
&&& \\
& x_{00}=1 \ar@{~>}[lu]_{\ell_{0}}\ar@{~>}[ld]^{\ell_{1}} && \\
&&& \\
&& x_{0}=1 \ar@{~>}[luu]_{\ell_{0}}\ar@{~>}[ldd]^{\ell_{1}} & \\
&&& \\
& x_{10}=i \ar@{~>}[lu]_{\ell_{0}}\ar@{~>}[ld]^{\ell_{1}} && \\
&&& \\
\dots\hspace{30pt} &&& t=1 \ar@{~>}[luuuu]_{\ell_{0}}\ar@{~>}[ldddd]^{\ell_{1}} \\
&&& \\
& x_{01}=-1 \ar@{~>}[lu]_{\ell_{0}}\ar@{~>}[ld]^{\ell_{1}} && \\
&&& \\
&& x_{1}=-1 \ar@{~>}[luu]_{\ell_{0}}\ar@{~>}[ldd]^{\ell_{1}} & \\
&&& \\
& x_{11}=-i \ar@{~>}[lu]_{\ell_{0}}\ar@{~>}[ld]^{\ell_{1}} && \\
&&&
}$$

Notice that distinct choices of paths $\ell_{0},\ell_{1}$ induce distinct left-hand trees of preimages. For instance
$$\xymatrix{ x_{1} & 0 & x_{0} \ar@{~>}@/^{2pc}/[ll]^{\textstyle \ell_{1}}\ar@{~>}@(dr,ur)[]_{\textstyle \ell_{0}} }$$
gives after lifting
$$\xymatrix{ & x_{11} & \\ x_{01} \ar@{~>}@/^{1pc}/[ru]^{\textstyle \ell_{1}}\ar@{~>}@(ul,dl)[]_{\textstyle \ell_{0}} & 0 & x_{00} \ar@{~>}@/^{1pc}/[ld]^{\textstyle \ell_{1}}\ar@{~>}@(dr,ur)[]_{\textstyle \ell_{0}} \\ & x_{10} & }$$
and thus the left-hand tree of preimages $T^{(L)}(Q_{0},1)$ becomes
$$\xymatrix @!0 @R=0.6pc @C=6pc {
&&& \\
& x_{00}=1 \ar@{~>}[lu]_{\ell_{0}}\ar@{~>}[ld]^{\ell_{1}} && \\
&&& \\
&& x_{0}=1 \ar@{~>}[luu]_{\ell_{0}}\ar@{~>}[ldd]^{\ell_{1}} & \\
&&& \\
& x_{10}=-i \ar@{~>}[lu]_{\ell_{0}}\ar@{~>}[ld]^{\ell_{1}} && \\
&&& \\
\dots\hspace{30pt} &&& t=1 \ar@{~>}[luuuu]_{\ell_{0}}\ar@{~>}[ldddd]^{\ell_{1}} \\
&&& \\
& x_{01}=-1 \ar@{~>}[lu]_{\ell_{0}}\ar@{~>}[ld]^{\ell_{1}} && \\
&&& \\
&& x_{1}=-1 \ar@{~>}[luu]_{\ell_{0}}\ar@{~>}[ldd]^{\ell_{1}} & \\
&&& \\
& x_{11}=i \ar@{~>}[lu]_{\ell_{0}}\ar@{~>}[ld]^{\ell_{1}} && \\
&&&
}$$
for this choice of $\ell_{0},\ell_{1}$.
\end{example}

\begin{prop}\label{PropEdges}
For every vertex $x_{w}$ of $T^{(L)}(p,t)$ labeled with a word $w\in\mathcal{E}^{\star}$, the preimages of $x_{w}$ under $p$ are
$$p^{-1}(x_{w})=\{x_{w\varepsilon}\ /\ \varepsilon\in \mathcal{E}\}$$
\end{prop}
\begin{proof}
Since $x_{w}$ has exactly $d$ preimages under $p$, one only needs to check that $x_{w\varepsilon}$ is a preimage of $x_{w}$ for every word $w\in\mathcal{E}^{\star}$ and every letter $\varepsilon\in\mathcal{E}$. The main idea is that the path from $x_{\varepsilon}$ to $x_{w\varepsilon}$ formed by concatenation of some lifts (following Definition \ref{DefLeftTree}) is a $p$-lift of the path from $t$ to $x_{w}$ formed by concatenation of some lifts (following Definition \ref{DefLeftTree} as well). However an induction will be used in order to avoid overloaded notations.

The result obviously holds for the empty word, that is $p^{-1}(x_{\emptyset})=p^{-1}(t)=\{x_{0},x_{1},\dots,x_{d-1}\}$. Let $w=\varepsilon_{1}\varepsilon_{2}\dots\varepsilon_{n}\in \mathcal{E}^{\star}$ be a word of length $n\geqslant 1$. From Definition \ref{DefLeftTree}, $x_{w}$ is the terminal point of the $p^{n-1}$-lift of $\ell_{\varepsilon_{1}}$ from $x_{\varepsilon_{2}\dots\varepsilon_{n}}$, say $\mathscr{L}_{\varepsilon_{2}\dots\varepsilon_{n}}^{\varepsilon_{1}}$. Assume by induction that the result holds for the word $\varepsilon_{2}\dots\varepsilon_{n}$ of length $n-1$ and let $x_{\varepsilon_{2}\dots\varepsilon_{n}\varepsilon}$ be a preimage of $x_{\varepsilon_{2}\dots\varepsilon_{n}}$ for some $\varepsilon\in \mathcal{E}$. Following Definition \ref{DefLeftTree}, there is a unique $p^{n}$-lift of $\ell_{\varepsilon_{1}}$ from $x_{\varepsilon_{2}\dots\varepsilon_{n}\varepsilon}$, say $\mathscr{L}_{\varepsilon_{2}\dots\varepsilon_{n}\varepsilon}^{\varepsilon_{1}}$, whose terminal point is denoted by $x_{x_{\varepsilon_{1}\varepsilon_{2}\dots\varepsilon_{n}\varepsilon}}=x_{w\varepsilon}$. Since $\mathscr{L}_{\varepsilon_{2}\dots\varepsilon_{n}\varepsilon}^{\varepsilon_{1}}$ is a $p$-lift of $\mathscr{L}_{\varepsilon_{2}\dots\varepsilon_{n}}^{\varepsilon_{1}}$, $x_{w\varepsilon}$ is a preimage of $x_{w}$ under $p$. The result follows by induction.
\vspace{-10pt}
$$\xymatrix{ x_{w}=x_{\varepsilon_{1}\varepsilon_{2}\dots\varepsilon_{n}} &&& x_{\varepsilon_{2}\dots\varepsilon_{n}} \ar@{~>}[lll]_-{\textstyle \mathscr{L}_{\varepsilon_{2}\dots\varepsilon_{n}}^{\varepsilon_{1}}} \\ &&& \\ x_{w\varepsilon}=x_{\varepsilon_{1}\varepsilon_{2}\dots\varepsilon_{n}\varepsilon} \ar[uu]^{\textstyle p} &&& x_{\varepsilon_{2}\dots\varepsilon_{n}\varepsilon} \ar[uu]^{\textstyle p} \ar@{~>}[lll]^-{\textstyle \mathscr{L}_{\varepsilon_{2}\dots\varepsilon_{n}\varepsilon}^{\varepsilon_{1}}} }$$
\vspace{-20pt}
\end{proof}

\vspace{-10pt}
\begin{example}\quad\\
For $Q_{0}:\C\backslash\{0\}\rightarrow\C\backslash\{0\},\ z\mapsto z^{2}$ with $t=1$, $x_{0}=1$ and $x_{1}=-1$, choose the paths $\ell_{0},\ell_{1}$ as follows
\vspace{-5pt}
$$\xymatrix{ x_{1} & 0 & x_{0} \ar@{~>}@/_{2pc}/[ll]_{\textstyle \ell_{1}}\ar@{~>}@(dr,ur)[]_{\textstyle \ell_{0}} }$$
\vspace{-10pt}
Recall the first two levels of the left-hand and right-hand trees of preimages.
$$\xymatrix @!0 @R=0.6pc @ C=6pc {
&&&&&& \\
& x_{00}=1 \ar@{~>}[lu]_{\ell_{0}}\ar@{~>}[ld]^{\ell_{1}} &&&& 1 \ar@{<-}[ru]^{Q_{0}}\ar@{<-}[rd]_{Q_{0}} & \\
&&&&&& \\
&& x_{0}=1 \ar@{~>}[luu]_{\ell_{0}}\ar@{~>}[ldd]^{\ell_{1}} && 1 \ar@{<-}[ruu]^{Q_{0}}\ar@{<-}[rdd]_{Q_{0}} && \\
&&&&&& \\
& x_{10}=i \ar@{~>}[lu]_{\ell_{0}}\ar@{~>}[ld]^{\ell_{1}} &&&& -1 \ar@{<-}[ru]^{Q_{0}}\ar@{<-}[rd]_{Q_{0}} & \\
&&&&&& \\
\dots\hspace{30pt} &&& t=1 \ar@{~>}[luuuu]_{\ell_{0}}\ar@{~>}[ldddd]^{\ell_{1}}\ar@{<-}[ruuuu]^{Q_{0}}\ar@{<-}[rdddd]_{Q_{0}} &&& \hspace{30pt}\dots \\
&&&&&& \\
& x_{01}=-1 \ar@{~>}[lu]_{\ell_{0}}\ar@{~>}[ld]^{\ell_{1}} &&&& i \ar@{<-}[ru]^{Q_{0}}\ar@{<-}[rd]_{Q_{0}} && \\
&&&&&& \\
&& x_{1}=-1 \ar@{~>}[luu]_{\ell_{0}}\ar@{~>}[ldd]^{\ell_{1}} && -1 \ar@{<-}[ruu]^{Q_{0}}\ar@{<-}[rdd]_{Q_{0}} && \\
&&&&&& \\
& x_{11}=-i \ar@{~>}[lu]_{\ell_{0}}\ar@{~>}[ld]^{\ell_{1}} &&&& -i \ar@{<-}[ru]^{Q_{0}}\ar@{<-}[rd]_{Q_{0}} & \\
&&&&&&
}$$
One can deduce the induced labeling on the first levels of the tree of preimages $T(Q_{0},1)$
$$\xymatrix @!0 @R=0.3pc @C=7pc {
&&&& \\
&&& x_{000}=1 \ar@{-}[ru]\ar@{-}[rd] & \\
&&&& \\
&& x_{00}=1 \ar@{-}[ruu]\ar@{-}[rdd] && \\
&&&& \\
&&& x_{001}=-1 \ar@{-}[ru]\ar@{-}[rd] & \\
&&&& \\
& x_{0}=1 \ar@{-}[ruuuu]\ar@{-}[rdddd] &&& \\
&&&& \\
&&& x_{010}=i \ar@{-}[ru]\ar@{-}[rd] & \\
&&&& \\
&& x_{01}=-1 \ar@{-}[ruu]\ar@{-}[rdd] && \\
&&&& \\
&&& x_{011}=-i \ar@{-}[ru]\ar@{-}[rd] & \\
&&&& \\
t=1 \ar@{-}[ruuuuuuuu]\ar@{-}[rdddddddd] &&&& \hspace{30pt}\dots \\
&&&& \\
&&& x_{100}=e^{i\pi/4} \ar@{-}[ru]\ar@{-}[rd] & \\
&&&& \\
&& x_{10}=i \ar@{-}[ruu]\ar@{-}[rdd] && \\
&&&& \\
&&& x_{101}=e^{-3i\pi/4} \ar@{-}[ru]\ar@{-}[rd] & \\
&&&& \\
& x_{1}=-1 \ar@{-}[ruuuu]\ar@{-}[rdddd] &&& \\
&&&& \\
&&& x_{101}=e^{3i\pi/4} \ar@{-}[ru]\ar@{-}[rd] & \\
&&&& \\
&& x_{11}=-i \ar@{-}[ruu]\ar@{-}[rdd] && \\
&&&& \\
&&& x_{111}=e^{-i\pi/4} \ar@{-}[ru]\ar@{-}[rd] & \\
&&&&
}$$
\end{example}

\begin{prop}\label{PropTreeIsomorphism}
Label all vertices of the tree of preimages $T(p,t)$ like those of $T^{(L)}(p,t)$ for some given labeling choice $(L)$. Then the following holds
\vspace{-10pt}
\begin{description}
\item[-] The edges of $T(p,t)$ are all the pairs $\{x_{w},x_{w\epsilon}\}$ where $w$ is a word in $\mathcal{E}^{\star}$ and $\varepsilon$ is a letter in $\mathcal{E}$ (compare with the edges of $T^{(L)}(p,t)$)
\vspace{-10pt}
\item[-] The map $\varphi^{(L)}:x_{w}\mapsto w$ is a tree isomorphism from $T(p,t)$ onto the regular rooted tree $T_{d}$
\end{description}
\vspace{-10pt}
\end{prop}
\begin{proof}
The first point and the edge preserving axiom for $\varphi^{(L)}$ follow from Proposition \ref{PropEdges}. The level preserving axiom comes from Definition \ref{DefLeftTree}.
\end{proof}

\section{Iterated monodromy group}

Let $p:\mathcal{M}_{1}\rightarrow\mathcal{M}$ be a partial self-covering of degree $d\geqslant 2$ and $t$ be a point in $\mathcal{M}$.

\subsection{Monodromy action}

From Proposition \ref{PropLift}, the fundamental group $\pi_{1}(\mathcal{M},t)$ acts on $p^{-n}(t)$ for every $n\geqslant 0$, that is on the set of vertices of level $n$ in the tree of preimages $T(p,t)$.

\begin{defn}\label{DefMonodromyAction}
The action of $\pi_{1}(\mathcal{M},t)$ on the set of all vertices in the tree of preimages $T(p,t)$ is called the monodromy action. It may be seen as the following group homomorphism.
$$\Phi:\pi_{1}(\mathcal{M},t)\rightarrow\Sym\left(\bigsqcup_{n\geqslant 0} p^{-n}(t)\right),\ [\gamma]\mapsto\Big(\Phi_{[\gamma]}:x\mapsto[\gamma]x\Big)$$

Furthermore for any labeling choice $(L)$, the tree isomorphism $\varphi^{(L)}$ from Proposition \ref{PropTreeIsomorphism} induces a monodromy action on the set of all words $\mathcal{E}^{\star}$ defined as follows
$$\Phi^{(L)}:\pi_{1}(\mathcal{M},t)\rightarrow\Sym\left(\mathcal{E}^{\star}\right),\ [\gamma]\mapsto\left(\Phi^{(L)}_{[\gamma]}:w\mapsto[\gamma]w\right)\quad\text{where}\quad x_{[\gamma]w}=[\gamma]x_{w}$$
\end{defn}
More precisely the monodromy action induced by a given labeling choice $(L)$ is defined as follows
$$\forall[\gamma]\in\pi_{1}(\mathcal{M},t),\ \forall w\in\mathcal{E}^{\star},\quad\Phi^{(L)}_{[\gamma]}(w)=[\gamma]w=\varphi^{(L)}\Big([\gamma]x_{w}\Big)=\left(\varphi^{(L)}\circ\Phi_{[\gamma]}\circ\left(\varphi^{(L)}\right)^{-1}\right)(w)$$
In particular the monodromy action induced by another labeling choice $(L')$ is conjugate to that one coming from $(L)$ by the map $\varphi^{(L),(L')}=\varphi^{(L)}\circ\left(\varphi^{(L')}\right)^{-1}\in\Aut(T_{d})$ in the following way
$$\forall[\gamma]\in\pi_{1}(\mathcal{M},t),\quad\Phi^{(L')}_{[\gamma]}=\left(\varphi^{(L),(L')}\right)^{-1}\circ\Phi^{(L)}_{[\gamma]}\circ\left(\varphi^{(L),(L')}\right)$$
As a consequence, the monodromy action on $\mathcal{E}^{\star}$ is well defined up to conjugation by a tree automorphism of the form $\varphi^{(L),(L')}$ for any pair of labeling choices $(L)$ and $(L')$.

In practice, it is more convenient to use $\Phi^{(L)}$ than $\Phi$ for a ``relevant'' labeling choice $(L)$ in order to compute the monodromy action of a partial self-covering (since $(L)$ provides a labeling of every vertex in the tree of preimages).

As Theorem \ref{TheoTreeAction} will show, the monodromy action actually acts by tree automorphisms, namely $\Phi_{[\gamma]}\in\Aut(T(p,t))$ for every homotopy class $[\gamma]\in\pi_{1}(\mathcal{M},t)$ (and therefore $\Phi^{(L)}_{[\gamma]}\in\Aut(T_{d})$ for any labeling choice $(L)$).

\newpage

\begin{example}\quad\\
For $Q_{0}:\C\backslash\{0\}\rightarrow\C\backslash\{0\},\ z\mapsto z^{2}$ with $t=1$, $x_{0}=1$ and $x_{1}=-1$, choose the paths $\ell_{0},\ell_{1}$ as follows
\vspace{-10pt}\\
$$\xymatrix{ x_{1} & 0 & x_{0} \ar@{~>}@/_{2pc}/[ll]_{\textstyle \ell_{1}}\ar@{~>}@(dr,ur)[]_{\textstyle \ell_{0}} }$$
The fundamental group $\pi_{1}(\C\backslash\{0\},1)$ may be described as the infinite cyclic group generated by the homotopy class $[\gamma]$ coming from the following loop
\vspace{-15pt}\\
$$\xymatrix{ && \\ & 0 & t \ar@{~>} `ul[ul] `dl[dl]_{\textstyle \gamma} `dr[] [] \\ && }$$
\vspace{-10pt}\\
where $\gamma$ surrounds the point $0$ in a counterclockwise motion. Lifting this loop gives
\vspace{-10pt}\\
$$\xymatrix{ x_{1} \ar@{~>}@/_{1pc}/[rr]_{\textstyle \Gamma_{1}} & 0 & x_{0} \ar@{~>}@/_{1pc}/[ll]_{\textstyle \Gamma_{0}}}$$
\vspace{-10pt}\\
One can deduce the monodromy action of $[\gamma]$ on the first level $Q_{0}^{-1}(t)$, that is $[\gamma]x_{0}=x_{1}$ and $[\gamma]x_{1}=x_{0}$. Equivalently the action on $\mathcal{E}^{1}=\{0,1\}$ is given by $[\gamma]0=1$ and $[\gamma]1=0$.
Lifting the loop $\gamma$ by $Q_{0}^{2}:z\mapsto z^{4}$ gives
\vspace{-15pt}\\
$$\xymatrix{ & x_{10} \ar@{~>}@/_{1pc}/[ld]_{\textstyle \Gamma_{10}} & \\ x_{01} \ar@{~>}@/_{1pc}/[rd]_{\textstyle \Gamma_{01}} & 0 & x_{00} \ar@{~>}@/_{1pc}/[lu]_{\textstyle \Gamma_{00}} \\ & x_{11} \ar@{~>}@/_{1pc}/[ru]_{\textstyle \Gamma_{11}} & }$$
One can deduce the action of $[\gamma]$ on the second level $Q_{0}^{-2}(t)$ or equivalently on $\mathcal{E}^{2}=\{00,10,01,11\}$, that is $[\gamma]00=10$, $[\gamma]01=01$, $[\gamma]01=11$ and $[\gamma]11=00$.
\end{example}

More generally, it turns out that the $p^{n}$-lifts of any loop $\gamma$ are needed to compute the monodromy action of $[\gamma]$ on the $n$-th level. However the following lemma gives a recursive way to compute the monodromy action that only uses the $p$-lifts of $\gamma$.

\begin{lem}\label{LemComputation}
Let $[\gamma]$ be a homotopy class in $\pi_{1}(\mathcal{M},t)$ and $w$ be a word in $\mathcal{E}^{\star}$. For every letter $\varepsilon\in \mathcal{E}$, denote by $\Gamma_{\varepsilon}$ the $p$-lift of $\gamma$ from $x_{\varepsilon}$. Then, for any labeling choice $(L)$,
$$[\gamma]\varepsilon w=\Big([\gamma]\varepsilon\Big)\Big(\left[\ell_{\varepsilon}.\Gamma_{\varepsilon}.\ell_{[\gamma]\varepsilon}^{-1}\right]w\Big)$$
The following graph depicts the concatenation of paths $\ell_{\varepsilon}.\Gamma_{\varepsilon}.\ell_{[\gamma]\varepsilon}^{-1}$.
$$\xymatrix @!0 @R=2pc @C=4pc { x_{\varepsilon} \ar@{~>}[dd]_{\textstyle \Gamma_{\varepsilon}} && \\ && t \ar@{~>}@(ur,dr)[]^{\textstyle [\gamma]}\ar@{~>}[llu]_{\textstyle \ell_{\varepsilon}}\ar@{~>}[lld]^-{\textstyle \ell_{[\gamma]\varepsilon}} \\ [\gamma]x_{\varepsilon}=x_{[\gamma]\varepsilon} && }$$
 In particular, it is a loop with base point $t$ and the homotopy class $\left[\ell_{\varepsilon}.\Gamma_{\varepsilon}.\ell_{[\gamma]\varepsilon}^{-1}\right]$ is well defined in $\pi_{1}(\mathcal{M},t)$.
\end{lem}

\begin{proof}
Let $\delta$ be the loop $\ell_{\varepsilon}.\Gamma_{\varepsilon}.\ell_{[\gamma]\varepsilon}^{-1}$. Consider the $p^{n}$-lift of $\delta$ from $x_{w}$ and denote by $x_{v}$ its terminal point, that is $v=[\delta]w$. This lift is exactly the concatenation of three paths $\mathscr{L}_{w}^{\varepsilon}.\Gamma_{\varepsilon w}.\big(\mathscr{L}_{v}^{[\gamma]\varepsilon}\big)^{-1}$ where
\begin{itemize}
\item $\mathscr{L}_{w}^{\varepsilon}$ is the $p^{n}$-lift of $\ell_{\varepsilon}$ from $x_{w}$ (whose terminal point is $x_{\varepsilon w}$ from Definition \ref{DefLeftTree})
\item $\Gamma_{\varepsilon w}$ is the $p^{n+1}$-lift of $\gamma$ from $x_{\varepsilon w}$ (whose terminal point is $[\gamma]x_{\varepsilon w}=x_{[\gamma]\varepsilon w}$)
\item $\mathscr{L}_{v}^{[\gamma]\varepsilon}$ is the $p^{n}$-lift of $\ell_{[\gamma]\varepsilon}$ from $x_{v}$ (whose terminal point is $x_{([\gamma]\varepsilon)v}$ from Definition \ref{DefLeftTree})
\end{itemize}
$$\xymatrix{ x_{\varepsilon w} \ar@{~>}[dd]_{\textstyle \Gamma_{\varepsilon w}} &&& x_{w} \ar@{~>}[lll]_{\textstyle \mathscr{L}_{w}^{\varepsilon}} \\ &&& \\ x_{[\gamma]\varepsilon w}=x_{([\gamma]\varepsilon)v} &&& x_{v} \ar@{~>}[lll]^-{\textstyle \mathscr{L}_{v}^{[\gamma]\varepsilon}} }$$
In particular $\Gamma_{\varepsilon w}$ and $\mathscr{L}_{v}^{[\gamma]\varepsilon}$ have the same terminal point, and thus $[\gamma]\varepsilon w=([\gamma]\varepsilon)v=([\gamma]\varepsilon)([\delta]w)$.
\end{proof}

\begin{example}\quad\\
Go further with the partial self-covering $Q_{0}:\C\backslash\{0\}\rightarrow\C\backslash\{0\},\ z\mapsto z^{2}$ using the same labeling choice as before. Recall that the $Q_{0}$-lifts $\Gamma_{0},\Gamma_{1}$ of $\gamma$ are
$$\xymatrix{ x_{1} \ar@{~>}@/_{1pc}/[rr]_{\textstyle \Gamma_{1}} & 0 & x_{0} \ar@{~>}@/_{1pc}/[ll]_{\textstyle \Gamma_{0}} \ar@{~>}@/_{3pc}/[ll]_{\textstyle \ell_{1}} \ar@{~>}@(dr,ur)[]_{\textstyle \ell_{0}} }$$
The loop $\ell_{0}.\Gamma_{0}.\ell_{[\gamma]0}^{-1}=\ell_{0}.\Gamma_{0}.\ell_{1}^{-1}$ is homotopic to the constant loop at base point $t$ and the loop $\ell_{1}.\Gamma_{1}.\ell_{[\gamma]1}^{-1}=\ell_{1}.\Gamma_{1}.\ell_{0}^{-1}$ is homotopic to $\gamma$. It follows from Lemma \ref{LemComputation} that
$$\forall w\in \mathcal{E}^{\star},\quad[\gamma]0w=1w\quad\text{and}\quad[\gamma]1w=0\left([\gamma]w\right)$$
Therefore the tree automorphism $g=(w\mapsto[\gamma]w)\in\Aut(T_{2})$ may be described as the wreath recursion $g=(0,1)\recursion[\Id,g]$ which is the adding machine on $T_{2}$, namely the process of adding one to a binary integer (see Section \ref{SecTreeAutomorphism}). One can depict this monodromy action on every vertex of the tree of preimages $T(Q_{0},1)$ as follows
$$\xymatrix @!0 @R=0.4pc @C=9pc {
&&&& \\
&&& x_{000} \ar@{-}[ru]\ar@{-}[rd] \ar@{~>}@/_{3pc}/[dddddddddddddddd] & \\
&&&& \\
&& x_{00}=1 \ar@{-}[ruu]\ar@{-}[rdd] \ar@{~>}@/_{2pc}/[dddddddddddddddd] && \\
&&&& \\
&&& x_{001} \ar@{-}[ru]\ar@{-}[rd] \ar@{~>}@/_{3pc}/[dddddddddddddddd] & \\
&&&& \\
& x_{0}=1 \ar@{-}[ruuuu]\ar@{-}[rdddd] \ar@{~>}@/_{1pc}/[dddddddddddddddd] &&& \\
&&&& \\
&&& x_{010} \ar@{-}[ru]\ar@{-}[rd] \ar@{~>}@/_{3pc}/[dddddddddddddddd] & \\
&&&& \\
&& x_{01}=-1 \ar@{-}[ruu]\ar@{-}[rdd] \ar@{~>}@/_{2pc}/[dddddddddddddddd] && \\
&&&& \\
&&& x_{011} \ar@{-}[ru]\ar@{-}[rd] \ar@{~>}@/_{3pc}/[dddddddddddddddd] & \\
&&&& \\
t=1 \ar@{-}[ruuuuuuuu]\ar@{-}[rdddddddd] &&&& \hspace{30pt}\dots \\
&&&& \\
&&& x_{100} \ar@{-}[ru]\ar@{-}[rd] \ar@{~>}@/_{1pc}/[uuuuuuuu] & \\
&&&& \\
&& x_{10}=i \ar@{-}[ruu]\ar@{-}[rdd] \ar@{~>}@/_{1pc}/[uuuuuuuu] && \\
&&&& \\
&&& x_{101} \ar@{-}[ru]\ar@{-}[rd] \ar@{~>}@/_{1pc}/[uuuuuuuu] & \\
&&&& \\
& x_{1}=-1 \ar@{-}[ruuuu]\ar@{-}[rdddd] \ar@{~>}@/_{1pc}/[uuuuuuuuuuuuuuuu] &&& \\
&&&& \\
&&& x_{110} \ar@{-}[ru]\ar@{-}[rd] \ar@{~>}@/_{3pc}/[uuuuuuuuuuuuuuuuuuuu] & \\
&&&& \\
&& x_{11}=-i \ar@{-}[ruu]\ar@{-}[rdd] \ar@{~>}@/_{2pc}/[uuuuuuuuuuuuuuuuuuuuuuuu] && \\
&&&& \\
&&& x_{111} \ar@{-}[ru]\ar@{-}[rd] \ar@{~>}@/_{4pc}/[uuuuuuuuuuuuuuuuuuuuuuuuuuuu] & \\
&&&&
}$$
Remark that the monodromy action of $[\gamma]$ actually acts by tree automorphism on $T(Q_{0},1)$. In particular it is edge preserving (whatever the labeling of all vertices) as it is shown in the graph above. 
\end{example}

\newpage

The previous remark can be generalized for any monodromy action.

\begin{thm}\label{TheoTreeAction}
The monodromy action acts by tree automorphisms on the tree of preimages $T(p,t)$. Equivalently speaking, the monodromy action may be seen as the following group homomorphism.
$$\Phi:\pi_{1}(\mathcal{M},t)\rightarrow\Aut\left(T(p,t)\right),\ [\gamma]\mapsto\Big(\Phi_{[\gamma]}:x\mapsto[\gamma]x\Big)$$
\end{thm}
This result motivates the introduction of the right-hand tree of preimages $T(p,t)$. Indeed the monodromy action is not necessarily edge preserving on the left-hand tree of preimages $T^{(L)}(p,t)$.
\begin{proof}
Let $[\gamma]$ be a homotopy class in $\pi_{1}(\mathcal{M},t)$. If follows from Definition \ref{DefMonodromyAction} and Proposition \ref{PropLift} that $\Phi_{[\gamma]}$ is level preserving. So one only needs to check that $\Phi_{[\gamma]}$ is furthermore edge preserving.

Let $\{p(x),x\}$ be an edge of $T(p,t)$ where $x\in p^{-(n+1)}(t)$ for some $n\geqslant 0$. Let $\Gamma_{p(x)}$ be the $p^{n}$-lift of $\gamma$ from $p(x)$ (whose terminal point is $[\gamma]p(x))$ and $\Gamma_{x}$ be the $p^{n+1}$-lift of $\gamma$ from $x$ (whose terminal point is $[\gamma]x$). Since $\Gamma_{x}$ is a $p$-lift of $\Gamma_{p(x)}$, it follows that $p([\gamma]x)=[\gamma]p(x)$ and thus $\{[\gamma]p(x),[\gamma]x\}$, which is also equal to $\{\Phi_{[\gamma]}(p(x)),\Phi_{[\gamma]}(x)\}$, is an edge of $T(p,t)$.
$$\xymatrix{p(x) \ar@{~>}[dd]_{\textstyle \Gamma_{p(x)}} &&& x \ar[lll]_{\textstyle p} \ar@{~>}[dd]^{\textstyle \Gamma_{x}} \\ &&& \\ [\gamma]p(x)=p([\gamma]x)  &&& [\gamma]x \ar[lll]^-{\textstyle p}}$$
\end{proof}

The monodromy action may also be seen as a group homomorphism from $\pi_{1}(\mathcal{M},t)$ into $\Aut(T_{d})$. In this case, its image is well defined for a given labeling choice $(L)$ or up to conjugation by a tree automorphism of the form $\varphi^{(L),(L')}=\varphi^{(L)}\circ\left(\varphi^{(L')}\right)^{-1}$ for any pair of labeling choices $(L)$ and $(L')$.

However these group homomorphisms are in general not injective or equivalently, in terms of group action, the monodromy action is in general not faithful.

\subsection{Definition}

\begin{defn}\label{DefIteratedMonodromyGroup}
The iterated monodromy group of the degree $d\geqslant 2$ partial self-covering $p:\mathcal{M}_{1}\rightarrow\mathcal{M}$ with base point $t\in\mathcal{M}$ is defined to be
$$\IMG(p,t)=\pi_{1}(\mathcal{M},t)/\Ker(\Phi)\quad\text{where}\quad\Ker(\Phi)=\Big\{[\gamma]\in\pi_{1}(\mathcal{M},t)\ /\ \forall x\in\bigsqcup_{n\geqslant 0} p^{-n}(t),\ [\gamma]x=x\Big\}$$
Equivalently speaking, it may be seen as
\begin{itemize}
\item the image of the monodromy action in $\Aut(T(p,t))$ which induces a faithful action by tree automorphisms on the tree of preimages $T(p,t)$
\item the following subgroup of $\Aut(T_{d})$
$$\IMG(p,t)=\Big\{(w\mapsto[\gamma]w)\in\Aut(T_{d})\ /\ [\gamma]\in\pi_{1}(\mathcal{M},t)\Big\}$$
which is defined for a given labeling choice $(L)$
\end{itemize}
\end{defn}

Recall that, up to group isomorphism, the fundamental group $\pi_{1}(\mathcal{M},t)$ does not depend on the choice of base point $t\in\mathcal{M}$. The same obviously holds, up to tree isomorphism, for the tree of preimages $T(p,t)$ as well. Consequently, up to group isomorphism, the iterated monodromy group $\IMG(p,t)$ only depends on the partial self-covering $p:\mathcal{M}_{1}\rightarrow\mathcal{M}$.

The definition of the iterated monodromy group $\IMG(p,t)$ as a subgroup of $\Aut(T_{d})$ depends on a labeling choice $(L)$. Recall that another labeling choice $(L')$ induces a monodromy action on $T_{d}$ which is conjugate to that one coming from $(L)$ by the map $\varphi^{(L),(L')}=\varphi^{(L)}\circ\left(\varphi^{(L')}\right)^{-1}\in\Aut(T_{d})$. Therefore the iterated monodromy group $\IMG(p,t)$ is well defined as subgroup of $\Aut(T_{d})$ up to conjugation by a tree automorphism of the form $\varphi^{(L),(L')}$ for any pair of labeling choices $(L)$ and $(L')$. In particular, it is also well defined up to conjugation by a tree automorphism of $T_{d}$ (or equivalently up to post-composition with an inner automorphism of $\Aut(T_{d})$), but there is then a loss of information since $\Aut(T_{d})$ is much bigger than its subgroup of all maps of the form $\varphi^{(L),(L')}$ for any pair of labeling choices $(L)$ and $(L')$.

\begin{example}\quad\\
Back to the partial self-covering $Q_{0}:\C\backslash\{0\}\rightarrow\C\backslash\{0\},\ z\mapsto z^{2}$. Recall that $\pi_{1}(\C\backslash\{0\},1)$ is the infinite cyclic group generated by $[\gamma]$ and that $[\gamma]$ acts as the adding machine $g=(0,1)\recursion[\Id,g]$. In particular $[\gamma]$ acts as a cyclic permutation of order $2^{n}$ on the set of all vertices of level $n$, and thus the kernel of the monodromy action on the $n$-th level is $K_{n}=\langle[\gamma^{2^{n}}]\rangle$. It follows that $\Ker(\Phi)=\bigcap_{n\geqslant 0}K_{n}$ only contains the identity element and the monodromy action is faithful. Finally $\IMG(Q_{0},1)$ is isomorphic to $\pi_{1}(\C\backslash\{0\},1)$, that is isomorphic to $\Z$.
\end{example}

The following result deals with one of the many remarkable properties satisfied by iterated monodromy groups.

\begin{thm}[Nekrashevych]
The iterated monodromy group $\IMG(p,t)$ seen as a subgroup of $\Aut(T_{d})$ (for any given labeling choice $(L)$) is a self-similar group.
\end{thm}
\begin{proof}
Recall that a subgroup of $\Aut(T_{d})$ is said to be self-similar if it is invariant under any renormalization (see Definition \ref{DefSelfSimilar}). Furthermore a quick induction shows that one only needs to check that it is invariant under renormalizations at every vertex in the first level.

So let $[\gamma]$ be a homotopy class in $\pi_{1}(\mathcal{M},t)$ seen as a tree automorphism $(w\mapsto[\gamma]w)\in\Aut(T_{d})$ and $\varepsilon$ be a letter in $\mathcal{E}$. For every word $w\in\mathcal{E}^{\star}$, Lemma \ref{LemComputation} gives $[\gamma]\varepsilon w=([\gamma]\varepsilon)([\ell_{\varepsilon}.\Gamma_{\varepsilon}.\ell_{[\gamma]\varepsilon}^{-1}]w)$ with $[\ell_{\varepsilon}.\Gamma_{\varepsilon}.\ell_{[\gamma]\varepsilon}^{-1}]\in\pi_{1}(\mathcal{M},t)$ and therefore
$$\mathcal{R}_{\varepsilon}\Big(w\mapsto[\gamma]w\Big)=\Big(w\mapsto[\ell_{\varepsilon}.\Gamma_{\varepsilon}.\ell_{[\gamma]\varepsilon}^{-1}]w\Big)$$
belongs to the iterated monodromy group $\IMG(p,t)$ making it a self-similar group.
\end{proof}
For further reading about self-similar groups, see \cite{SelfSimilarGroups} and \cite{SelfSimilarity}.

Abusing notation, every tree automorphism $(w\mapsto[\gamma]w)\in\Aut(T_{d})$ induced by some homotopy class $[\gamma]\in\pi_{1}(\mathcal{M},t)$ is often denoted simply by $\gamma$ (instead of $\Phi^{(L)}_{[\gamma]}$) for convenience. As it is shown in the proof above, Lemma \ref{LemComputation} allows to compute efficiently the wreath recursion of every such tree automorphism (see Section \ref{SecTreeAutomorphism}).

\newpage

\subsection{Examples}\label{SecExamples}

\subsubsection*{Basilica group}

Consider the quadratic polynomial $Q_{-1}:z\mapsto z^{2}-1$ whose Julia set, called Basilica, is shown in Figure \ref{FigBasilica1}. Its critical point 0 is periodic of period 2.
$$\xymatrix{ -1 \ar@/_{1pc}/[r]_{1:1} & 0 \ar@/_{1pc}/[l]_{2:1} }$$
That induces a degree 2 partial self-covering $Q_{-1}:\C\backslash\{-1,0,1\}\rightarrow\C\backslash\{-1,0\}$.

\begin{figure}[!h]
\begin{center}
\includegraphics[width=10cm]{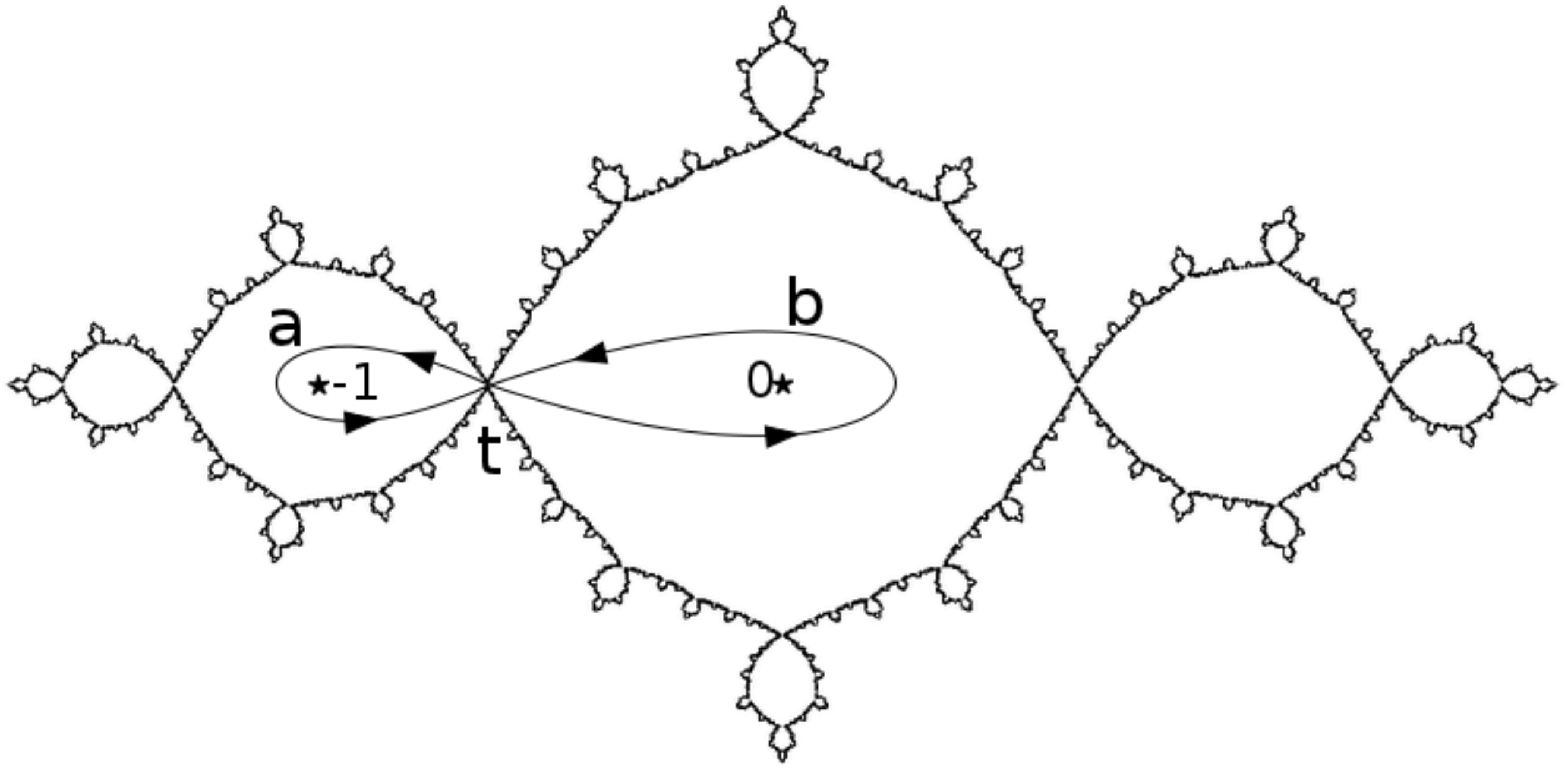}
\caption{The Basilica and two generators of $\pi_{1}(\C\backslash\{-1,0\},t)$}\label{FigBasilica1}
\end{center}
\end{figure}

Choose the fixed point $t=\frac{1-\sqrt{5}}{2}$ as base point. The fundamental group $\pi_{1}(\C\backslash\{-1,0\},t)$ may be described as the free group generated by two homotopy classes $[a],[b]$ where the loop $a$ surrounds the post-critical point $-1$ and the loop $b$ the post-critical point $0$ both in a counterclockwise motion (see Figure \ref{FigBasilica1}).

\begin{figure}[!h]
\begin{center}
\includegraphics[width=10cm]{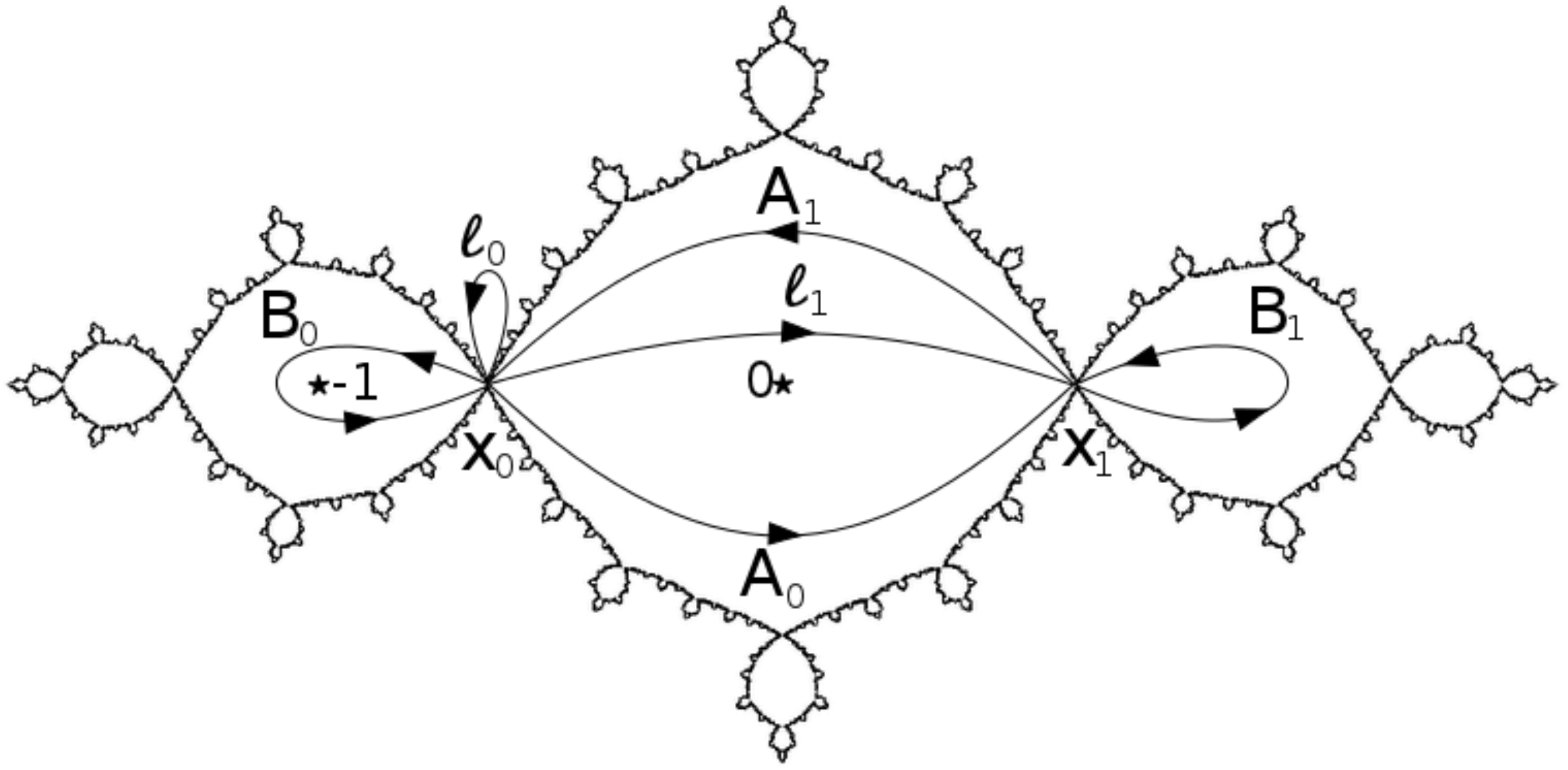}
\caption{The labeling choice $\ell_{0},\ell_{1}$ and the lifts of $a,b$}\label{FigBasilica2}
\end{center}
\end{figure}

Let $x_{0}=t,x_{1}=-t$ be the preimages of $t$ and choose two paths $\ell_{0},\ell_{1}$ from $t$ to $x_{0},x_{1}$ as it is shown in Figure \ref{FigBasilica2}. This picture also depicts the lifts of the loops $a$ and $b$. In particular, one can deduce the monodromy action of $[a]$ and $[b]$ on the first level.
$$\left\{\begin{array}{l} [a]0=1 \\ {[a]}1=0 \end{array}\right.\quad\text{and}\quad\left\{\begin{array}{l} [b]0=0 \\ {[b]}1=1 \end{array}\right.$$
Furthermore
$$\left\{\begin{array}{l} [\ell_{0}.A_{0}.\ell_{1}^{-1}]=[b] \\ {[\ell_{1}.A_{1}.\ell_{0}^{-1}]}=[1_{t}] \end{array}\right.\quad\text{and}\quad\left\{\begin{array}{l} [\ell_{0}.B_{0}.\ell_{0}^{-1}]=[a] \\ {[\ell_{1}.B_{1}.\ell_{1}^{-1}]}=[1_{t}] \end{array}\right.$$
where $[1_{t}]$ is the homotopy class of the constant loop at base point $t$ (that is the identity element of the fundamental group $\pi_{1}(\C\backslash\{-1,0\},t)$). It follows from Lemma \ref{LemComputation} that
$$\forall w\in \mathcal{E}^{\star},\quad\left\{\begin{array}{l} [a]0w=1([b]w) \\ {[a]}1w=0w \end{array}\right.\quad\text{and}\quad\left\{\begin{array}{l} [b]0w=0([a]w) \\ {[b]}1w=1w \end{array}\right.$$

Therefore the iterated monodromy group of $Q_{-1}$ seen as a subgroup of $\Aut(T_{2})$ is generated by the following wreath recursions
$$\IMG(Q_{-1},t)=\Big\langle a=(0,1)\recursion[b,\Id],b=\recursion[a,\Id]\Big\rangle$$
This group is called the Basilica group. It is not isomorphic to the free group on a set of two elements. Indeed it follows from Lemma \ref{LemComputationsWreathRecursion} that 

\begin{center}
\begin{tabular}{|c|c|c|c|}
\hline
$a=(0,1)\recursion[b,\Id]$ & $b=\recursion[a,\Id]$ & $a^{-1}=(0,1)\recursion[\Id,b^{-1}]$ & $b^{-1}=\recursion[a^{-1},\Id]$ \\
\hline
$\xymatrix @!0 @R=2pc @C=5pc { 0 \ar[rd]^(0.25){\textstyle b} & 0 \\ 1 \ar[ru] & 1 }$ & $\xymatrix @!0 @R=2pc @C=5pc { 0 \ar[r]^{\textstyle a} & 0 \\ 1 \ar[r] & 1 }$ & $\xymatrix @!0 @R=2pc @C=5pc { 0 \ar[rd] & 0 \\ 1 \ar[ru]_(0.25){\textstyle b^{-1}} & 1 }$ & $\xymatrix @!0 @R=2pc @C=5pc { 0 \ar[r]^{\textstyle a^{-1}} & 0 \\ 1 \ar[r] & 1 }$ \\
\hline
\end{tabular}
\end{center}

\noindent And thus the monodromy action of $[b^{-1}.a^{-1}.b^{-1}.a.b.a^{-1}.b.a]$ is given by the following wreath recursion
$$\xymatrix @!0 @R=2pc @C=5pc { 0 \ar[r]^{\textstyle a^{-1}} & 0 \ar[rd] & 0 \ar[r]^{\textstyle a^{-1}} & 0 \ar[rd]^(0.25){\textstyle b} & 0 \ar[r]^{\textstyle a} & 0 \ar[rd] & 0 \ar[r]^{\textstyle a} & 0 \ar[rd]^(0.25){\textstyle b} & 0 \\ 1 \ar[r] & 1 \ar[ru]_(0.25){\textstyle b^{-1}} & 1 \ar[r] & 1 \ar[ru] & 1 \ar[r] & 1 \ar[ru]_(0.25){\textstyle b^{-1}} & 1 \ar[r] & 1 \ar[ru] & 1 }$$
that is after ``untying'' (see Lemma \ref{LemComputationsWreathRecursion})
$$\xymatrix @!0 @R=2pc @C=40pc { 0 \ar[r]^{\textstyle a^{-1}.a=\Id} & 0 \\ 1 \ar[r]^{\textstyle b^{-1}.a^{-1}.b.b^{-1}.a.b=\Id} & 1}$$
Therefore $[b^{-1}.a^{-1}.b^{-1}.a.b.a^{-1}.b.a]\in\Ker(\Phi)$. In particular for every pair of generators of $\IMG(Q_{-1},t)$, the relation $b^{-1}.a^{-1}.b^{-1}.a.b.a^{-1}.b.a=\Id$ implies a relation between these generators. It follows that $\IMG(Q_{-1},t)$ is not isomorphic to the free group on a set of two elements.

Notice that the monodromy action of $Q_{-1}:\C\backslash\{-1,0,1\}\rightarrow\C\backslash\{-1,0\}$ is not faithful. More precisely this group was studied in \cite{TorsionFreeWeaklyBranchGroup} where it was in particular proved that
$$\Ker(\Phi)=\Big\langle[b^{-p}.a^{-p}.b^{-p}.a^{p}.b^{p}.a^{-p}.b^{p}.a^{p}],[a^{-2p}.b^{-p}.a^{-2p}.b^{p}.a^{2p}.b^{-p}.a^{2p}.b^{p}]\ /\ p=2^{j},j\geqslant 0\Big\rangle$$

\newpage

\subsubsection*{Chebyshev polynomials and infinite dihedral group}

Consider the degree $d\geqslant 2$ Chebyshev polynomials defined by $C_{d}:z\mapsto\cos(d\arccos(z))$ or equivalently by the following recursive formula
$$\forall z\in\C,\quad C_{0}(z)=1,\ C_{1}(z)=z\text{ and }C_{d}(z)=2zC_{d-1}(z)-C_{d-2}(z)$$
Its Julia set is the real segment $[-2,2]$. For every $k\in\{1,2,\dots,d-1\}$, the point $c_{k}=cos(\frac{\pi k}{d})$ is a simple critical point and is mapped to $C_{d}(c_{k})=(-1)^{k}$. Moreover $C_{d}(1)=1$, $C_{d}(-1)=(-1)^{d}$, and thus the post-critical set is $\{-1,1\}$. It follows that every Chebyshev polynomial induces a partial self-covering $C_{d}:\C\backslash C_{d}^{-1}(\{-1,1\})\longrightarrow\C\backslash\{-1,1\}$.

Choose $t=0$ as base point. The fundamental group $\pi_{1}(\C\backslash\{-1,1\},t)$ may be described as the free group generated by two homotopy classes $[a],[b]$ where the loop $a$ surrounds the post-critical point $-1$ and the loop $b$ the post-critical point $1$ both in a counterclockwise motion (see the graph below).
\vspace{-10pt}
$$\xymatrix{ &&&& \\ & -1 & t \ar@{~>} `ul[ul] `dl[dl]_{\textstyle a} `dr[] []  \ar@{~>} `dr[dr] `ur[ur] `ul[]_{\textstyle b} [] & 1 & \\ &&&& }$$

The preimages of $t=0$ are $x_{\varepsilon}=\cos(\frac{\pi}{2d}+\frac{\pi\varepsilon}{d})$ where the letter $\varepsilon$ belongs to the alphabet $\mathcal{E}=\{0,1,\dots,d-1\}$. For every letter $\varepsilon\in \mathcal{E}$, let $\ell_{\varepsilon}$ be the straight path from $t$ to $x_{\varepsilon}$. Remark that every real segment $[x_{k+1},x_{k}]$ contains only one critical point, namely $c_{k+1}$, and the restriction of $C_{d}$ on this segment is a double covering map onto $[-1,0]$ for $k$ even and onto $[0,1]$ for $k$ odd.

Each of the loops $a$ and $b$ has exactly $d$ lifts. The pattern of these lifts is depicted in the following graph
\vspace{-10pt}
$$\xymatrix{ &&&&&&&&&&& \\ \dots & x_{4} \ar@{~>}@/_{2pc}/[rr]_{\textstyle B_{4}} & c_{4} & x_{3} \ar@{~>}@/_{2pc}/[ll]_{\textstyle B_{3}} \ar@{~>}@/_{2pc}/[rr]_{\textstyle A_{3}} & c_{3} & x_{2} \ar@{~>}@/_{2pc}/[rr]_{\textstyle B_{2}} \ar@{~>}@/_{2pc}/[ll]_{\textstyle A_{2}} & c_{2} & x_{1} \ar@{~>}@/_{2pc}/[ll]_{\textstyle B_{1}} \ar@{~>}@/_{2pc}/[rr]_{\textstyle A_{1}} & c_{1} & x_{0} \ar@{~>} `dr[dr] `ur[ur] `ul[]_{\textstyle B_{0}} [] \ar@{~>}@/_{2pc}/[ll]_{\textstyle A_{0}} & 1 & \\ &&&&&&&&&&& }$$
and the last lift on left, which is a loop surrounding $-1$, is a lift of $b$ for $d$ even and of $a$ for $d$ odd. It follows from Lemma \ref{LemComputation} that
$$\forall w\in\mathcal{E}^{\star},\quad\left\{\begin{array}{l} [a]0w=1w \\ {[a]}1w=0w \\ {[a]}2w=3w \\ {[a]}3w=2w \\ \dots \end{array}\right.\quad\text{and}\quad\left\{\begin{array}{l} [b]0w=0([b]w) \\ {[b]}1w=2w \\ {[b]}2w=1w \\ {[b]}3w=4w \\ {[b]}4w=3w \\ \dots \end{array}\right.$$
$$\text{with}\quad\left\{\begin{array}{cc} {[b]}(d-1)w=(d-1)([a]w) & \text{if }d\text{ is even} \\ {[a]}(d-1)w=(d-1)([a]w) & \text{if }d\text{ is odd} \end{array}\right.$$

That leads to the following wreath recursions
$$\left\{\begin{array}{cccc} a=\sigma_{a}\recursion[\Id,\Id,\dots,\Id,\Id] & \text{and} & b=\sigma_{b}\recursion[b,\Id,\dots,\Id,a] & \text{if }d\text{ is even} \\ a=\sigma_{a}\recursion[\Id,\Id,\dots,\Id,a] & \text{and} & b=\sigma_{b}\recursion[b,\Id,\dots,\Id,\Id] & \text{if }d\text{ is odd} \end{array}\right.$$
$$\text{where}\quad\sigma_{a}=(0,1)(2,3)\dots\quad\text{and}\quad\sigma_{b}=(1,2)(3,4)\dots$$

Using Lemma \ref{LemComputationsWreathRecursion}, one may compute the wreath recursions $a^{2}$ and $b^{2}$.
$$\left\{\begin{array}{cccc} a^{2}=\recursion[\Id,\Id,\dots,\Id,\Id] & \text{and} & b^{2}=\recursion[b^{2},\Id,\dots,\Id,a^{2}] & \text{if }d\text{ is even} \\ a^{2}=\recursion[\Id,\Id,\dots,\Id,a^{2}] & \text{and} & b=\recursion[b^{2},\Id,\dots,\Id,\Id] & \text{if }d\text{ is odd} \end{array}\right.$$
Applying Lemma \ref{LemIdentityTreeAutomorphism}, it follows that $a^{2}=b^{2}=\Id$, or equivalently, in terms of monodromy action,
$$\Big\langle[a^{2}],[b^{2}]\Big\rangle\subset Ker(\Phi)\subset\Big\langle[a],[b]\Big\rangle=\pi_{1}(\C\backslash\{-1,1\},t)$$
Remark that if $\Ker(\Phi)$ is strictly larger than $\langle[a^{2}],[b^{2}]\rangle$, then $\Ker(\Phi)$ must contain at least one element of the form either $[(a.b)^{j}]$ or $[(b.a)^{j}]$ with $j\geqslant 1$ or of the form either $[a.(b.a)^{j}]$ or $[b.(a.b)^{j}]$ with $j\geqslant 0$. One will prove that is not the case.

The wreath recursion $a.b$ is given by (using Lemma \ref{LemComputationsWreathRecursion})
$$\left\{\begin{array}{cc} a.b=\sigma_{a.b}\recursion[\Id,b,\Id,\dots,\Id,a,\Id] & \text{if }d\text{ is even} \\ a.b=\sigma_{a.b}\recursion[\Id,b,\Id,\dots,\Id,\Id,a] & \text{if }d\text{ is odd} \end{array}\right.$$
where (using circular notation)
\begin{eqnarray*}
\sigma_{a.b} & = & \Big((1,2)(3,4)\dots\Big)\circ\Big((0,1)(2,3)\dots\Big) \\
& = & \left\{\begin{array}{cc} (0,2,4,\dots,d-4,d-2,d-1,d-3,\dots,5,3,1) & \text{if }d\text{ is even} \\ (0,2,4,\dots,d-3,d-1,d-2,d-4,\dots,5,3,1) & \text{if }d\text{ is odd} \end{array}\right.
\end{eqnarray*}
In particular $[a.b]$ acts as a cyclic permutation of order $d$ on the first level of the regular rooted tree $T_{d}$ (since $\sigma_{a.b}$ is of order $d$). Moreover, using Lemma \ref{LemComputationsWreathRecursion} again, it appears that
$$(a.b)^{d}=\recursion[a.b,a.b,a.b,\dots,b.a,b.a,b.a]$$
Recall that $(b.a)=(a.b)^{-1}$ since $a^{2}=b^{2}=\Id$, and thus $[b.a]$ also acts as a cyclic permutation of order $d$ on the first level. Therefore a quick induction implies that $[a.b]$ and $[b.a]$ act as a cyclic permutation of order $d^{n}$ on the $n$-th level of $T_{d}$. It follows that none of the elements of the form either $[(a.b)^{j}]$ or $[(b.a)^{j}]$ with $j\geqslant 1$ acts as the identity tree automorphism. The same holds as well for the elements of the form either $[a.(b.a)^{j}]$ or $[b.(a.b)^{j}]$ with $j\geqslant 0$ since $a$ and $b$ are of order 2 in $\Aut(T_{d})$. Finally

$$\Ker(\Phi)=\Big\langle[a^{2}],[b^{2}]\Big\rangle\quad\text{and}\quad\IMG(C_{d},t)=\pi_{1}(\C\backslash\{-1,1\},t)/\Ker(\Phi)=\Big\langle[a],[b]\Big\rangle/\Big\langle[a^{2}],[b^{2}]\Big\rangle$$

This group is called the infinite dihedral group. It is isomorphic to the isometry group of $\Z$ (for instance the permutations $\alpha\mapsto-\alpha$ and $\alpha\mapsto 1-\alpha$ play the same role as $a$ and $b$).

\newpage

\subsubsection*{The quadratic rational map $z\mapsto\left(\frac{z-1}{z+1}\right)^{2}$}

Consider the quadratic rational map $R:z\mapsto\left(\frac{z-1}{z+1}\right)^{2}$ whose Julia set is shown in Figure \ref{FigQuadraticExample1}. The critical points are $-1$ and $1$. Since $R(-1)=\infty$, $R^{2}(-1)=1$ and $R(1)=0$, $R^{2}(1)=1$, the post-critical set is $\{0,1,\infty\}$.
$$\xymatrix{ -1 \ar@/_{1pc}/[dr]_{2:1} & 0 \ar@/_{1pc}/[r]_{1:1} & 1 \ar@/_{1pc}/[l]_{2:1} \\ & \infty \ar@/_{1pc}/[ur]_{1:1} & }$$
That induces a degree 2 partial self-covering $R:\C\backslash\{-1,0,1\}\longrightarrow\C\backslash\{0,1\}$.

\begin{figure}[!h]
\begin{center}
\includegraphics[width=8cm]{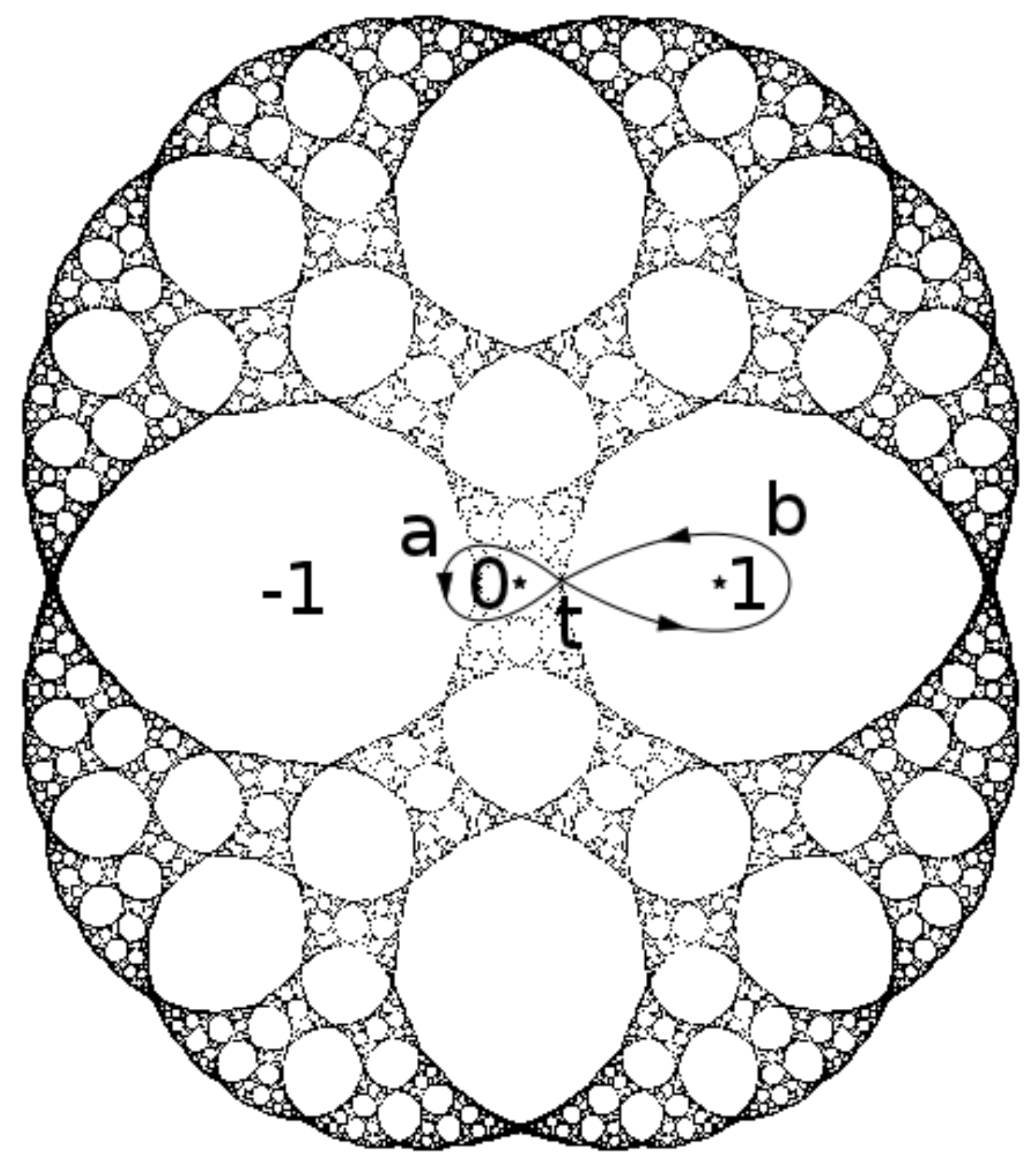}
\caption{The Julia set of $R:z\mapsto\left(\frac{z-1}{z+1}\right)^{2}$ and two generators of $\pi_{1}(\C\backslash\{0,1\},t)$}\label{FigQuadraticExample1}
\end{center}
\end{figure}

Choose the real fixed point $t\approx 0.296$ as base point. The fundamental group $\pi_{1}(\C\backslash\{0,1\},t)$ may be described as the free group generated by two homotopy classes $[a],[b]$ where the loop $a$ surrounds the post-critical point 0 and the loop $b$ the post-critical point 1 both in a counterclockwise motion (see Figure \ref{FigQuadraticExample1}).

\begin{figure}[!h]
\begin{center}
\includegraphics[width=8cm]{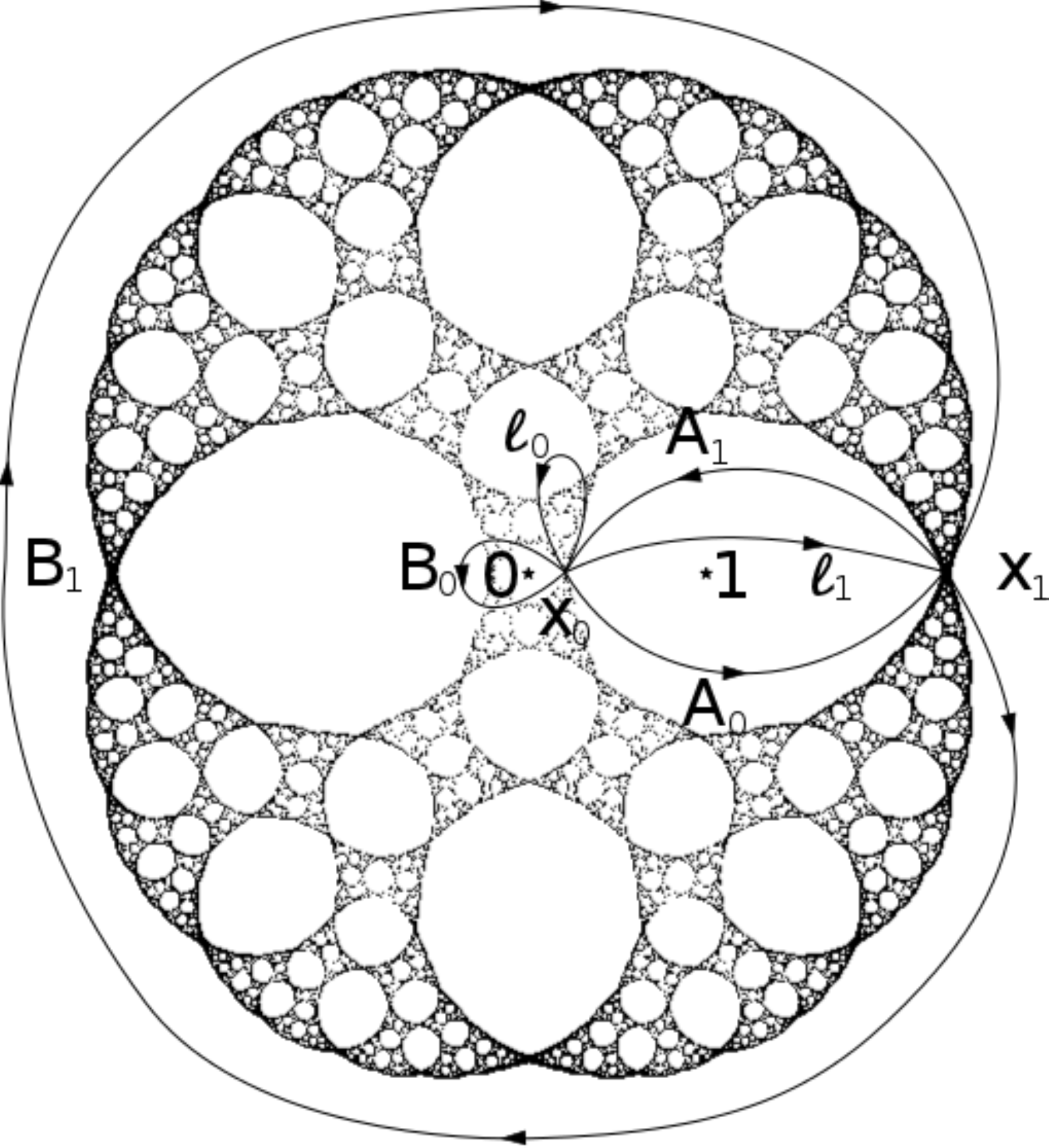}
\caption{The labeling choice $\ell_{0},\ell_{1}$ and the lifts of $a,b$}\label{FigQuadraticExample2}
\end{center}
\end{figure}

Let $x_{0}=t,x_{1}=t^{-1}\approx 3.383$ be the preimages of $t$ and choose two paths $\ell_{0},\ell_{1}$ from $t$ to $x_{0},x_{1}$ as it is shown in Figure \ref{FigQuadraticExample2}. This picture also depicts the lifts of the loops $a$ and $b$. By using Lemma \ref{LemComputation}, one can compute the monodromy action of $[a]$ and $[b]$.
$$\forall w\in \mathcal{E}^{\star},\quad\left\{\begin{array}{l} [a]0w=1([b]w) \\ {[a]}1w=0w \end{array}\right.\quad\text{and}\quad\left\{\begin{array}{l} [b]0w=0([a]w) \\ {[b]}1w=1([b^{-1}.a^{-1}]w) \end{array}\right.$$

Therefore the iterated monodromy group of $R$ seen as a subgroup of $\Aut(T_{2})$ is generated by the following wreath recursions
$$\IMG(R,t)=\Big\langle a=(0,1)\recursion[b,\Id],b=\recursion[a,b^{-1}.a^{-1}]\Big\rangle$$

Notice that these wreath recursions are not convenient to compute with since two generators occur in the renormalization $\mathcal{R}_{1}(b)=b^{-1}.a^{-1}$ and only one renormalization is the identity tree automorphism. However one may find a nicer pair of generators of $\IMG(R,t)$ by taking another pair of generators of $\pi_{1}(\C\backslash\{0,1\},t)$ and another labeling choice.

Indeed consider the tree automorphism $g$ in $\Aut(T_{2})$ whose wreath recursion is given by
$$g=\recursion[g.a,g]$$
This tree automorphism is well defined by induction on the successive levels of the regular rooted tree $T_{2}$. Now consider the following wreath recursions

\begin{center}
\begin{tabular}{ccccccc}
\hline
$a'$ & = & $g.(a.b).g^{-1}$ & = & \parbox[c]{255pt}{$\xymatrix @!0 @R=3pc @C=5pc { 0 \ar[r]^{\textstyle g.a} & 0 \ar[rd]^(0.25){\textstyle b} & 0 \ar[r]^{\textstyle a} & 0 \ar[r]^{\textstyle a^{-1}.g^{-1}} & 0 \\ 1 \ar[r]^{\textstyle g} & 1 \ar[ru] & 1 \ar[r]^{\textstyle b^{-1}.a^{-1}} & 1 \ar[r]^{\textstyle g^{-1}} & 1 }$} && \\
&&&&&& \\
&&& = & \parbox[c]{255pt}{$\xymatrix @!0 @R=3pc @C=20pc {0 \ar[rd]^(0.25){\textstyle \quad g.a.b.b^{-1}.a^{-1}.g^{-1}=\Id} & 0 \\ 1 \ar[ru]_(0.25){\textstyle \quad g.a.a^{-1}.g^{-1}=\Id} & 1 }$} & = & $(0,1)\recursion[\Id,\Id]$ \\
\hline
$b'$ & = & $g.a.g^{-1}$ & = & \parbox[c]{180pt}{$\xymatrix @!0 @R=3pc @C=5pc { 0 \ar[r]^{\textstyle g.a} & 0 \ar[rd]^(0.25){\textstyle b} & 0 \ar[r]^{\textstyle a^{-1}.g^{-1}} & 0 \\ 1 \ar[r]^{\textstyle g} & 1 \ar[ru] & 1 \ar[r]^{\textstyle g^{-1}} & 1 }$} && \\
&&&&&& \\
&&& = & \parbox[c]{180pt}{$\xymatrix @!0 @R=3pc @C=15pc {0 \ar[rd]^(0.25){\textstyle \quad g.a.b.g^{-1}=a'} & 0 \\ 1 \ar[ru]_(0.25){\textstyle \quad g.a^{-1}.g^{-1}=b'^{-1}} & 1 }$} & = & $(0,1)\recursion[a',b'^{-1}]$ \\
\hline
\end{tabular}
\end{center}

\newpage

Remark that $\langle a,b\rangle$ is obviously generated by $a.b$ and $a$. Consequently $\langle a,b\rangle$ and $\langle a',b'\rangle=g.\langle a.b,a\rangle.g^{-1}=g.\langle a,b\rangle.g^{-1}$ are conjugate subgroups in $\Aut(T_{2})$

Recall that the iterated monodromy group $\IMG(R,t)$ seen as a subgroup of $\Aut(T_{2})$ is defined for a given labeling choice. In particular the subgroup $\langle a,b\rangle$ was obtained for the labeling choice, say $(L)$, depicted in Figure \ref{FigQuadraticExample2}. With similar computations, one can show that the labeling choice $(L')$ depicted in Figure \ref{FigQuadraticExample3} gives the subgroup $\langle a',b'\rangle$. In other words, the tree automorphism $g\in\Aut(T_{2})$ corresponds to the map $\varphi^{(L),(L')}=\varphi^{(L)}\circ\left(\varphi^{(L')}\right)^{-1}\in\Aut(T_{2})$ which describes the change of labeling choices from $(L)$ to $(L')$ (see Proposition \ref{PropTreeIsomorphism} and Definition \ref{DefMonodromyAction}).

\begin{figure}[!h]
\begin{center}
\includegraphics[width=8cm]{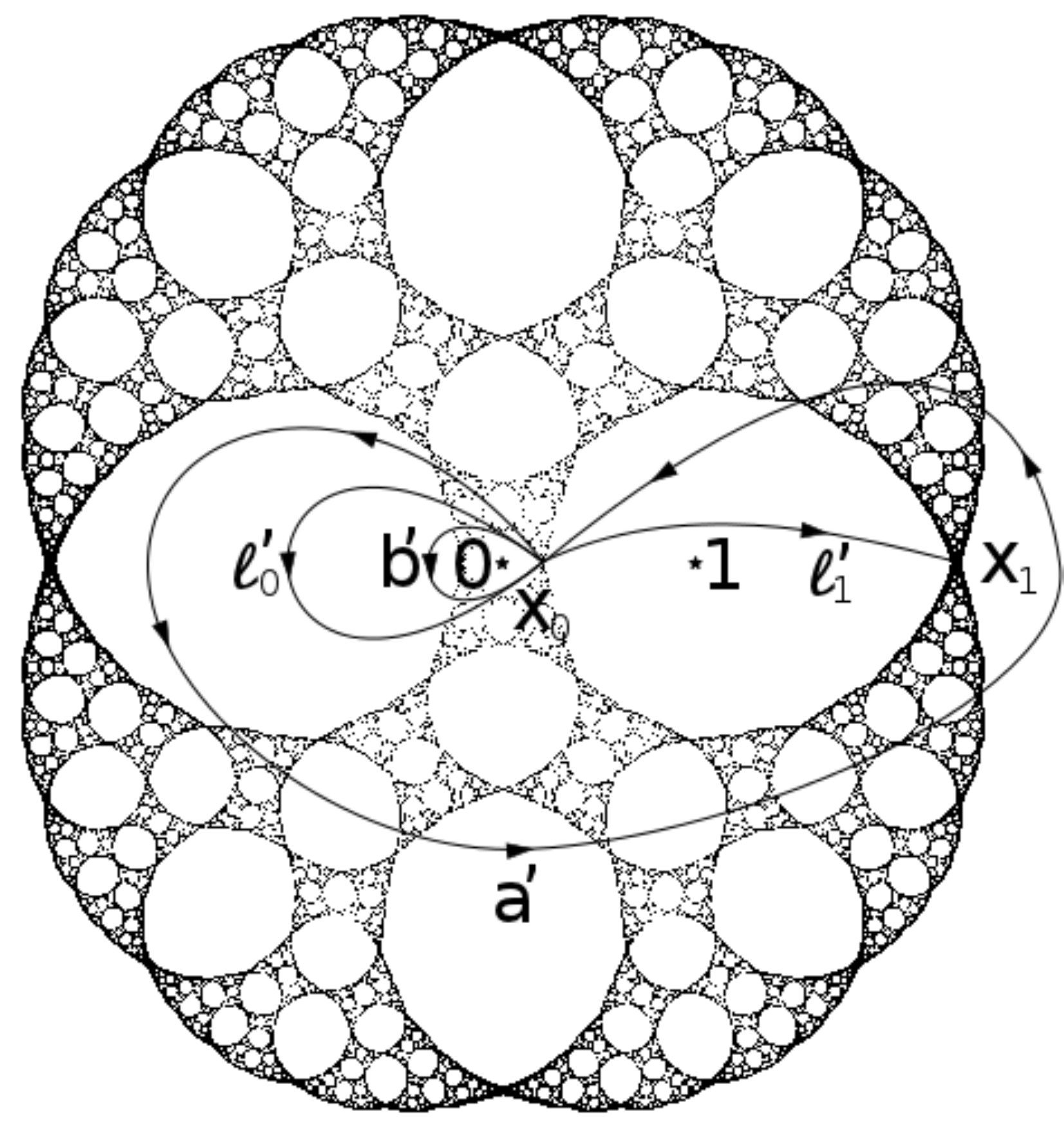}
\caption{The labeling choice $\ell'_{0},\ell'_{1}$ and two generators $[a']=[a.b],[b']=[a]$ of $\pi_{1}(\C\backslash\{0,1\},t)$}\label{FigQuadraticExample3}
\end{center}
\end{figure}

It follows from the labeling choice $(L')$ that
$$\IMG(R,t)=\Big\langle a'=(0,1)\recursion[\Id,\Id],b'=(0,1)\recursion[a',b'^{-1}]\Big\rangle$$
for which the wreath recursions of generators are nicer than for $\langle a,b\rangle$. Indeed, that immediately shows $a'^{2}=\Id$, and thus $\IMG(R,t)$ is not isomorphic to the free group on a set of two elements, or equivalently the monodromy action of $R:\C\backslash\{-1,0,1\}\rightarrow\C\backslash\{0,1\}$ is not faithful.

\newpage

\subsubsection*{Sierpinski gasket and towers of Hanoi (due to Grigorchuk and \u{S}uni\'{c} \cite{SelfSimilarity})}

Consider the cubic rational map $H:z\mapsto z^{2}-\frac{16}{27z}$ whose Julia set is a Sierpinski gasket (see Figure \ref{FigSierpinski1}). The critical points are $\infty$ and $c_{k}=-\frac{2}{3}\zeta^{k}$ where $k\in\{0,1,2\}$ and $\zeta=-\frac{1}{2}+\frac{\sqrt{3}}{2}$ (a third root of unity). Since $H(\infty)=\infty$ and $H(c_{k})=\frac{4}{3}\zeta^{2k}$, $H^{2}(c_{k})=\frac{4}{3}\zeta^{k}$, the post-critical set is $\left\{\frac{4}{3},\frac{4}{3}\zeta,\frac{4}{3}\zeta^{2},\infty\right\}$.
$$\xymatrix @!0 @R=3pc @C=4pc {
&&\frac{4}{3}\zeta \ar@/_{4pc}/[dddd]_{1:1} && \\
&&& -\frac{2}{3}\zeta^{2} \ar[ul]^{2:1} & \\
\infty \ar@(ul,dl)[]_{2:1} && -\frac{2}{3} \ar[rr]^(0.75){2:1} && \frac{4}{3} \ar@(dr,ur)[]_{1:1} \\
&&& -\frac{2}{3}\zeta \ar[dl]_{2:1} & \\
&&\frac{4}{3}\zeta^{2} \ar@/^{2pc}/[uuuu]^{1:1} &&
}$$
That induces a degree $3$ partial self-covering $H:\C\backslash H^{-1}\left(\left\{\frac{4}{3},\frac{4}{3}\zeta,\frac{4}{3}\zeta^{2}\right\}\right)\longrightarrow\C\backslash\left\{\frac{4}{3},\frac{4}{3}\zeta,\frac{4}{3}\zeta^{2}\right\}$.

\begin{figure}[!h]
\begin{center}
\includegraphics[width=8cm]{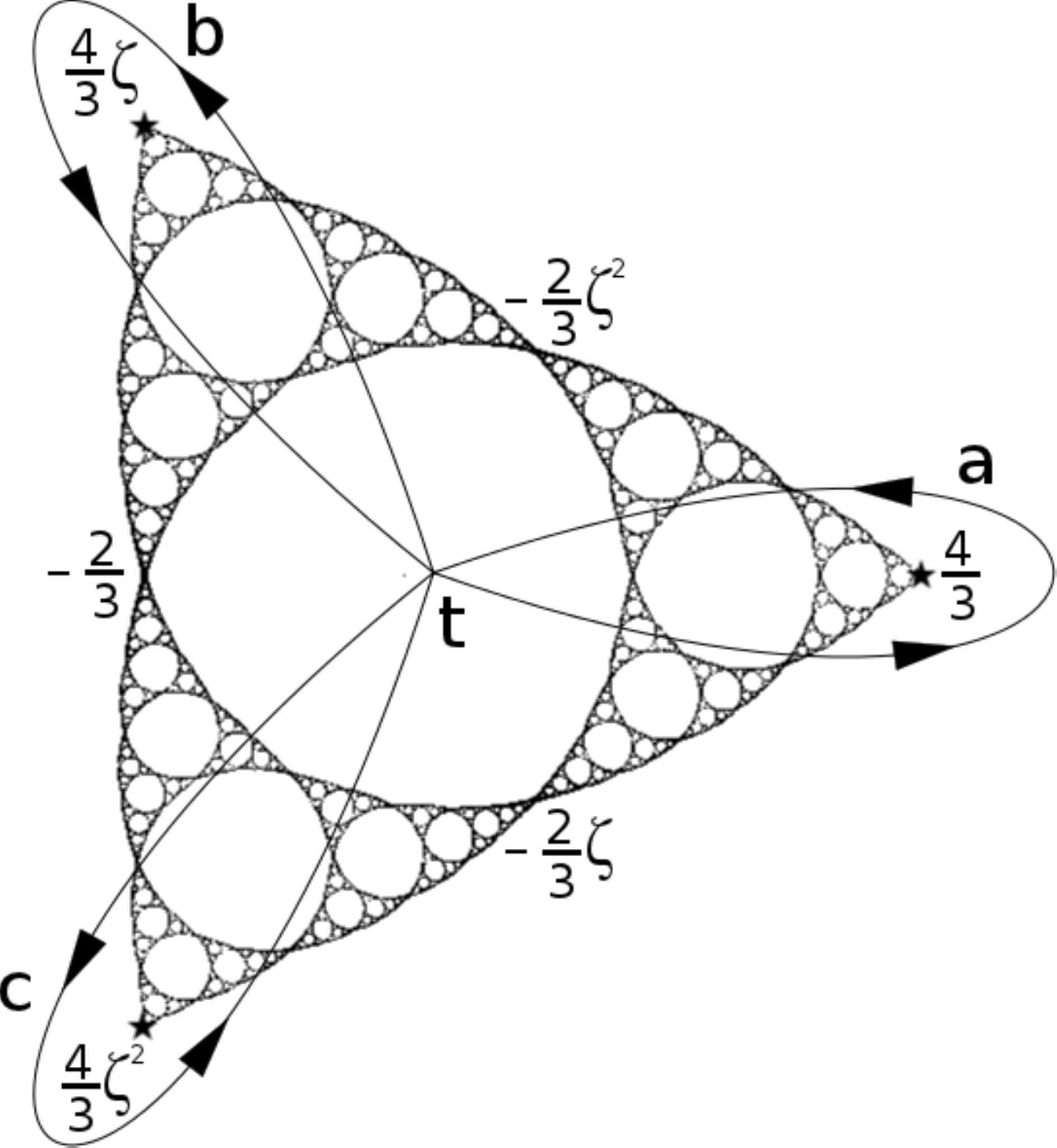}
\caption{A Sierpinski gasket and three generators of $\pi_{1}(\C\backslash\{\frac{4}{3},\frac{4}{3}\zeta,\frac{4}{3}\zeta^{2}\},t)$}\label{FigSierpinski1}
\end{center}
\end{figure}

Choose $t=0$ as base point. The fundamental group $\pi_{1}(\C\backslash\{\frac{4}{3},\frac{4}{3}\zeta,\frac{4}{3}\zeta^{2}\},t)$ may be described as the free group generated by three homotopy classes $[a],[b],[c]$ where the loops $a,b,c$ surround the post-critical points $\frac{4}{3},\frac{4}{3}\zeta,\frac{4}{3}\zeta^{2}$ respectively in a counterclockwise motion (see Figure \ref{FigSierpinski1}).

The preimages of $t$ are $x_{\varepsilon}=\frac{2^{4/3}}{3}\zeta^{\varepsilon}$ where the letter $\varepsilon$ belongs to the alphabet $\mathcal{E}=\{0,1,2\}$. For every letter $\varepsilon\in\mathcal{E}$, let $\ell_{\varepsilon}$ be the straight path (in $\C$) from $t$ to $x_{\varepsilon}$ as it is shown in Figure \ref{FigSierpinski2}. This picture also depicts the lifts of the loops $a,b,c$.

\begin{figure}[!h]
\begin{center}
\includegraphics[width=8cm]{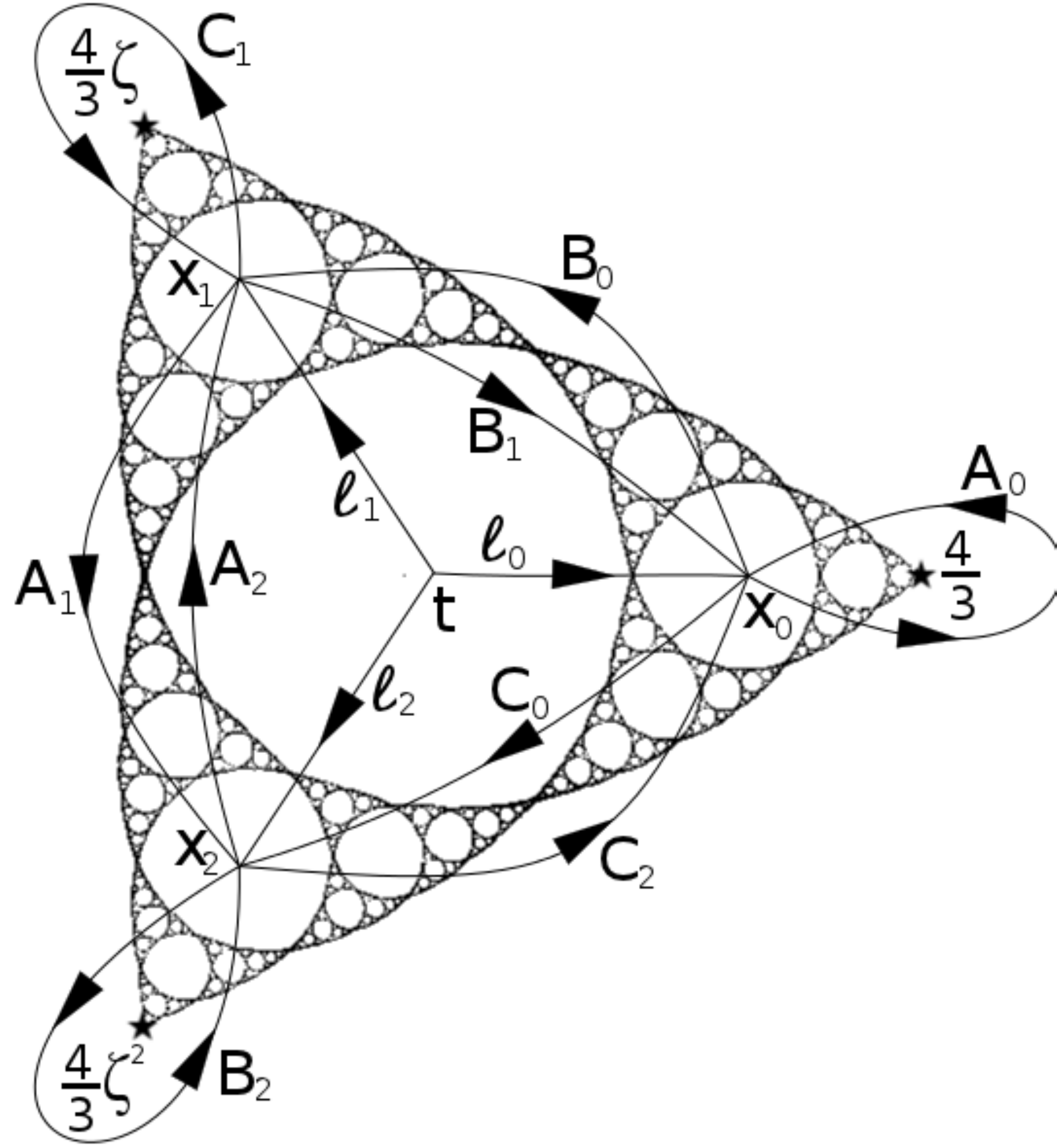}
\caption{The labeling choice $\ell_{0},\ell_{1},\ell_{2}$ and the lifts of $a,b,c$}\label{FigSierpinski2}
\end{center}
\end{figure}

One can deduce the monodromy action of $[a],[b],[c]$ on the first level.
$$\left\{\begin{array}{l} [a]0=0 \\ {[a]}1=2 \\ {[a]}2=1 \end{array}\right.\quad\text{and}\quad\left\{\begin{array}{l} [b]0=1 \\ {[b]}1=0 \\ {[b]}2=2 \end{array}\right.\quad\text{and}\quad\left\{\begin{array}{l} [c]0=2 \\ {[c]}1=1 \\ {[c]}2=0 \end{array}\right.$$
Furthermore
$$\left\{\begin{array}{l} [\ell_{0}.A_{0}.\ell_{0}^{-1}]=[a] \\ {[\ell_{1}.A_{1}.\ell_{2}^{-1}]}=[1_{t}] \\ {[\ell_{2}.A_{2}.\ell_{1}^{-1}]}=[1_{t}] \end{array}\right.\quad\text{and}\quad\left\{\begin{array}{l} [\ell_{0}.B_{0}.\ell_{1}^{-1}]=[1_{t}] \\ {[\ell_{1}.B_{1}.\ell_{0}^{-1}]}=[1_{t}] \\ {[\ell_{2}.B_{2}.\ell_{2}^{-1}]}=[c] \end{array}\right.\quad\text{and}\quad\left\{\begin{array}{l} [\ell_{0}.C_{0}.\ell_{2}^{-1}]=[1_{t}] \\ {[\ell_{1}.C_{1}.c_{1}^{-1}]}=[b] \\ {[\ell_{2}.C_{2}.\ell_{0}^{-1}]}=[1_{t}] \end{array}\right.$$
where $[1_{t}]$ is the homotopy class of the constant loop at base point $t$ (that is the identity element of the fundamental group $\pi_{1}(\C\backslash\{\frac{4}{3},\frac{4}{3}\zeta,\frac{4}{3}\zeta^{2}\},t)$). It follows from Lemma \ref{LemComputation} that
$$\forall w\in \mathcal{E}^{\star},\quad\left\{\begin{array}{l} [a]0w=0([a]w) \\ {[a]}1w=2w \\ {[a]}2w=1w \end{array}\right.\quad\text{and}\quad\left\{\begin{array}{l} [b]0w=1w \\ {[b]}1w=0w \\ {[b]}2w=2([c]w) \end{array}\right.\quad\text{and}\quad\left\{\begin{array}{l} [c]0w=2w \\ {[c]}1w=1([b]w) \\ {[c]}2w=0w \end{array}\right.$$

Therefore the iterated monodromy group of $H$ seen as a subgroup of $\Aut(T_{3})$ is generated by the following wreath recursions
$$\IMG(H,t)=\Big\langle a=(1,2)\recursion[a,\Id,\Id],b=(0,1)\recursion[\Id,\Id,c],c=(0,2)\recursion[\Id,b,\Id]\Big\rangle$$

$\IMG(H,t)$ is not isomorphic to the free group on a set of three elements (for instance one can prove that $a^{2}=b^{2}=c^{2}=\Id$ by using Lemma \ref{LemComputationsWreathRecursion} and Lemma \ref{LemIdentityTreeAutomorphism}) or equivalently the monodromy action of $H:\C\backslash H^{-1}\left(\left\{\frac{4}{3},\frac{4}{3}\zeta,\frac{4}{3}\zeta^{2}\right\}\right)\rightarrow\C\backslash\left\{\frac{4}{3},\frac{4}{3}\zeta,\frac{4}{3}\zeta^{2}\right\}$ is not faithful.

\newpage

Moreover $\IMG(H,t)$ looks like the Hanoi Towers group $\mathcal{H}$ (see Section \ref{SecTreeAutomorphism}). Indeed consider two tree automorphisms $g,h$ in $\Aut(T_{3})$ whose wreath recursions are given by
\vspace{-10pt}\\
$$g=(1,2)\recursion[h,h,h]\quad\text{and}\quad h=\recursion[g,g,g]$$
\vspace{-10pt}\\
This pair of tree automorphisms is well defined by induction on the successive levels of the regular rooted tree $T_{3}$. Now consider the following wreath recursions
\vspace{-10pt}
\begin{center}
\begin{tabular}{ccccccccc}
\hline
$a'$ & = & $g.a.g^{-1}$ & = & \parbox[c]{150pt}{$\xymatrix @!0 @R=2pc @C=4pc { 0 \ar[r]^{\textstyle h} & 0 \ar[r]^{\textstyle a} & 0 \ar[r]^{\textstyle h^{-1}} & 0 \\ 1 \ar[rd]^(0.25){\textstyle h} & 1 \ar[rd] & 1 \ar[rd]^(0.25){\textstyle h^{-1}} & 1 \\ 2 \ar[ru]_(0.25){\textstyle h} & 2 \ar[ru] & 2 \ar[ru]_(0.25){\textstyle h^{-1}} & 2}$} & = & \parbox[c]{55pt}{$\xymatrix @!0 @R=2pc @C=4pc {0 \ar[r]^{\textstyle h.a.h^{-1}} & 0 \\ 1 \ar[rd] & 1 \\ 2 \ar[ru] & 2}$} & = & $(1,2)\recursion[h.a.h^{-1},\Id,\Id]$ \\
\hline
$b'$ & = & $g.c.g^{-1}$ & = & \parbox[c]{150pt}{$\xymatrix @!0 @R=2pc @C=4pc { 0 \ar[r]^{\textstyle h} & 0 \ar@(r,l)[rdd] & 0 \ar[r]^{\textstyle h^{-1}} & 0 \\ 1 \ar[rd]^(0.25){\textstyle h} & 1 \ar[r]^(0.25){\textstyle b} & 1 \ar[rd]^(0.25){\textstyle h^{-1}} & 1 \\ 2 \ar[ru]_(0.25){\textstyle h} & 2 \ar@(r,l)[ruu] & 2 \ar[ru]_(0.25){\textstyle h^{-1}} & 2 }$} & = & \parbox[c]{55pt}{$\xymatrix @!0 @R=2pc @C=4pc {0 \ar[rd] & 0 \\ 1 \ar[ru] & 1 \\ 2 \ar[r]^{\textstyle h.b.h^{-1}} & 2}$} & = & $(0,1)\recursion[\Id,\Id,h.b.h^{-1}]$ \\
\hline
$c'$ & = & $g.b.g^{-1}$ & = & \parbox[c]{150pt}{$\xymatrix @!0 @R=2pc @C=4pc { 0 \ar[r]^{\textstyle h} & 0 \ar[rd] & 0 \ar[r]^{\textstyle h^{-1}} & 0 \\ 1 \ar[rd]^(0.25){\textstyle h} & 1 \ar[ru] & 1 \ar[rd]^(0.25){\textstyle h^{-1}} & 1 \\ 2 \ar[ru]_(0.25){\textstyle h} & 2 \ar[r]^{\textstyle c} & 2 \ar[ru]_(0.25){\textstyle h^{-1}} & 2 }$} & = & \parbox[c]{55pt}{$\xymatrix @!0 @R=2pc @C=4pc { 0 \ar@(r,l)[rdd] & 0 \\ 1 \ar[r]^{\textstyle h.c.h^{-1}} & 1 \\ 2 \ar@(r,l)[ruu] & 2 }$} & = & $(0,2)\recursion[\Id,h.c.h^{-1},\Id]$ \\
\hline
\end{tabular}
\end{center}
with
\begin{center}
\begin{tabular}{ccccccc}
\hline
$h.a.h^{-1}$ & = & \parbox[c]{150pt}{$\xymatrix @!0 @R=2pc @C=4pc { 0 \ar[r]^{\textstyle g} & 0 \ar[r]^{\textstyle a} & 0 \ar[r]^{\textstyle g^{-1}} & 0 \\ 1 \ar[r]^{\textstyle g} & 1 \ar[rd] & 1 \ar[r]^{\textstyle g^{-1}} & 1 \\ 2 \ar[r]^{\textstyle g} & 2 \ar[ru] & 2 \ar[r]^{\textstyle g^{-1}} & 2}$} & = & \parbox[c]{55pt}{$\xymatrix @!0 @R=2pc @C=4pc {0 \ar[r]^{\textstyle g.a.g^{-1}} & 0 \\ 1 \ar[rd] & 1 \\ 2 \ar[ru] & 2}$} & = & $(1,2)\recursion[a',\Id,\Id]$ \\
\hline
$h.b.h^{-1}$ & = & \parbox[c]{150pt}{$\xymatrix @!0 @R=2pc @C=4pc { 0 \ar[r]^{\textstyle g} & 0 \ar[rd] & 0 \ar[r]^{\textstyle g^{-1}} & 0 \\ 1 \ar[r]^{\textstyle g} & 1 \ar[ru] & 1 \ar[r]^{\textstyle g^{-1}} & 1 \\ 2 \ar[r]^{\textstyle g} & 2 \ar[r]^{\textstyle c} & 2 \ar[r]^{\textstyle g^{-1}} & 2 }$} & = & \parbox[c]{55pt}{$\xymatrix @!0 @R=2pc @C=4pc {0 \ar[rd] & 0 \\ 1 \ar[ru] & 1 \\ 2 \ar[r]^{\textstyle g.c.g^{-1}} & 2}$} & = & $(0,1)\recursion[\Id,\Id,b']$ \\
\hline
$h.c.h^{-1}$ & = & \parbox[c]{150pt}{$\xymatrix @!0 @R=2pc @C=4pc { 0 \ar[r]^{\textstyle g} & 0 \ar@(r,l)[rdd] & 0 \ar[r]^{\textstyle g^{-1}} & 0 \\ 1 \ar[r]^{\textstyle g} & 1 \ar[r]^(0.25){\textstyle b} & 1 \ar[r]^{\textstyle g^{-1}} & 1 \\ 2 \ar[r]^{\textstyle g} & 2 \ar@(r,l)[ruu] & 2 \ar[r]^{\textstyle g^{-1}} & 2 }$} & = & \parbox[c]{55pt}{$\xymatrix @!0 @R=2pc @C=4pc { 0 \ar@(r,l)[rdd] & 0 \\ 1 \ar[r]^{\textstyle g.b.g^{-1}} & 1 \\ 2 \ar@(r,l)[ruu] & 2 }$} & = & $(0,2)\recursion[\Id,c',\Id]$ \\
\hline
\end{tabular}
\end{center}

\noindent Therefore a quick induction gives
\vspace{-10pt}
$$\left\{\begin{array}{l} a'=(1,2)\recursion[a',\Id,\Id] \\ b'=(0,1)\recursion[\Id,\Id,b'] \\ c'=(0,2)\recursion[\Id,c',\Id]\end{array}\right.$$
\vspace{-10pt}

These wreath recursions are the three generators of the Hanoi Towers group $\mathcal{H}$ (see Section \ref{SecTreeAutomorphism}), and thus $g.\langle a,b,c\rangle.g^{-1}=\mathcal{H}$. Consequently the iterated monodromy group $\IMG(H,t)$ seen as a subgroup of $\Aut(T_{3})$ is conjugate to the Hanoi Towers group $\mathcal{H}$ by a tree automorphism of $T_{3}$.

But one can show that $g$ does not correspond to a map of the form $\varphi^{(L),(L')}=\varphi^{(L)}\circ\left(\varphi^{(L')}\right)^{-1}$ for some pair of labeling choices $(L)$ and $(L')$, or equivalently there is no labeling choice for which $\IMG(H,t)$ and $\mathcal{H}$ are equal as subgroup of $\Aut(T_{3})$.

However similar computations show that $\IMG(\overline{H},t)=\mathcal{H}$ where $\overline{H}:z\mapsto\overline{H(z)}=\overline{z}^{2}-\frac{16}{27\overline{z}}$.

\section{Some properties in holomorphic dynamics}

\subsection{Combinatorial invariance}\label{SecCombinatorialInvariance}

Let $f$ and $g$ be two post-critically finite branched coverings on the topological sphere $\S^{2}$ and denote by $P_{f}$ and $P_{g}$ their respective post-critical sets. Recall that $f$ and $g$ are said to be combinatorially equivalent (or Thurston equivalent) if there exist two orientation-preserving homeomorphisms $\psi_{0},\psi_{1}$ on $\S^{2}$ such that
\begin{description}
\item[(i)] $\psi_{0}\circ f=g\circ\psi_{1}$
\item[(ii)] $\psi_{0}(P_{f})=\psi_{1}(P_{f})=P_{g}$ and $\psi_{0}|_{P_{f}}=\psi_{1}|_{P_{f}}$
\item[(iii)] $\psi_{0}$ is isotopic to $\psi_{1}$ relatively to $P_{f}$
\end{description}

Let $t\in\S^{2}\backslash P_{f}$ be a base point. Remark that the push-forward map $(\psi_{0})_{\star}:[\gamma]\mapsto[\psi_{0}\circ\gamma]$ realizes a group isomorphism from $\pi_{1}(\S^{2}\backslash P_{f},t)$ onto $\pi_{1}(\S^{2}\backslash P_{g},\psi_{0}(t))$ since $\psi_{0}(P_{f})=P_{g}$. Applying the fundamental homomorphism theorem, one gets a group isomorphism, say $(\psi_{0})_{\star}$ again, from $\IMG(f,t)$ onto a quotient group of $\pi_{1}(\S^{2}\backslash P_{g},\psi_{0}(t))$.

\begin{prop}\label{PropCombinatorialEquivalence}
If $f$ and $g$ are two combinatorially equivalent post-critically finite branched coverings on the topological sphere $\S^{2}$, say $\psi_{0}\circ f=g\circ \psi_{1}$, then for any base point $t\in\S^{2}\backslash P_{f}$
$$(\psi_{0})_{\star}\Big(\IMG(f,t)\Big)=\IMG(g,\psi_{0}(t))$$
More precisely, for every loop $\gamma$ in $\S^{2}\backslash P_{f}$ with base point $t$, the monodromy action induced by $[\psi_{0}\circ\gamma]$ on the tree of preimages $T(g,\psi_{0}(t))$ is the same as that one induced by $[\gamma]$ on the tree of preimages $T(f,t)$ (up to conjugation by tree isomorphism).
\end{prop}
Equivalently speaking, the iterated monodromy group is an invariant with respect to combinatorial equivalence classes. But it is not a complete invariant (some additional algebraic data are required, see \cite{SelfSimilarGroups}).

\begin{proof}
At first, remark that $h=\psi_{1}^{-1}\circ\psi_{0}$ is an orientation-preserving homeomorphism isotopic to $\Id_{\S^{2}}$ relatively to $P_{f}$ such that $f\circ h=\psi_{0}^{-1}\circ g\circ\psi_{0}$. It follows that $(\psi_{0})_{\star}(\IMG(f\circ h,t))=\IMG(g,\psi_{0}(t))$ (by definition of the push-forward isomorphism $(\psi_{0})_{\star}$). Therefore one only needs to check that $\IMG(f,t)=\IMG(f\circ h,t)$, or equivalently from Definition \ref{DefIteratedMonodromyGroup}, the kernel of the monodromy action of $f\circ h$ is the same as that one of $f$ (as subgroups of $\pi_{1}(\S^{2}\backslash P_{f},t)$).

So let $\gamma$ be a loop in $\S^{2}\backslash P_{f}$ with base point $t$ such that $[\gamma]$ is in the kernel of the monodromy action of $f$, namely for every preimage $y\in\bigsqcup_{n\geqslant 0} f^{-n}(t)$ the $f^{n}$-lift of $\gamma$ from $y$ is again a loop.

Let $x_{1}$ be a preimage in $(f\circ h)^{-1}(t)$. Since $[\gamma]$ is in the kernel of the monodromy action of $f$, the $f$-lift $\Gamma_{h(x_{1})}$ of $\gamma$ from $h(x_{1})\in f^{-1}(t)$ is a loop. It easily follows that $h^{-1}\circ\Gamma_{h(x_{1})}$, which is the $(f\circ h)$-lift of $\gamma$ from $x_{1}$, is a loop as well. Furthermore remark that $h^{-1}\circ\Gamma_{h(x_{1})}$ is homotopic to $\Gamma_{h(x_{1})}$ in $\S^{2}\backslash P_{f}$ (since $h$ is isotopic to $\Id_{\S^{2}}$ relatively to $P_{f}$).

Now assume by induction that every $(f\circ h)^{n}$-lift of $\gamma$ is a loop homotopic to some $f^{n}$-lift of $\gamma$. Let $x_{n+1}$ be a preimage in $(f\circ h)^{-(n+1)}(t)$. From assumption, the $(f\circ h)^{n}$-lift of $\gamma$ from $x_{n}=(f\circ h)(x_{n+1})$, say $\widetilde{\Gamma}_{x_{n}}$, is a loop homotopic to some $f^{n}$-lift of $\gamma$, say $\Gamma_{y_{n}}$ where $y_{n}\in f^{-n}(t)$. Using the homotopy lifting property from Proposition \ref{PropLift}, it follows that the $f$-lift $\widetilde{\Gamma}_{h(x_{n+1})}$ of $\widetilde{\Gamma}_{x_{n}}$ from $h(x_{n+1})\in f^{-1}(x_{n})$ is homotopic to some $f$-lift of $\Gamma_{y_{n}}$, say $\Gamma_{y_{n+1}}$ where $y_{n+1}\in f^{-(n+1)}(t)$. Since $[\gamma]$ is in the kernel of the monodromy action of $f$, $\Gamma_{y_{n+1}}$ is a loop (as $f^{n+1}$-lift of $\gamma$) and thus $\widetilde{\Gamma}_{h(x_{n+1})}$ also. It easily follows that $h^{-1}\circ\widetilde{\Gamma}_{h(x_{n+1})}$, which is the $(f\circ h)^{(n+1)}$-lift of $\gamma$ from $x_{n+1}$, is a loop as well. Furthermore $h^{-1}\circ\widetilde{\Gamma}_{h(x_{n+1})}$ is homotopic to $\widetilde{\Gamma}_{h(x_{n+1})}$ and therefore to $\Gamma_{y_{n+1}}$.

The following graph depicts this argument in short.

$$\xymatrix{ &&&&&&&&&&& \\ && x_{n+1} \ar@{~>} `ul[ul] `dl[dl]_{\textstyle h^{-1}\circ\widetilde{\Gamma}_{h(x_{n+1})}} `dr[] [] \ar[ddd]_{\textstyle h} \ar@/^{4pc}/[dddrrrr]^{\textstyle f\circ h} &&&&&&&&& \\ &&&&&&&&&&& \\ &&&&&&&&&&& \\ && h(x_{n+1}) \ar@{~>} `ul[ul] `dl[dl]_{\textstyle \widetilde{\Gamma}_{h(x_{n+1})}} `dr[] [] \ar@{-->}[dddd] \ar[rr]^{\textstyle f}  &&&& x_{n} \ar@{~>} `ul[ul] `dl[dl]_{\textstyle \widetilde{\Gamma}_{x_{n}}} `dr[] [] \ar@{-->}[dddd] &&&&& \\ &&&&&&&&&&& \\ &&&&&&&&&&& \\ &&&&& \quad &&&&& \quad & \\ && y_{n+1} \ar@{~>} `ul[ul] `dl[dl]_{\textstyle \Gamma_{y_{n+1}}} `dr[] [] \ar[rr]^{\textstyle f} &&&& y_{n} \ar@{~>} `ul[ul] `dl[dl]_{\textstyle \Gamma_{y_{n}}} `dr[] [] \ar[rrr]^{\textstyle f^{n}} &&&&& t \ar@{~>} `ul[ul] `dl[dl]_{\textstyle \gamma} `dr[] [] \\ &&&&&&&&&&& }$$
It follows by induction that $[\gamma]$ is in the kernel of the monodromy action of $f\circ h$. Consequently
$$(\psi_{0})_{\star}(\IMG(f,t))\supset(\psi_{0})_{\star}(\IMG(f\circ h,t))=\IMG(g,\psi_{0}(t))$$
The reciprocal inclusion follows by symmetry since $(\psi_{0})_{\star}=(\psi_{1})_{\star}^{-1}$.
\end{proof}

\begin{example}[Example - Classification of quadratic branched coverings with three post-critical points]\quad\\
Although the iterated monodromy group is not a complete invariant, it may be used to characterize some combinatorial equivalence classes. For instance, consider degree 2 branched coverings on the topological sphere $\S^{2}$ whose post-critical sets contain exactly three points. Such branched covering has exactly two simple critical points. By a quick exhaustion, there are exactly four ramification portraits with two simple critical points and three post-critical points, which are as follows
\begin{center}
\begin{tabular}{|c|c|c|c|}
\hline
\parbox[c]{45pt}{\vspace{10pt}$\xymatrix{ \star \ar@(dr,ur)[]_{2:1} & \\ \star \ar@/_{1pc}/[r]_{2:1} & \star \ar@/_{1pc}/[l]_{1:1} }$} & \parbox[c]{110pt}{$\xymatrix{ \star \ar@(dr,ur)[]_{2:1} &&  \\ \ar[r]^{2:1} & \star \ar[r]^{1:1} & \star \ar@(dr,ur)[]_{1:1} }$} & \parbox[c]{110pt}{$\xymatrix{ \ar[r]^{2:1} & \star \ar[r]^{1:1} & \star \ar@/_{1pc}/[r]_{2:1} & \star \ar@/_{1pc}/[l]_{1:1} }$} & \parbox[c]{80pt}{$\xymatrix{ \star \ar@/_{1pc}/[r]_{2:1} & \star \ar@/_{1pc}/[r]_{2:1} & \star \ar@/_{1pc}/[ll]_{1:1} }$} \\
\hline
\end{tabular}
\end{center}
According to Thurston topological characterization of post-critically finite rational maps \cite{ThurstonProof}, any combinatorial equivalence class of quadratic branched coverings with exactly three post-critical points contains a unique rational map up to conjugation by a Möbius map. Some easy computations show that each of the ramification portraits above corresponds to exactly one quadratic rational map up to conjugation by a Möbius map. These rational models are
$$Q_{-1}:z\mapsto z^{2}-1\quad\text{,}\quad C_{2}:z\mapsto 2z^{2}-1\quad\text{,}\quad R:z\mapsto\left(\frac{z-1}{z+1}\right)^{2}\quad\text{and}\quad F:z\mapsto 1-\frac{1}{z^{2}}$$
The first three of them have already studied in Section \ref{SecExamples}. The rational map $F$ (and $R$ too) was studied in $\cite{BartholdiNekrashevych}$ where the authors proved that the corresponding iterated monodromy group (seen as a subgroup of $\Aut(T_{2})$, for some base point $t\in\C\backslash\{0,1\}$ and some labeling choice) is generated by the following wreath recursions
$$\IMG(F,t)=\Big\langle a=(0,1)\recursion[\Id,a^{-1}.b^{-1}],b=\recursion[a,\Id]\Big\rangle$$
It follows that each row of the following table corresponds to exactly one combinatorial equivalence class.

\hspace{-10pt}
\begin{tabular}{|c|c|c|c|}
\hline
ramification portrait & rational model & generators of $\IMG$ & Julia set \\
\hline
\hline
\parbox[c]{45pt}{\vspace{10pt}$\xymatrix{ \star \ar@(dr,ur)[]_{2:1} & \\ \star \ar@/_{1pc}/[r]_{2:1} & \star \ar@/_{1pc}/[l]_{1:1} }$} & $Q_{-1}:z\mapsto z^{2}-1$ & $\begin{array}{l} a=(0,1)\recursion[b,\Id] \\ b=\hspace{25pt}\recursion[a,\Id] \end{array}$ & \parbox[c]{120pt}{\includegraphics[width=120pt]{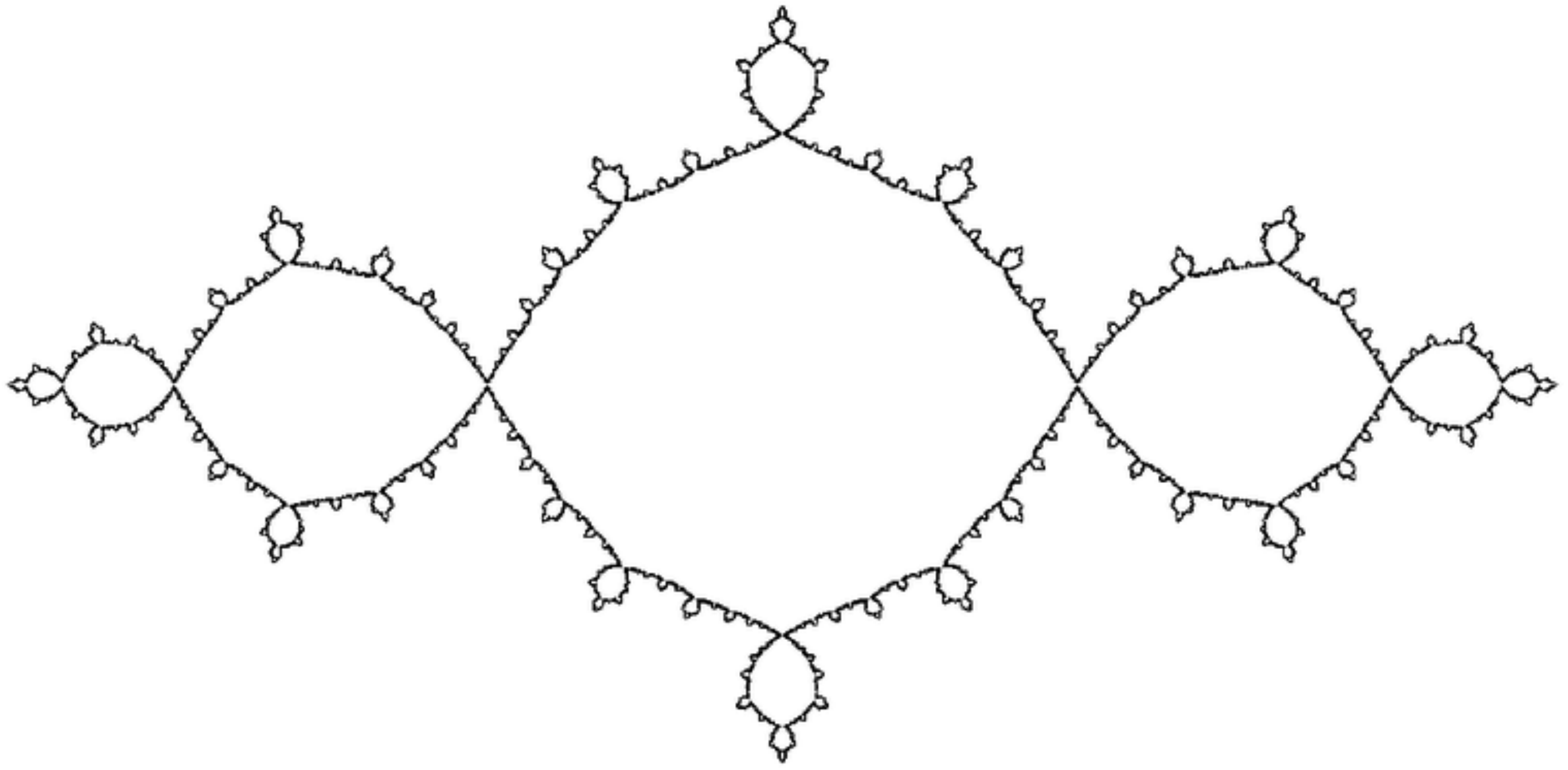}} \\
\hline
\parbox[c]{110pt}{\vspace{10pt}$\xymatrix{ \star \ar@(dr,ur)[]_{2:1} &&  \\ \ar[r]^{2:1} & \star \ar[r]^{1:1} & \star \ar@(dr,ur)[]_{1:1} }$\vspace{10pt}} & $\begin{array}{c} C_{2}:z\mapsto 2z^{2}-1 \\ \text{(or }Q_{-2}:z\mapsto z^{2}-2\text{)} \end{array}$ & $\begin{array}{l} a=(0,1)\recursion[\Id,\Id] \\ b=\hspace{25pt}\recursion[b,a] \end{array}$ & \parbox[c]{120pt}{\includegraphics[width=120pt]{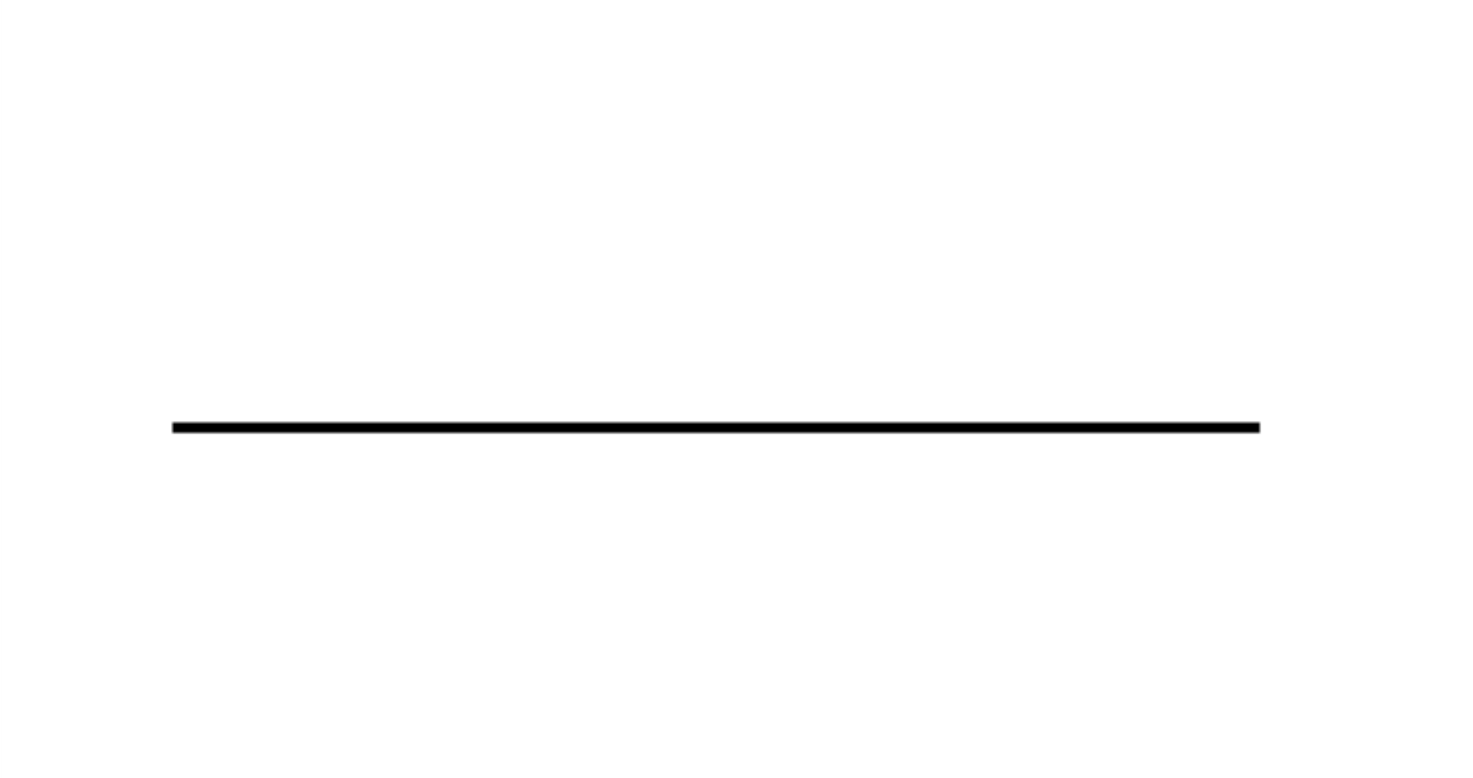}} \\ 
\hline
 \parbox[c]{110pt}{$\xymatrix{ \ar[r]^{2:1} & \star \ar[r]^{1:1} & \star \ar@/_{1pc}/[r]_{2:1} & \star \ar@/_{1pc}/[l]_{1:1} }$} & $R:z\mapsto\left(\dfrac{z-1}{z+1}\right)^{2}$ & $\begin{array}{l} a=(0,1)\recursion[\Id,\Id] \\ b=(0,1)\recursion[a,b^{-1}] \end{array}$ & \parbox[c]{120pt}{\vspace{5pt}\hspace{10pt}\includegraphics[width=100pt]{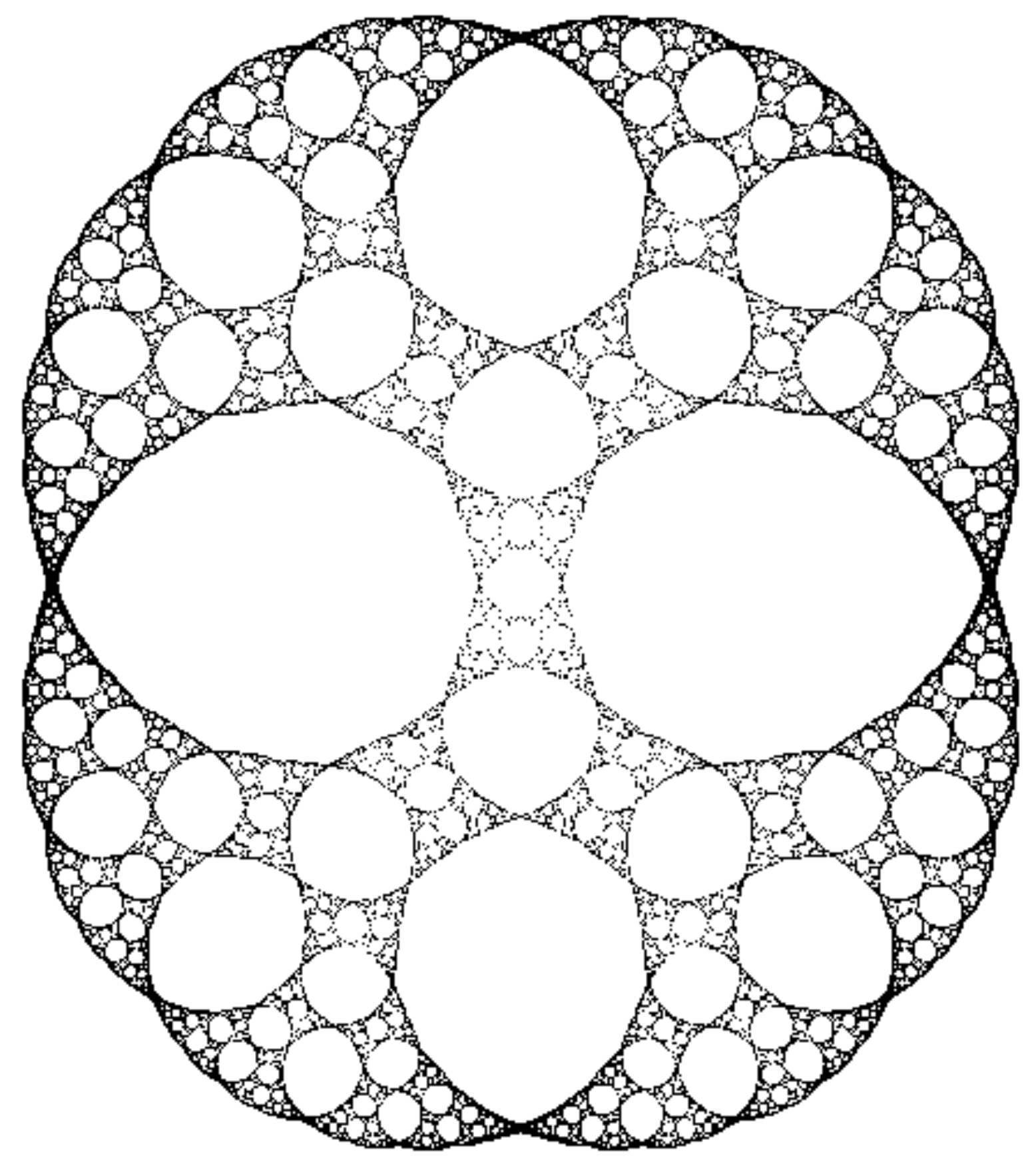}\vspace{5pt}} \\
\hline
\parbox[c]{80pt}{$\xymatrix{ \star \ar@/_{1pc}/[r]_{2:1} & \star \ar@/_{1pc}/[r]_{2:1} & \star \ar@/_{1pc}/[ll]_{1:1} }$} & $F:z\mapsto 1-\dfrac{1}{z^{2}}$ & $\begin{array}{l} a=(0,1)\recursion[\Id,a^{-1}.b^{-1}] \\ b=\hspace{25pt}\recursion[a,\Id] \end{array}$ & \parbox[c]{120pt}{\vspace{5pt}\includegraphics[width=120pt]{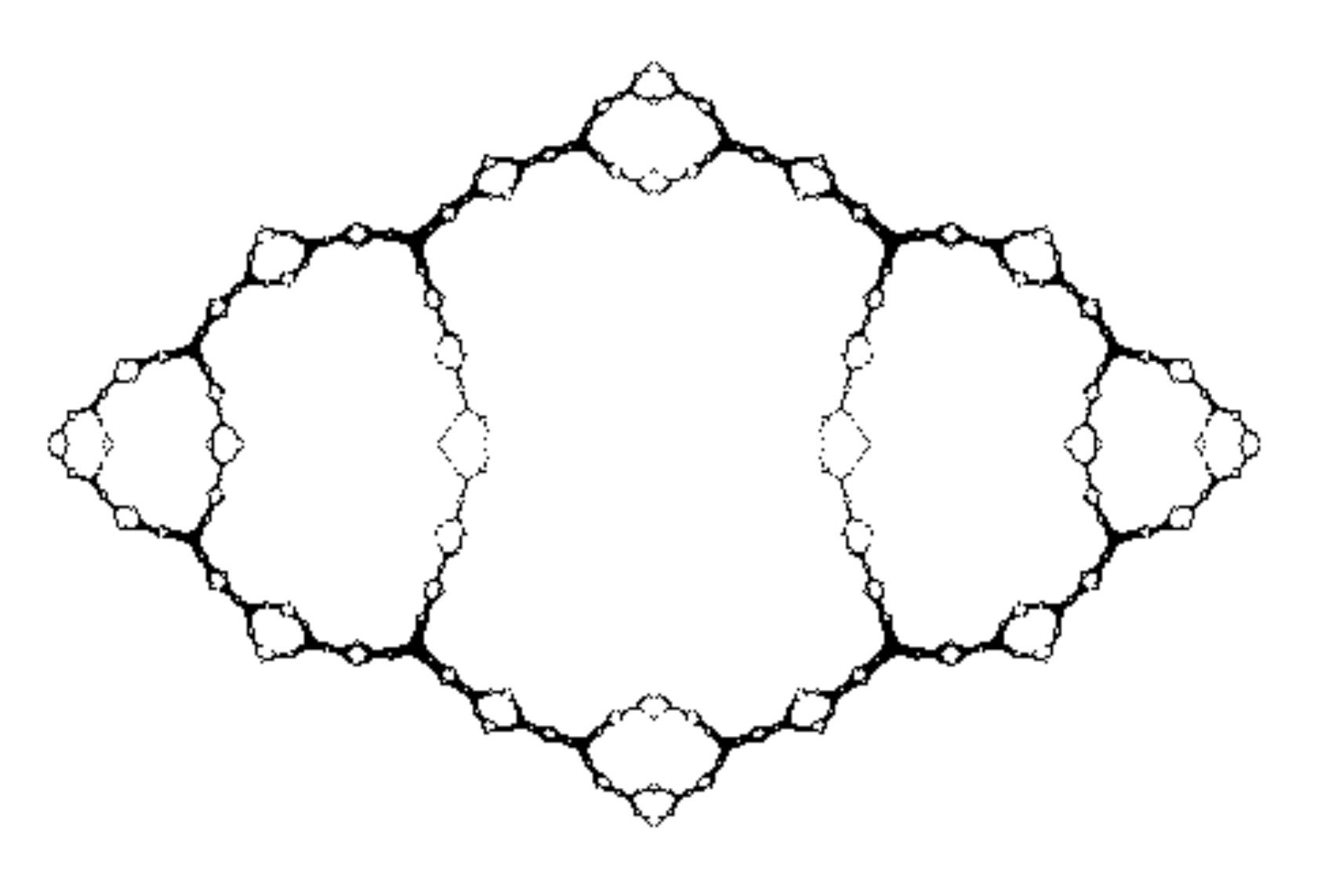}} \\
\hline
\end{tabular}

Since these iterated monodromy groups are pairwise not isomorphic (for instance, see \cite{ClassificationGroups3state2letter}), they entirely characterize the four combinatorial equivalence classes of quadratic branched coverings with exactly three post-critical points.
\end{example}

\begin{example}[Example - Hubbard Twisted Rabbit Problem (due to Bartholdi and Nekrashevych \cite{BartholdiNekrashevych})]\quad\\
In the previous example, iterated monodromy groups are redundant since the four combinatorial equivalence classes are actually characterized by their ramification portraits. However there exist combinatorial equivalence classes with same ramification portrait which are distinguished by their iterated monodromy groups. The first example is due to Pilgrim \cite{DessinsDEnfants}. Another example is provided by the solution of the Hubbard Twisted Rabbit Problem given by Bartholdi and Nekrashevych \cite{BartholdiNekrashevych} and for which Proposition \ref{PropCombinatorialEquivalence} is essential.

\newpage

Namely consider a quadratic polynomial $Q_{c}:z\mapsto z^{2}+c$ where $c\in\C$ is chosen in order that the critical point 0 is on a periodic orbit of period 3. There are precisely three such parameters $c$ which are denoted by $c_{airplane}\approx -1.755$, $c_{rabbit}\approx-0.123+0.745i$ and $c_{corabbit}\approx-0.123-0.745i$. 
\vspace{-7pt}
$$\xymatrix{ \infty \ar@(dr,ur)[]_{2:1} && 0 \ar@/_{1pc}/[r]_{2:1} & c \ar@/_{1pc}/[r]_{1:1} & c^{2}+c \ar@/_{1pc}/[ll]_{1:1} }$$
\vspace{-10pt}
\begin{center}
\begin{tabular}{|c|c|c|}
\hline
airplane & rabbit & corabbit \\
\hline
\parbox[c]{120pt}{\includegraphics[width=120pt]{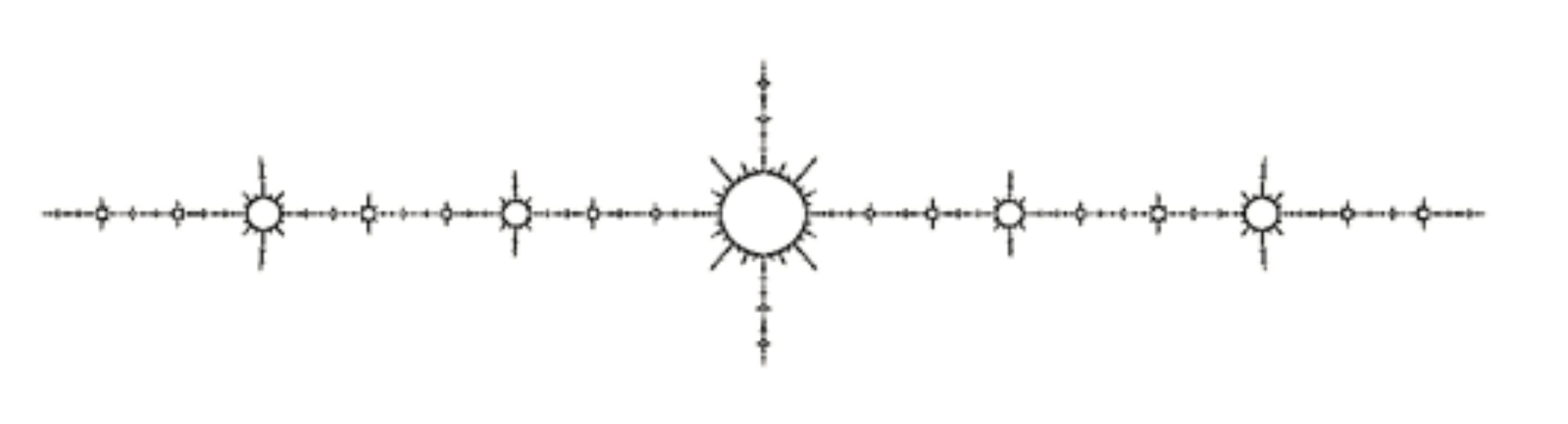}} & \parbox[c]{120pt}{\vspace{5pt}\includegraphics[width=120pt]{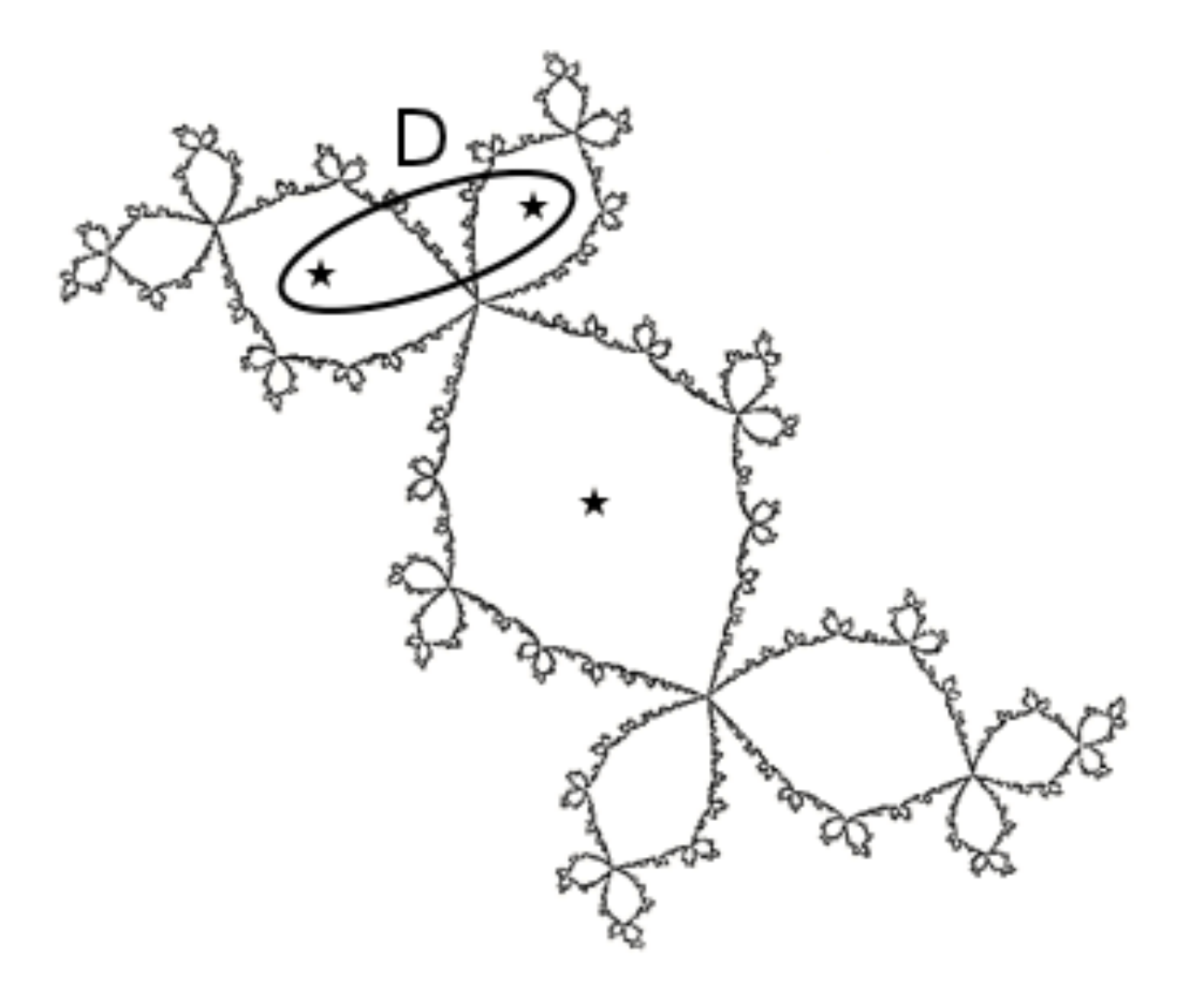}} & \parbox[c]{120pt}{\vspace{5pt}\includegraphics[width=120pt]{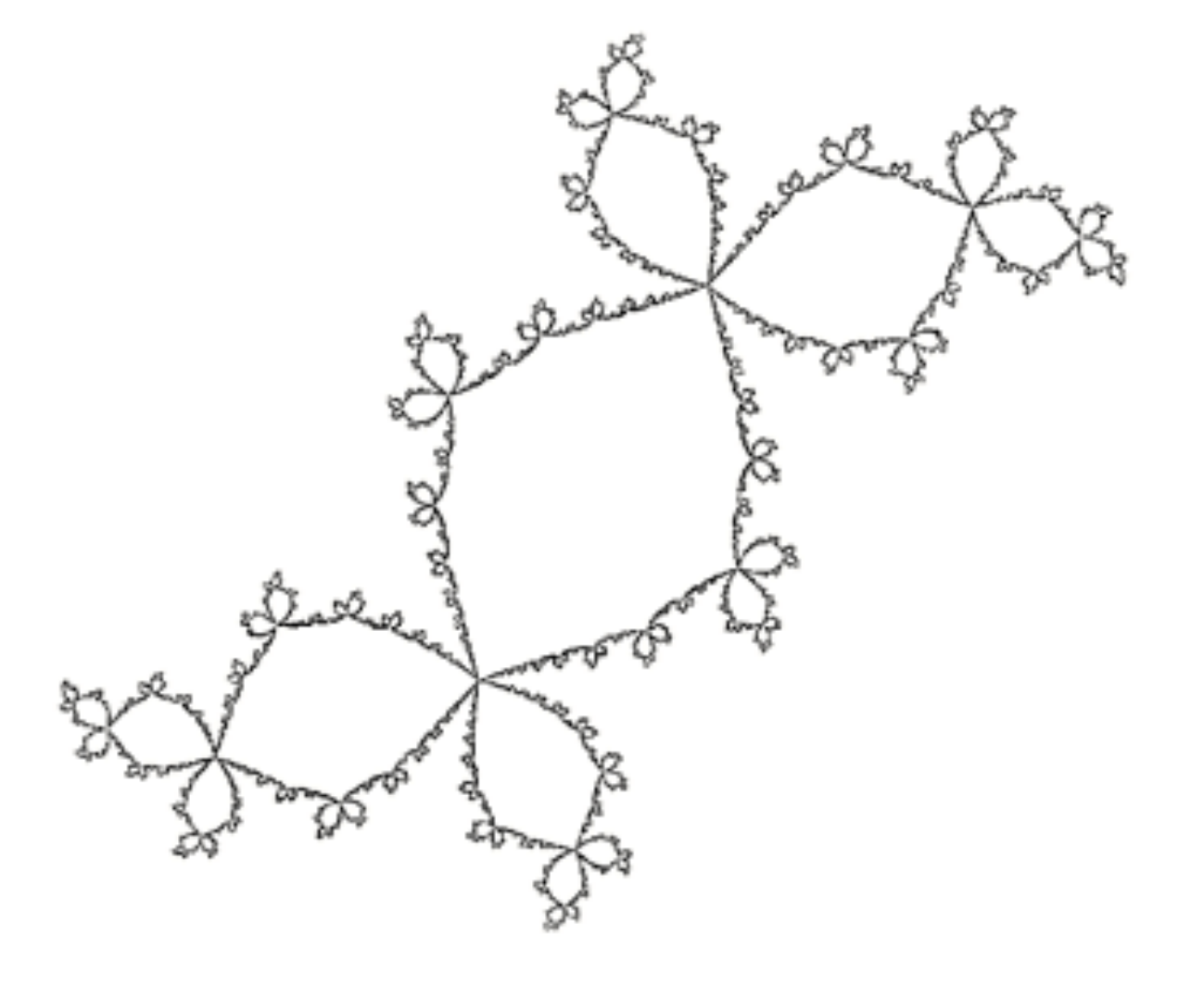}}\\
\hline
\end{tabular}
\end{center}
\vspace{-7pt}
Consider now a Dehn twist $D$ around the two nonzero post-critical points of $Q_{c_{rabbit}}$. Then for every integer $m\in\Z$, the map $D^{m}\circ Q_{c_{rabbit}}$ is again a branched covering with the same post-critical set (and same ramification portrait) as $Q_{c_{rabbit}}$. According to Thurston topological characterization of post-critically finite rational maps \cite{ThurstonProof}, the map $D^{m}\circ Q_{c_{rabbit}}$ is combinatorially equivalent to one of $Q_{c_{airplane}}$, $Q_{c_{rabbit}}$, $Q_{c_{corabbit}}$. The question asked by Hubbard \cite{ThurstonProof} is which one ?
A sketch of the solution given by Bartholdi and Nekrashevych is to compute the iterated monodromy group of $D^{m}\circ Q_{c_{rabbit}}$ and to compare it with those ones of $Q_{c_{airplane}}$, $Q_{c_{rabbit}}$ and $Q_{c_{corabbit}}$. 
\end{example}

\begin{example}[Example - Characterization of Thurston obstructions]\quad\\
Proposition \ref{PropCombinatorialEquivalence} suggests that the topological criterion from Thurston characterization of post-critically finite rational maps (see \cite{ThurstonProof}) may be algebraically reformulated in terms of the associated iterated monodromy group. Actually it was done in \cite{HurwitzLikeClassification} (see also \cite{AlgebraicThurstonEquivalence}). To illustrate this, one can easily show how to check that a certain kind of Thurston obstructions, namely Levy cycles (see \cite{LevyThesis}), does not occur by using the wreath recursions of the associated iterated monodromy group.

Let $f$ be a post-critically finite branched covering on the topological sphere $\S^{2}$ of degree $d\geqslant 2$ and denote by $P_{f}$ its post-critical set. Recall that a multi-curve is a finite set $\Gamma=\{\gamma_{1},\gamma_{2},\dots,\gamma_{m}\}$ of $m\geqslant 1$ disjoint Jordan curves in $\S^{2}\backslash P_{f}$ which are non-isotopic and non-peripheral (namely each connected component of $\S^{2}\backslash\gamma_{k}$ contains at least two points of $P_{f}$ for every $k\in\{1,2,\dots,m\}$). Also recall that a multi-curve $\Gamma=\{\gamma_{1},\gamma_{2},\dots,\gamma_{m}\}$ is called a Levy cycle if for every $k\in\{1,2,\dots,m\}$, $f^{-1}(\gamma_{k})$ has a connected component $\delta_{k-1}$ isotopic to $\gamma_{k-1}$ relatively to $P_{f}$ (with notation $\gamma_{0}=\gamma_{m}$) and the restriction $f|_{\delta_{k-1}}:\delta_{k-1}\rightarrow\gamma_{k}$ is of degree one.

Up to isotopy, all the loops in a multi-curve $\Gamma=\{\gamma_{1},\gamma_{2},\dots,\gamma_{m}\}$ may be assumed to have a common base point $t\in\S^{2}\backslash P_{f}$. Using the monodromy action, every loop in $\Gamma$ induces a tree automorphism of $T_{d}$ (for any given labeling choice) which may be uniquely written as wreath recursion as follows (see Definition \ref{DefWreathRecursion})
$$\left\{\begin{array}{rcl} \gamma_{1} & = & \sigma_{\gamma_{1}}\recursion[\gamma_{1,0},\gamma_{1,1},\dots,\gamma_{1,d-1}] \\ \gamma_{2} & = & \sigma_{\gamma_{2}}\recursion[\gamma_{2,0},\gamma_{2,1},\dots,\gamma_{2,d-1}] \\ & \dots & \\ \gamma_{m} & = & \sigma_{\gamma_{m}}\recursion[\gamma_{m,0},\gamma_{m,1},\dots,\gamma_{m,d-1}] \end{array}\right.$$
It follows from Lemma \ref{LemComputation} that if the multi-curve $\Gamma=\{\gamma_{1},\gamma_{2},\dots,\gamma_{m}\}$ is a Levy cycle then for every $k\in\{1,2,\dots,m\}$, there exists a letter $\varepsilon_{k}\in\mathcal{E}=\{0,1,\dots,d-1\}$ such that $\sigma_{\gamma_{k}}(\varepsilon_{k})=\varepsilon_{k}$ and $\gamma_{k,\varepsilon_{k}},\gamma_{k-1}$ (with notation $\gamma_{0}=\gamma_{m}$) are two tree automorphisms of $T_{d}$ induced by two loops which are isotopic relatively to $P_{f}$. Although this algebraic necessary condition is not a sufficient condition since the monodromy action is in general not faithful (two non-homotopic loops may induce the same tree automorphism in $\Aut(T_{d})$), it may be used in order to prove that a combinatorial equivalence class does not contain some Levy cycles (according to Proposition \ref{PropCombinatorialEquivalence}).
\end{example}

\subsection{Matings}\label{SecMatings}

Let $f_{1}$ and $f_{2}$ be two monic polynomials on the complex plane $\C$ of same degree $d\geqslant 2$. Let $\C_{1}$ and $\C_{2}$ be two copies of the complex plane and let $f_{1}$ and $f_{2}$ act on the corresponding copy. Compactify each copy by adding circles at infinity, that is
$$\widetilde{\C_{1}}=\C_{1}\cup\{\infty\cdot e^{2i\pi\theta}\ /\ \theta\in\R/\Z\}\quad\text{and}\quad\widetilde{\C_{2}}=\C_{2}\cup\{\infty\cdot e^{2i\pi\theta}\ /\ \theta\in\R/\Z\}$$
The actions of $f_{1}$ and $f_{2}$ are continuously extended to the action $\infty\cdot e^{2i\pi\theta}\mapsto\infty\cdot e^{2i\pi d\theta}$ on the circle at infinity. Now glue the copies $\widetilde{\C_{1}}$ and $\widetilde{\C_{2}}$ along the circle at infinity in the opposite direction, namely 
$$\C_{1}\Mating\C_{2}=\widetilde{\C_{1}}\sqcup\widetilde{\C_{2}}/(\infty\cdot e^{2i\pi\theta}\in\widetilde{\C_{1}})\sim(\infty\cdot e^{-2i\pi\theta}\in\widetilde{\C_{2}})$$
This makes $\C_{1}\Mating\C_{2}$ a topological sphere and that provides a degree $d$ branched covering $f_{1}\Mating f_{2}$ on $\C_{1}\Mating\C_{2}$ whose restrictions on the hemispheres $\C_{1}$ and $\C_{2}$ are equal to $f_{1}$ and $f_{2}$ respectively. $f_{1}\Mating f_{2}$ is called the formal mating of $f_{1}$ and $f_{2}$. Furthermore, $f_{1}$ and $f_{2}$ are said to be matable if the formal mating $f_{1}\Mating f_{2}$ is combinatorially equivalent to a rational map (see \cite{MatingQuadraticPolynomials}).

If $f_{1}$ and $f_{2}$ are post-critically finite, then $f_{1}\Mating f_{2}$ is post-critically finite as well. In that case one can consider the iterated monodromy group of partial self-covering induced by $f_{1}\Mating f_{2}$ (or that one induced by the corresponding rational map in case $f_{1}$ and $f_{2}$ are matable) and compare it with those ones induced by $f_{1}$ and $f_{2}$ respectively. The following result is an easy consequence of Definition \ref{DefIteratedMonodromyGroup} and construction of the formal mating.

\begin{prop}
The iterated monodromy group of a formal mating $f_{1}\Mating f_{2}$ is generated by two subgroups which are isomorphic to the iterated monodromy groups of $f_{1}$ and $f_{2}$ respectively.
\end{prop}

Remark this provides a sufficient condition to prove that a rational map is not combinatorially equivalent to a formal mating (according to Proposition \ref{PropCombinatorialEquivalence}). Unfortunately this criterion (the iterated monodromy group of a given post-critically finite rational map of degree $d\geqslant 2$ is not generated by some pair of iterated monodromy groups of degree $d$ monic polynomials) is too hard to check.

The following result is more useful with this aim in view.

\begin{prop}\label{PropAddingMachineMating}
The iterated monodromy group of a formal mating $f_{1}\Mating f_{2}$ of degree $d\geqslant 2$ contains a tree automorphism which acts as a cyclic permutation of order $d^{n}$ on the $n$-th level for every $n\geqslant 1$.
\end{prop}

In order to prove that a rational map $f$ is not combinatorially equivalent to a formal mating, that provides finitely many conditions to check for every level $n\geqslant 1$. Indeed one only need to show that the finite subgroup of permutations generated by the monodromy actions on the $n$-th level induced by the finitely many generators of $\pi_{1}(\CC\backslash P_{f},t)$ (for any base point $t\in\CC\backslash P_{f}$) does not contain a cyclic permutation. Unfortunately this criterion is not an equivalence as it will be shown in the next example.

\begin{proof}
Without loss of generality, choose the base point $t$ to be the fixed point $\infty.e^{2i\pi 0}$ on the circle at infinity of the copies $\widetilde{\C_{1}}$ and $\widetilde{\C_{2}}$. Consider the loop $\gamma$ which describes the circle at infinity (in one turn from $t$ to $t$). Remark that the $(f_{1}\Mating f_{2})$-lifts of $\gamma$ are the arcs of circle at infinity between two consecutive preimages of $t$. For $\varepsilon$ describing the alphabet $\mathcal{E}=\{0,1,\dots,d-1\}$, denote successively by $x_{\varepsilon}$ these consecutive preimages and by $\ell_{\varepsilon}$ the arcs of circle at infinity between $t$ and $x_{\varepsilon}$ (going along the circle at infinity in the same direction as $\gamma$). By using Lemma \ref{LemComputation}, one can deduce the monodromy action of $[\gamma]$
$$\forall w\in \mathcal{E}^{\star},\quad\left\{\begin{array}{l} [\gamma]0w=1w \\ {[\gamma]}1w=2w \\ \dots \\ {[\gamma]}(d-2)w=(d-1)w \\ {[\gamma]}(d-1)w=0([\gamma]w) \end{array}\right.$$
Consequently, the iterated monodromy group $\IMG(f,t)$ seen as a subgroup of $\Aut(T_{d})$ contains the wreath recursion $\gamma=\sigma\recursion[\Id,\Id,\dots,\Id,\gamma]$ where $\sigma=(0,1,\dots,d-1)$ (using circular notation) which defines the adding machine. In particular, $[\gamma]$ acts as a cyclic permutation of order $d^{n}$ on the $n$-th level for every $n\geqslant 1$.
\end{proof}

\vspace{-10pt}

\begin{example}[Example - The non-mating Wittner example (due to Milnor and Tan Lei \cite{RemarksQuadraticRationalMaps})]\quad\\
Consider the following quadratic rational map
\vspace{-5pt}
$$W:z\mapsto\lambda\left(z+\dfrac{1}{z}\right)+\mu$$
where the parameters $\lambda,\mu\in\C$ are chosen in order that the critical point $c_{0}=1$ is on a periodic orbit of period 4 and the critical point $c'_{0}=-1$ is on a periodic orbit of period 3. Such parameters are actually unique and computation gives $\lambda\approx-0.138$ and $\mu\approx-0.303$. The following ramification portrait depicts the pattern of the critical orbits along the real line $\R$.
\vspace{-5pt}
$$\xymatrix{ \ar@{-}[rr] && c_{3} \ar@{-}[r] \ar@/_{4pc}/[rrrrrr]_{1:1} & c'_{0} \ar@{-}[r] \ar@/_{2pc}/[rr]_{2:1} & c_{1} \ar@{-}[r] \ar@/_{2pc}/[rrr]_(0.33){1:1} & c'_{1} \ar@{-}[r] \ar@/_{3pc}/[rrrr]_(0.66){1:1} & 0 \ar@{-}[r] & c_{2} \ar@{-}[r] \ar@/_{3pc}/[lllll]_(0.66){1:1} & c_{0} \ar@{-}[r] \ar@/_{2pc}/[llll]_(0.33){2:1} & c'_{2} \ar@{->}[rrr]^(0.75){\textstyle \R} \ar@/_{4pc}/[llllll]_{1:1} &&& }$$

Wittner proved that $W$ is not combinatorially equivalent to a formal mating (see \cite{RemarksQuadraticRationalMaps}). Proposition \ref{PropAddingMachineMating} suggests a different proof of this result by using iterated monodromy groups. Unfortunately one will show that there exists a tree automorphism in the iterated monodromy group of $W$ which acts as a cyclic permutation on every level although $W$ is not combinatorially equivalent to a formal mating. However one will see that iterated monodromy groups provide an efficient way to prove that $W^{5}$ is combinatorially equivalent to a formal mating.

For convenience, conjugate $W$ by the Möbius map $z\mapsto\frac{z+i}{iz+1}$ which interchanges the extended real line $\R\cup\{\infty\}$ and the unit circle $\S^{1}$ (keeping the critical points $1$ and $-1$ fixed). Abusing notation, the resulting map is still denoted by $W$ and the post-critical points, which belong to the unit circle $\S^{1}$, are still denoted by $c_{k}=W^{k}(1)$ and $c'_{k}=W^{k}(-1)$ for every integer $k$. Choose the fixed point $t=-i$ as base point (which corresponds to $\infty$ in the first model).

The fundamental group $\pi_{1}(\CC\backslash P_{W},t)$ may be described as the free group generated by six homotopy classes among $[a_{0}],[a_{1}],[a_{2}],[a_{3}]$ and $[b_{0}],[b_{1}],[b_{2}]$ where every loop $a_{k}$ (respectively $b_{k}$) surrounds the post-critical point $c_{k}$ (respectively $c'_{k}$) in a counterclockwise motion (see Figure \ref{FigWittner1}). In fact, these seven loops together are linked by a circular relation of the form $[b_{2}.a_{0}.a_{2}.b_{1}.a_{1}.b_{0}.a_{3}]=[1_{t}]$ where $[1_{t}]$ is the homotopy class of the constant loop at base point $t$ (that is the identity element of the fundamental group $\pi_{1}(\CC\backslash P_{W},t)$).

Let $x_{0}=-i,x_{1}=i$ be the preimages of $t=-i$. Choose $\ell_{0}$ to be the constant path at base point $t=x_{0}$ and $\ell_{1}$ to be the straight path from $t$ to $x_{1}$. Figure \ref{FigWittner2} depicts the $W$-lift of Figure $\ref{FigWittner1}$.

\newpage

\begin{figure}[!h]
\begin{center}
\includegraphics[height=6.5cm]{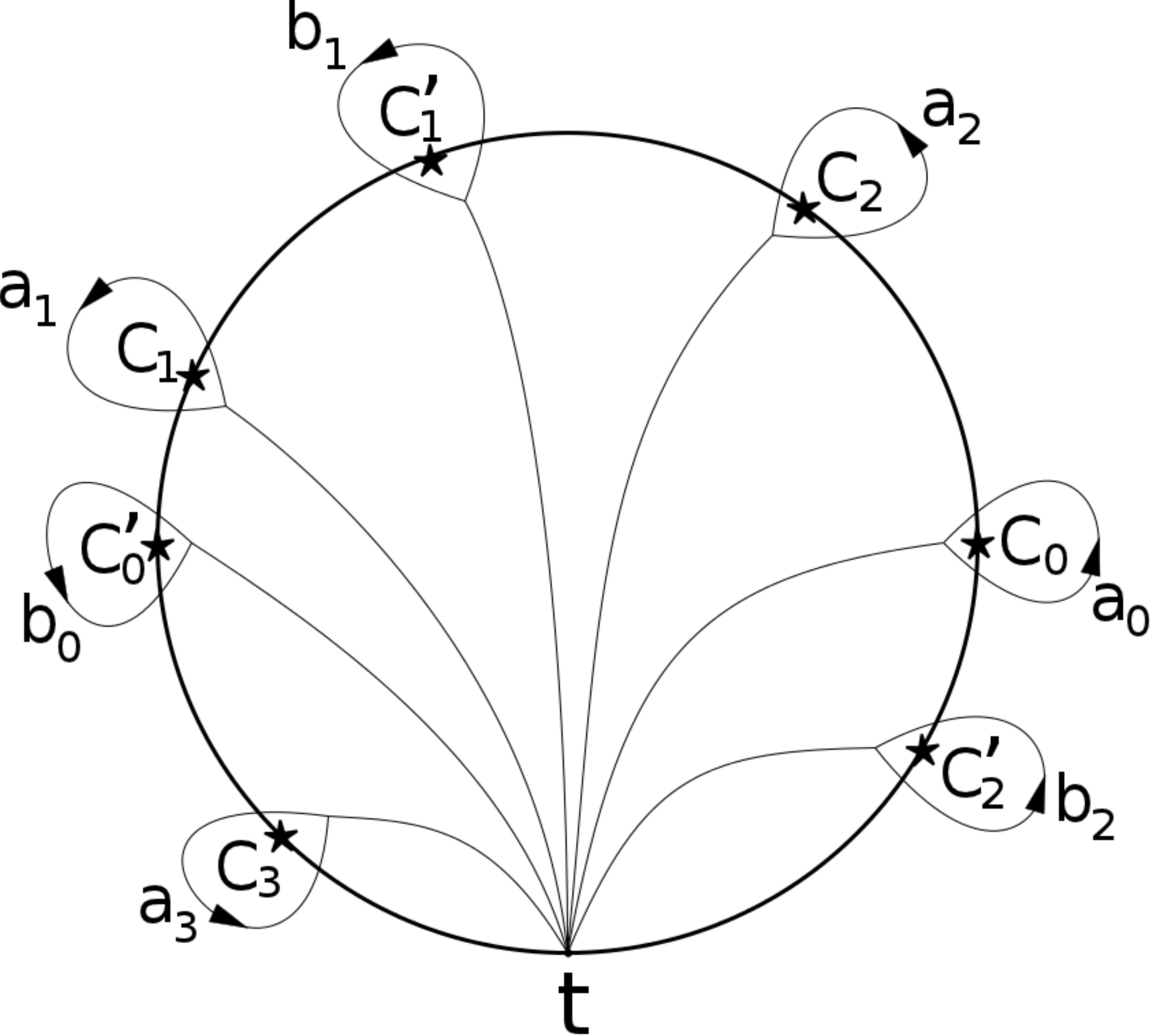}
\caption{Seven generators of $\pi_{1}(\CC\backslash P_{W},t)$}\label{FigWittner1}
\end{center}
\end{figure}

\begin{figure}[!h]
\begin{center}
\includegraphics[height=6.5cm]{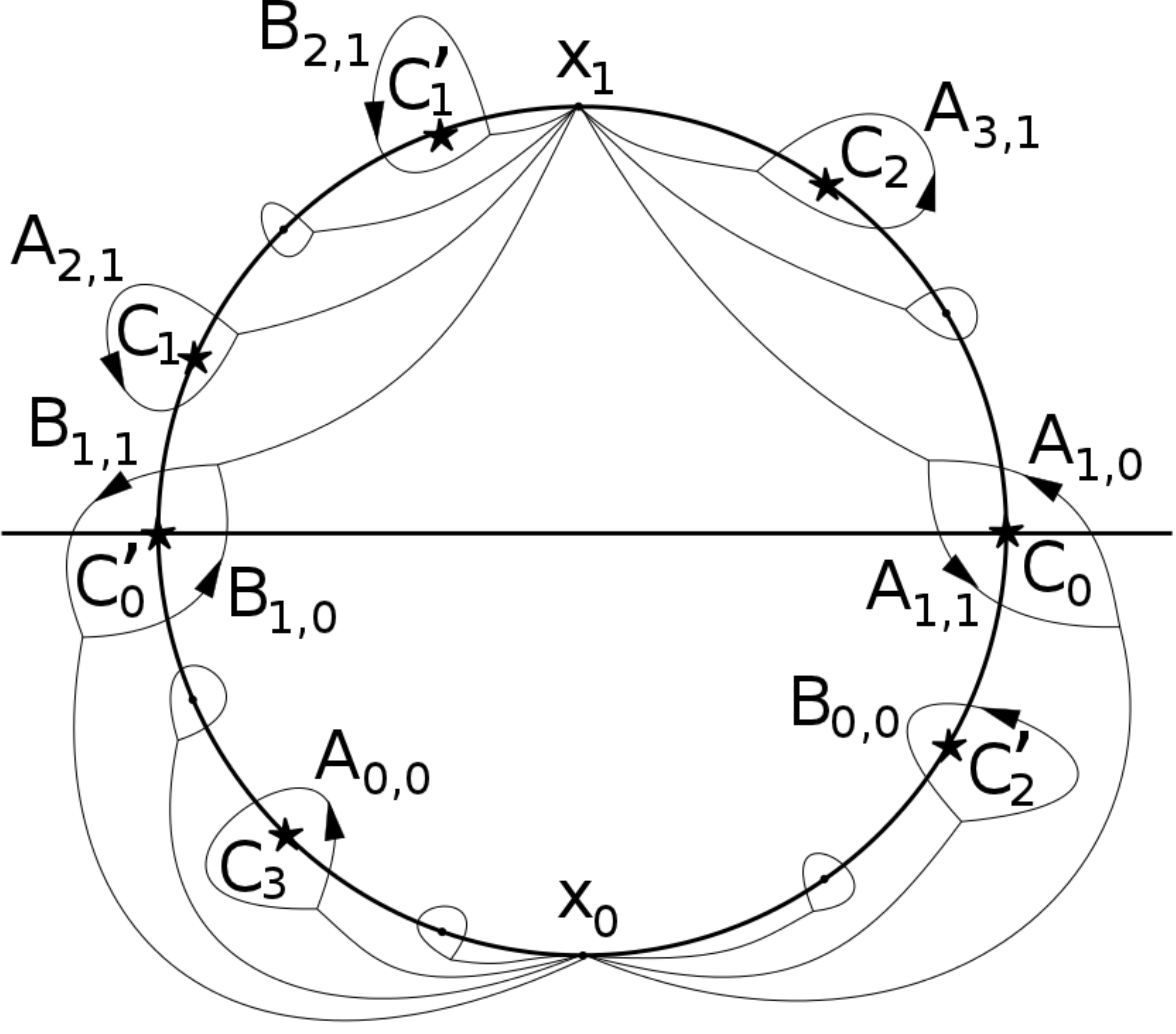}
\caption{The $W$-lift of Figure $\ref{FigWittner1}$}\label{FigWittner2}
\end{center}
\end{figure}

It follows from Lemma \ref{LemComputation} that the iterated monodromy group $\IMG(W,t)$ seen as a subgroup of $\Aut(T_{2})$ is generated by the following wreath recursions

\begin{center}
\begin{tabular}{|c|c|c|c|}
\hline 
&&& \\
$a_{0}=\recursion[a_{3},\Id]$ & $a_{1}=(0,1)\recursion[b_{2}.a_{0},b_{2}^{-1}]$ & $a_{2}=
\recursion[\Id,a_{1}]$ & $a_{3}=\recursion[\Id,a_{2}]$\\ 
&&& \\
\hline
$\xymatrix @!0 @R=2pc @C=5pc { 0 \ar[r]^{\textstyle a_{3}} & 0 \\ 1 \ar[r] & 1 }$ & $\xymatrix @!0 @R=2pc @C=5pc { 0 \ar[rd]^(0.25){\textstyle b_{2}.a_{0}} & 0 \\ 1 \ar[ru]_(0.25){\textstyle b_{2}^{-1}} & 1 }$ & $\xymatrix @!0 @R=2pc @C=5pc { 0 \ar[r] & 0 \\ 1 \ar[r]^{\textstyle a_{1}} & 1 }$ & $\xymatrix @!0 @R=2pc @C=5pc { 0 \ar[r] & 0 \\ 1 \ar[r]^{\textstyle a_{2}} & 1 }$ \\
\hline
\end{tabular}
\begin{tabular}{|c|c|c|}
\hline 
&& \\
$b_{0}=\recursion[b_{2},\Id]$ & $b_{1}=(0,1)\recursion[a_{3}^{-1},b_{0}.a_{3}]$ & $b_{2}=\recursion[\Id,b_{1}]$ \\ 
&& \\
\hline
$\xymatrix @!0 @R=2pc @C=5pc { 0 \ar[r]^{\textstyle b_{2}} & 0 \\ 1 \ar[r] & 1 }$ & $\xymatrix @!0 @R=2pc @C=5pc { 0 \ar[rd]^(0.25){\textstyle a_{3}^{-1}} & 0 \\ 1 \ar[ru]_(0.25){\textstyle b_{0}.a_{3}} & 1 }$ & $\xymatrix @!0 @R=2pc @C=5pc { 0 \ar[r] & 0 \\ 1 \ar[r]^{\textstyle b_{1}} & 1 }$ \\
\hline
\end{tabular}
\end{center}

\newpage

Recall that every homotopy class in $\pi_{1}(\CC\backslash P_{W},t)$ may be uniquely written as $[\gamma_{1}^{\nu_{1}}.\gamma_{2}^{\nu_{2}}.\dots.\gamma_{m}^{\nu_{m}}]$ where the loops $\gamma_{j}$ are chosen among $a_{0},a_{1},a_{2},b_{0},b_{1},b_{2}$ (disregarding $a_{3}$) and the exponents $\nu_{j}$ belong to $\Z$. Notice that for every $k\in\{0,1,2\}$, the following maps
$$\alpha_{k}:[\gamma_{1}^{\nu_{1}}.\gamma_{2}^{\nu_{2}}.\dots.\gamma_{m}^{\nu_{m}}]\longmapsto\sum_{\gamma_{j}=a_{k}}\nu_{j}\quad\text{and}\quad \beta_{k}:[\gamma_{1}^{\nu_{1}}.\gamma_{2}^{\nu_{2}}.\dots.\gamma_{m}^{\nu_{m}}]\longmapsto\sum_{\gamma_{j}=b_{k}}\nu_{j}$$
are well defined on $\pi_{1}(\CC\backslash P_{W},t)$ (the cartesian product of maps $\alpha_{0}\times\alpha_{1}\times\alpha_{2}\times\beta_{0}\times\beta_{1}\times\beta_{2}$ is actually a projection onto the abelianization of $\pi_{1}(\CC\backslash P_{W},t)$, namely $\Z^{6}$).

\begin{description}
\item[Claim:] Let $[\gamma]\in\pi_{1}(\CC\backslash P_{W},t)$ be a homotopy class such that $\alpha_{k}([\gamma])=0$ and $\beta_{k}([\gamma])=1$ for every $k\in\{0,1,2\}$. Then the tree automorphism induced by the monodromy action of $[\gamma]$ acts as a cyclic permutation of order $2^{n}$ on the $n$-th level for every $n\geqslant 1$.
\end{description}
For instance, the wreath recursion $b_{2}.b_{1}.b_{0}=(0,1)\recursion[a_{3}^{-1},b_{1}.b_{0}.a_{3}.b_{2}]$ satisfies the conclusion of Proposition \ref{PropAddingMachineMating} although $W$ is not combinatorially equivalent to a formal mating.
\begin{proof}[of the Claim]
At first, remark that the result holds on the first level for every tree automorphism $\gamma\in\Aut(T_{2})$ coming from a homotopy class $[\gamma]=[\gamma_{1}^{\nu_{1}}.\gamma_{2}^{\nu_{2}}.\dots.\gamma_{m}^{\nu_{m}}]$ which satisfies the assumptions. Indeed, since $\Sym(\mathcal{E})=\{\Id,(0,1)\}$ is an abelian group, the root permutation $\gamma|_{\mathcal{E}^{1}}\in\Sym(\mathcal{E})$ gives
$$\gamma|_{\mathcal{E}^{1}}=\gamma_{m}^{\alpha_{m}}|_{\mathcal{E}^{1}}\circ\dots\circ\gamma_{2}^{\alpha_{2}}|_{\mathcal{E}^{1}}\circ\gamma_{1}^{\alpha_{1}}|_{\mathcal{E}^{1}}=b_{0}|_{\mathcal{E}^{1}}\circ b_{1}|_{\mathcal{E}^{1}}\circ b_{2}|_{\mathcal{E}^{1}}=\Id\circ(0,1)\circ\Id=(0,1)$$

Let $n\geqslant 2$ be an integer and assume by induction that the result holds on the $(n-1)$-th level for every tree automorphism coming from a homotopy class which satisfies the assumptions. Let $\gamma\in\Aut(T_{2})$ be a tree automorphism coming from a homotopy class $[\gamma]$ which satisfies the assumptions. Denote by $\gamma_{0}$ and $\gamma_{1}$ the renormalizations of $\gamma$ at 0 and 1 induced by some homotopy classes $[\gamma_{0}]$ and $[\gamma_{1}]$ (see Lemma \ref{LemComputation}), in order that the wreath recursion of $\gamma$ is given by $\gamma=(0,1)\recursion[\gamma_{0},\gamma_{1}]$. Lemma \ref{LemComputationsWreathRecursion} gives
\begin{center}
\begin{tabular}{ccccccc}
$\gamma^{2}$ & = & \parbox[c]{130pt}{$\xymatrix @!0 @R=2pc @C=5pc { 0 \ar[rd]^(0.25){\textstyle \gamma_{0}} & 0 \ar[rd]^(0.25){\textstyle \gamma_{0}} & 0 \\ 1 \ar[ru]_(0.25){\textstyle \gamma_{1}} & 1 \ar[ru]_(0.25){\textstyle \gamma_{1}} & 1 }$} & = & \parbox[c]{130pt}{$\xymatrix @!0 @R=2pc @C=10pc { 0 \ar[r]^{\textstyle \gamma_{0}.\gamma_{1}} & 0  \\ 1 \ar[r]^{\textstyle \gamma_{1}.\gamma_{0}} & 1 }$} & = & $\recursion[\gamma_{0}.\gamma_{1},\gamma_{1}.\gamma_{0}]$\\ 
\end{tabular}
\end{center}
Furthermore, it follows from the wreath recursions of the generators $a_{0},a_{1},a_{2},b_{0},b_{1},b_{2}$ and from the relation $a_{3}^{-1}=b_{2}.a_{0}.a_{2}.b_{1}.a_{1}.b_{0}$ that
$$\quad\left\{\begin{array}{l} \alpha_{0}([\gamma_{0}])+\alpha_{0}([\gamma_{1}])=-\alpha_{0}([\gamma])+\alpha_{1}([\gamma])+\beta_{1}([\gamma])-\beta_{1}([\gamma])=0 \\ \alpha_{1}([\gamma_{0}])+\alpha_{1}([\gamma_{1}])=-\alpha_{0}([\gamma])+\alpha_{2}([\gamma])+\beta_{1}([\gamma])-\beta_{1}([\gamma])=0 \\ \alpha_{2}([\gamma_{0}])+\alpha_{2}([\gamma_{1}])=-\alpha_{0}([\gamma])+\beta_{1}([\gamma])-\beta_{1}([\gamma])=0 \\ \beta_{0}([\gamma_{0}])+\beta_{0}([\gamma_{1}])=-\alpha_{0}([\gamma])+\beta_{1}([\gamma])-\beta_{1}([\gamma])+\beta_{1}([\gamma])=1 \\ \beta_{1}([\gamma_{0}])+\beta_{1}([\gamma_{1}])=-\alpha_{0}([\gamma])+\beta_{1}([\gamma])-\beta_{1}([\gamma])+\beta_{2}([\gamma])=1 \\ \beta_{2}([\gamma_{0}])+\beta_{2}([\gamma_{1}])=-\alpha_{0}([\gamma])+\alpha_{1}([\gamma])+\beta_{0}([\gamma])+\beta_{1}([\gamma])-\beta_{1}([\gamma])=1 \end{array}\right.$$
Hence
$$\forall k\in\{0,1,2\},\quad\left\{\begin{array}{l} \alpha_{k}([\gamma_{0}.\gamma_{1}])=\alpha_{k}([\gamma_{1}.\gamma_{0}])=\alpha_{k}([\gamma_{0}])+\alpha_{k}([\gamma_{1}])=0 \\ \beta_{k}([\gamma_{0}.\gamma_{1}])=\beta_{k}([\gamma_{1}.\gamma_{0}])=\beta_{k}([\gamma_{0}])+\beta_{k}([\gamma_{1}])=1 \end{array}\right.$$
Therefore the renormalizations $\gamma_{0}.\gamma_{1}$ and $\gamma_{1}.\gamma_{0}$ of $\gamma^{2}$ satisfy the assumptions, and thus (from the inductive hypothesis) they act as cyclic permutations of order $2^{n-1}$ on the $(n-1)$-th level. Consequently $\gamma$ acts as a cyclic permutation of order $2^{n}$ on the $n$-th level and the result follows by induction.
\end{proof}

However the iterated monodromy group of $W$ may be used to prove that $W^{5}$ is combinatorially equivalent to a formal mating. The key is that wreath recursions provide an efficient way to compute the $W^{5}$-lifts of any loop (up to isotopy of loops). For instance, consider the wreath recursion $b_{2}.b_{1}.b_{0}=(0,1)\recursion[a_{3}^{-1},b_{1}.b_{0}.a_{3}.b_{2}]$ induced by the loop $b_{2}.b_{1}.b_{0}$. From Lemma \ref{LemComputation}, this loop has only one $W$-lift which is isotopic to $a_{3}^{-1}.b_{1}.b_{0}.a_{3}.b_{2}$ relatively to $P_{W}$. This $W$-lift induces the following wreath recursion (according to Lemma \ref{LemComputationsWreathRecursion})
\begin{center}
\begin{tabular}{ccl}
$a_{3}^{-1}.b_{1}.b_{0}.a_{3}.b_{2}$ & = & \parbox[c]{350pt}{$\xymatrix @!0 @R=2pc @C=5pc { 0 \ar[r] & 0 \ar[rd]^(0.25){\textstyle a_{3}^{-1}} & 0 \ar[r]^{\textstyle b_{2}} & 0 \ar[r] & 0 \ar[r] & 0 \\ 1 \ar[r]^{\textstyle a_{2}^{-1}} & 1 \ar[ru]_(0.25){\textstyle b_{0}.a_{3}} & 1 \ar[r] & 1 \ar[r]^{\textstyle a_{2}} & 1 \ar[r]^{\textstyle b_{1}} & 1 }$} \\
&& \\
& = & \parbox[c]{350pt}{$\xymatrix @!0 @R=2pc @C=25pc { 0 \ar[rd]^(0.25){\textstyle a_{3}^{-1}.a_{2}.b_{1}} & 0 \\ 1 \ar[ru]_(0.25){\textstyle a_{2}^{-1}.b_{0}.a_{3}.b_{2}} & 1 }$} \\
&& \\
& = & $(0,1)\recursion[a_{3}^{-1}.a_{2}.b_{1},a_{2}^{-1}.b_{0}.a_{3}.b_{2}]$
\end{tabular}
\end{center}
Therefore it follows from Lemma \ref{LemComputation} that the loop $a_{3}^{-1}.b_{1}.b_{0}.a_{3}.b_{2}$ has only one $W$-lift which is isotopic to $a_{3}^{-1}.a_{2}.b_{1}.a_{2}^{-1}.b_{0}.a_{3}.b_{2}$ relatively to $P_{W}$. Equivalently the loop $b_{2}.b_{1}.b_{0}$ has only one $W^{2}$-lift which is isotopic to $a_{3}^{-1}.a_{2}.b_{1}.a_{2}^{-1}.b_{0}.a_{3}.b_{2}$ relatively to $P_{W}$. Repeating this process gives after computations (using the circular relation $b_{2}.a_{0}.a_{2}.b_{1}.a_{1}.b_{0}.a_{3}=\Id$)
\begin{eqnarray*}
a_{3}^{-1}.a_{2}.b_{1}.a_{2}^{-1}.b_{0}.a_{3}.b_{2} & = & (0,1)\recursion[a_{3}^{-1}.a_{1}^{-1}.a_{2}.b_{1},a_{2}^{-1}.a_{1}.b_{0}.a_{3}.b_{2}] \\
a_{3}^{-1}.a_{1}^{-1}.a_{2}.b_{1}.a_{2}^{-1}.a_{1}.b_{0}.a_{3}.b_{2} & = & (0,1)\recursion[b_{2}.a_{1}.b_{0}.a_{3}.b_{2}.a_{0}.a_{2}.b_{1},a_{2}^{-1}.a_{0}^{-1}.b_{2}^{-1}.a_{3}^{-1}.a_{1}^{-1}.b_{2}^{-1}.b_{2}] \\
& = & (0,1)\recursion[b_{2},a_{2}^{-1}.a_{0}^{-1}.b_{2}^{-1}.a_{3}^{-1}.a_{1}^{-1}] \\
b_{2}.a_{2}^{-1}.a_{0}^{-1}.b_{2}^{-1}.a_{3}^{-1}.a_{1}^{-1} & = & (0,1)\recursion[a_{3}^{-1}.b_{2},b_{1}.a_{1}^{-1}.b_{1}^{-1}.a_{2}^{-1}.a_{0}^{-1}.b_{2}^{-1}] \\
& = & (0,1)\recursion[a_{3}^{-1}.b_{2},b_{1}.b_{0}.a_{3}]
\end{eqnarray*}
Finally the loop $b_{2}.b_{1}.b_{0}$ has only one $W^{5}$-lift which is isotopic to $a_{3}^{-1}.b_{2}.b_{1}.b_{0}.a_{3}$ relatively to $P_{W}$ (see Figure \ref{FigLoopWittner}). Remark that these two loops are actually orientation-preserving isotopic to a same Jordan curve relatively to $P_{W}=P_{W^{5}}$. It is known (see for instance \cite{UnmatingRationalMaps}) that the existence of such a Jordan curve, called an equator, is a sufficient condition to prove that $W^{5}$ is combinatorially equivalent to a formal mating.

\begin{figure}[!h]
\begin{center}
\begin{tabular}{|c|c|}
\hline
\parbox[c]{220pt}{\vspace{5pt}\includegraphics[width=220pt]{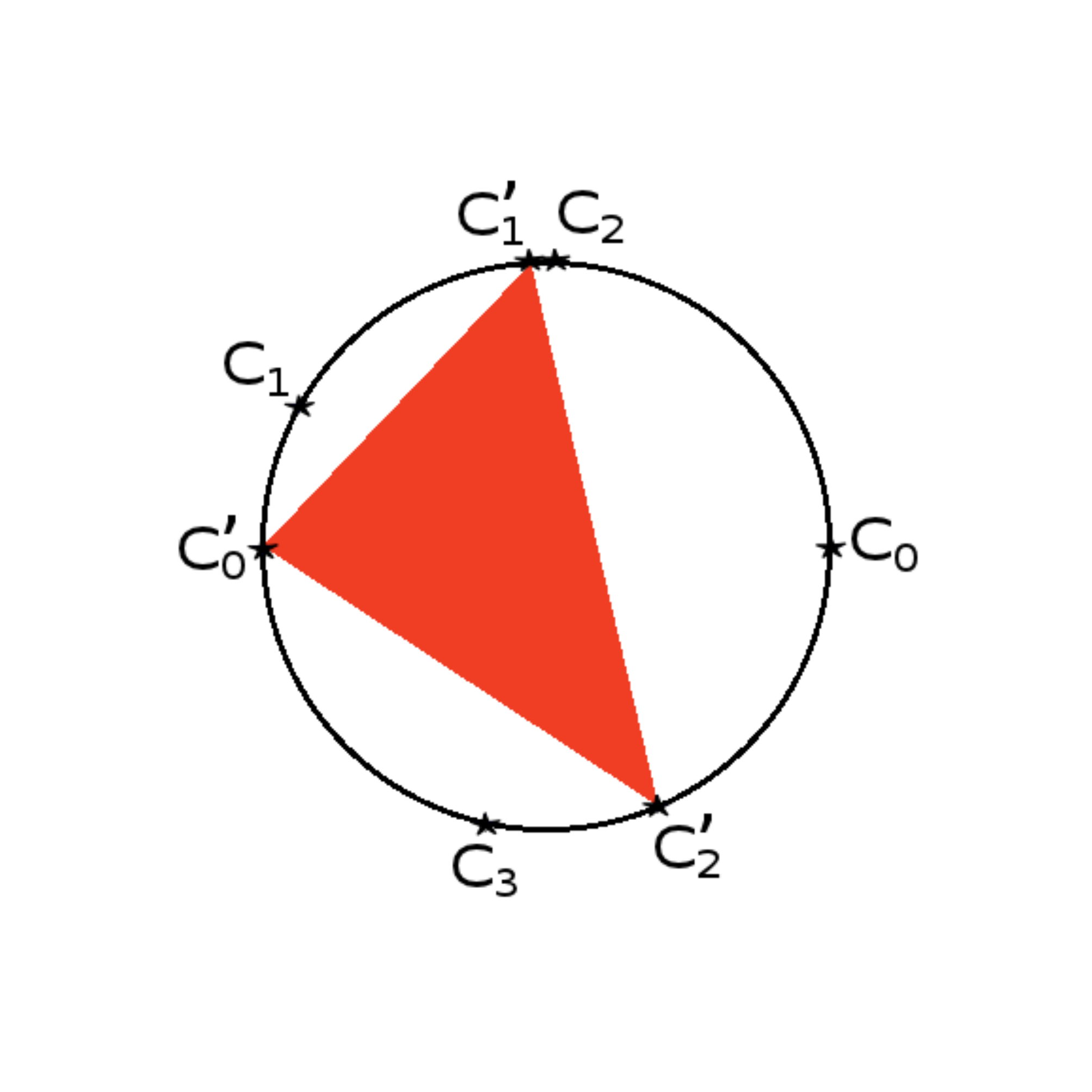}} & \parbox[c]{220pt}{\includegraphics[width=220pt]{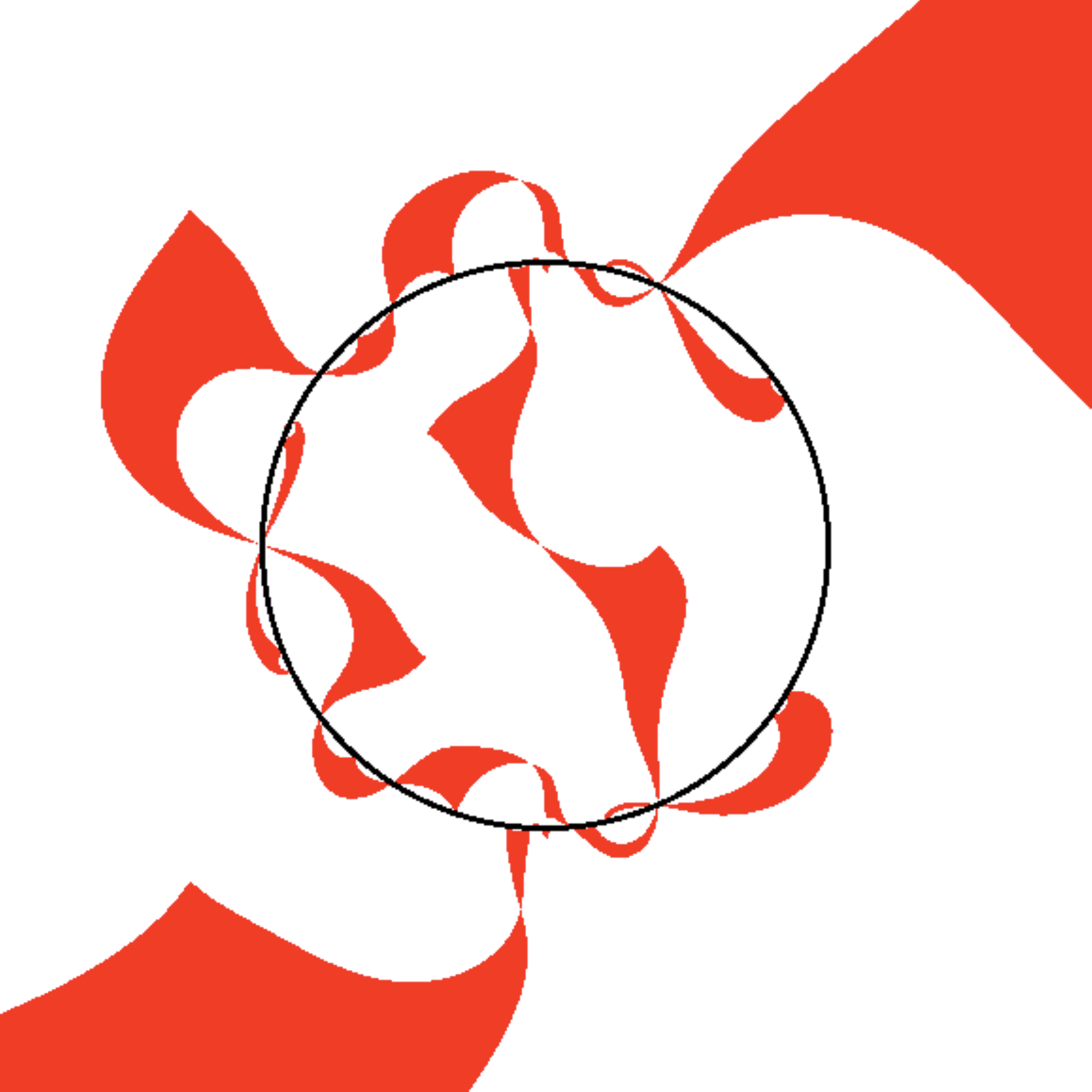}} \\
\hline
\end{tabular}
\caption{The loop $b_{2}.b_{1}.b_{0}$ and its $W^{5}$-lift}\label{FigLoopWittner}
\end{center}
\end{figure}
\end{example}

\newpage

\addcontentsline{toc}{section}{References}
\bibliographystyle{alpha}
\bibliography{biblio}

\vspace{1cm}

\end{document}